\documentclass[11pt,letterpaper]{amsart}
\usepackage{stmaryrd}
\usepackage{amsfonts}
\usepackage{esint}

\usepackage{amsmath}
\usepackage{amssymb}
\usepackage{latexsym}
\usepackage{amscd}
\usepackage{mathrsfs}

\usepackage{hyperref}
\usepackage{graphicx} 
\usepackage{amsthm}
\usepackage{xypic}
\usepackage{bm}
\usepackage[all]{xy}
\usepackage{color}

\newdimen\AAdi%
\newbox\AAbo%
\def\AAk#1#2{\s_etbox\AAbo=\hbox{#2}\AAdi=\wd\AAbo\kern#1\AAdi{}}%
\def\AAr#1#2#3{\s_etbox\AAbo=\hbox{#2}\AAdi=\ht\AAbo\raise#1\AAdi\hbox{#3}}%
\font\tenmsb=msbm10 at 12pt \font\sevenmsb=msbm7 at 8pt
\font\fivemsb=msbm5 at 6pt
\newfam\msbfam
\textfont\msbfam=\tenmsb \scriptfont\msbfam=\sevenmsb
\scriptscriptfont\msbfam=\fivemsb

\newtheorem{theorem}{Theorem}

\newtheorem{remark}[theorem]{Remark}

\newtheorem{corollary}[theorem]{Corollary}

\newtheorem{definition}[theorem]{Definition}

\newtheorem{lemma}[theorem]{Lemma}
\newtheorem{proposition}[theorem]{Proposition}

\numberwithin{equation}{section} \numberwithin{theorem}{section}

\renewcommand{\topmargin}{0cm}
\renewcommand{\oddsidemargin}{5mm}
\renewcommand{\evensidemargin}{5mm}
\renewcommand{\textwidth}{150mm}
\renewcommand{\textheight}{230mm}

\def\C{\mathbb C}
\def\R{\mathbb R}

\def\na{\nabla}
\def\bn{\overline\nabla}

\def\f#1#2{\frac{#1}{#2}}

\def\a{\alpha}
\def\be{\beta}

\def\r{\Re_{I\!V}}

\def\p#1{\partial #1}

\def\de{\delta}
\def\De{\Delta}
\def\e{\eta}
\def\ep{\epsilon}
\def\vep{\varepsilon}
\def\G{\Gamma}
\def\g{\gamma}
\def\k{\kappa}
\def\la{\lambda}
\def\La{\Lambda}
\def\lan{\langle}
\def\ran{\rangle}

\def\Om{\Omega}
\def\th{\theta}
\def\vth{\vartheta}
\def\Th{\Theta}
\def\si{\sigma}
\def\Si{\Sigma}

\def\r{\rho}
\def\z{\zeta}

\begin{document}

\title
[Hessian estimates for Lagrangian mean curvature equation]
{Hessian estimates for Lagrangian mean curvature equation with Lipschitz critical and supercritical phases}
\author{Qi Ding}
\address{Shanghai Center for Mathematical Sciences, Fudan University, Shanghai 200438, China}
\email{dingqi@fudan.edu.cn}

\thanks{
The author would like to express his sincere gratitude to Yu Yuan for his encouragement, interest and valuable discussion, which improved the quality of the manuscript. 
The author is partially supported by NSFC 12371053.}

\begin{abstract}

In this paper, we develop a new strategy to study Lagrangain mean curvature equation on open sets of $\R^{n}(n\ge2)$.
By establishing an Allard-type regularity theorem,
we obtain an interior Hessian estimate of solutions to this equation with prescribed Lipschitz critical and supercritical phases. Here, our condition on the phases is sharp. The proof heavily relies on geometric measure theory, geometry of Lagrangian graphs, and De Giorgi-Nash-Moser iteration.
We expect that the techniques and ideas developed here can be used in some other equations.

\end{abstract}

\maketitle

\section{Introduction}

Let $L$ be the graph in $\C^n\cong\R^n\times\R^n$ of a $C^2$ map $F:\,\Om\to\R^n$ for $n\ge2$, with an open set $\Om\subset\R^n$. $L$ is a Lagrangian submanifold of $\C^n$ if and only if the matrix $DF$ is symmetric with the derivative $D$ on $\R^n$. Let $\th_L$ be the phase of $L$, then $L$ has mean curvature $J\na^L\th_L$. Here, $J$ is the standard complex structure of $\C^n$, and $\na^L$ is the Levi-Civita connection on $L$.

For the simply connected $\Om$, there exists a $C^3$ function $u:\,\Om\to\R$ with $F=Du$.
Let $\la_1,\cdots,\la_n$ be the eigenvalues of Hessian matrix $D^2u$.
In this situation, $u$ satisfies the fully nonlinear equation 
\begin{equation}\aligned\label{LE*}
\mathrm{tr}(\arctan D^2u):=\sum_{i=1}^n\arctan\la_i=\th\qquad\qquad\mathrm{on}\ \Om,
\endaligned
\end{equation}
where $\th\in C^1(\Om,(-n\pi/2,n\pi/2))$ and $\th=\th_L$ (mod $2\pi$) if we regard $\th$ as a function on $L$ by identifying $\th(x,Du(x))=\th(x)$ for every $x\in\Om$. 
A Lagrangian submanifold is called \emph{special} if it is a minimal submanifold at the same time.
For an open set $\Om'\subset\R^n$, the graph of $Dw$ over $\Om'$ is a special (minimal) Lagrangian submanifold of $\C^n$  if and only if $w$ satisfies 
\begin{equation}\aligned\label{SLE}
\mathrm{tr}(\arctan D^2w)=\Th\qquad\qquad\mathrm{on}\ \Om',
\endaligned
\end{equation}
where $\Th\in(-n\pi/2,n\pi/2)$ is a constant. 
The equation \eqref{SLE} is called \emph{special Lagrangian equation}, and we call the equation \eqref{LE*} \emph{Lagrangian mean curvature equation}. 

Hessian estimate takes an important role in studying regularity of solutions to \eqref{SLE} or general \eqref{LE*}.
There is a long story on Hessian estimates for \eqref{SLE} since 1950s. 
Let us review it briefly.
For dimension $n = 2$, Heinz \cite{He} derived a Hessian bound for \eqref{SLE}  with $\Th=\pi/2$ (i.e., the Monge-Amp\'ere equation); Pogorelov \cite{P} got Hessian estimates for \eqref{SLE}  with $\Th>\pi/2$. 
For dimension $n=3$, Hessian bounds for \eqref{SLE} were obtained by Bao-Chen \cite{BC} for $\Th=\pi$, and Warren-Yuan \cite{WmY3,WmY4} for $|\Th|\ge\pi/2$.
$\si_2$ equations on $\R^3$ can be reduced to \eqref{SLE} with $\Th=\pi/2$ (see \cite{GQ,MSY,Q,SY0,SY1,SY2} on $\si_2$ equations in $\R^m(m\ge3)$ for more results).
For every dimension $n$, Chen-Warren-Yuan \cite{CWY} derived \emph{a priori} interior Hessian estimates for convex solutions to \eqref{SLE}. See Warren-Yuan \cite{WmY2} for Hessian estimates with small gradients, and Chen-Shankar-Yuan \cite{CSY} for convex viscosity solutions.

In \cite{Y1}, Yuan observed that the level set of $\{\sum_i\arctan\la_i=\Th\}$ is strictly convex (or convex) if and
 only if $|\Th|>\f{n-2}2\pi$ (or $|\Th|\ge\f{n-2}2\pi$). Naturally, $\pm\f{n-2}2\pi$ is called the \emph{critical} phase.
Moreover, the phase $\Th$ is said to be \emph{supercritical} if $\f{n-2}2\pi<|\Th|<\f n2\pi$, and  \emph{subcritical} if $|\Th|<\f{n-2}2\pi$.
In Nadirashvili-Vl$\mathrm{\breve{a}}$dut \cite{NV} and Wang-Yuan \cite{WdY0}, they constructed counter-examples showing that there are smooth solutions to \eqref{SLE} for each subcritical phase with uniformly bounded gradient but non-uniformly bounded Hessian (See Mooney-Savin \cite{M-S} for non $C^1$ examples).
In every dimension $n\ge3$, Wang-Yuan \cite{WdY} obtained \emph{a priori} interior Hessian estimates for \eqref{SLE} with the critical and supercritical phase, i.e., $|\Th|\ge\f{n-2}2\pi$ (see other proofs by Li \cite{Li} and myself \cite{D1} for $|\Th|>\f{n-2}2\pi$, and by Shankar \cite{Sh} for $|\Th|\ge\f{n-2}2\pi$).

In Theorem 8 of \cite{Wm}, Warren derived interior Hessian estimates for convex solutions to \eqref{LE*} with the $C^{1,1}$ phase $\th$.
For the $C^{1,1}$ critical and supercritical phase $\th$, interior Hessian estimates for \eqref{LE*} have been established by Bhattacharya \cite{Ba1,Ba2}, Bhattacharya-Mooney-Shankar \cite{BMS} and Lu \cite{Lu}. 
Recently, Zhou \cite{Z} derived interior Hessian estimates for solutions to \eqref{LE*} with the Lipschitz phase $\th$ provided $\th$ is supercritical, i.e., $\f{n-2}2\pi+\de\le|\th|<\f n2\pi$ for some $\de>0$, where the Hessian bound depends on $\de$. Moreover, Hessian estimate for convex solutions is also derived in \cite{Z}.

Compared with Wang-Yuan's Hessian estimates in \cite{WdY}, it's natural to expect Hessian estimates for solutions to \eqref{LE*}
only under the condition of the Lipschitz critical and supercritical phase $\th$, i.e., the Lipschitz $\th$ satisfies $\f{n-2}2\pi\le|\th|<\f n2\pi$. 

It's recorded as (2.1) in \cite{Y1}, solutions to \eqref{LE*} for supercritical phases are semi-convex, and (in principle) have better 'compactness', compared with solutions to \eqref{LE*} for critical and supercritical phases. 
Another important advantage of the supercritical phase $\th$ is that it can give a Jacobi inequality involving a divergence term of derivatives of $\th$ (see \cite{Z}), while such a Jacobi inequality does not hold for critical and supercritical phases.

Hessian estimates for the special Lagrangian equation \eqref{SLE} are equivalent to gradient estimates for special (minimal) Lagrangian graphs. In the case of codimension one, Bombieri-De Giorgi-Miranda \cite{BDM} derived interior gradient estimates for the minimal hypersurface equation, where 2-dimensional case had been obtained by Finn \cite{Fi}.

Before \cite{BDM}, De Giorgi had obtained a (purely
existential) gradient estimate for the minimal hypersurface equation in \cite{De}.
In the spirit of De Giorgi, we obtain \emph{a priori} Hessian estimates for  \eqref{LE*} with $\f{n-2}2\pi\le|\th|<\f n2\pi$ by
developing a new strategy based on geometric measure theory and geometry of Lagrangian graphs as well as De Giorgi-Nash-Moser iteration. 
\begin{theorem}\label{main}
Let $u$ be a  $C^4$-solution to \eqref{LE*} on the unit ball $B_1\subset\R^n$ for $n\ge2$ with the phase $\th\ge\f{n-2}2\pi$ on $B_1$, then $|D^2u|$ is bounded on $B_{1/2}$ by a constant depending only on $n$, $\sup_{B_1}|Du|$, and the Lipschitz norm of $\th$.
\end{theorem}

Here, $B_r\subset\R^n$ denotes the ball of radius $r$ centered at the origin.
In fact, Theorem \ref{main} implies $u\in C^{2,\a}$ for each $\a\in(0,1)$ by the Evans-Krylov-Safonov theory.
Heuristically, the key idea of our proof of Theorem \ref{main} consists of two aspects. On the one hand, we obtain regularity and rigidity results under selected constraints using PDE. On the other hand, we successfully combine these results to prove Theorem \ref{main} through several cases by establishing an Allard-type regularity theorem (see Theorem \ref{CaTanPlane}), where the constraints are satisfied in suitable situations. The proof will be completed in $\S$ 7.

The Dirichlet problem for \eqref{LE*} has been studied in \cite{Ba3, BMS, BW, CNS, CP, CPW, DDT, HL1, HL2, Lu, Y1} and so on.
As an application of Theorem \ref{main}, the condition of the critical and supercritical phases can be weakened from $C^{1,1}$ to Lipschitz one when we study solutions to \eqref{LE*} with $C^{1,1}$ Dirichlet boundary data (see Theorem \ref{Dirith} for details).
As a corollary, we immediately get the following result from Theorem \ref{main} (see Corollary \ref{main-Cor}). 
\begin{corollary}\label{intro}
For $n\ge2$, let $u\in C^{1,1}$ satisfy \eqref{LE*} a.e. on $B_1\subset\R^n$ with the Lipschitz phase $\th$ satisfying $|\th|\ge\f{n-2}2\pi$ on $B_1$, then $|D^2u|$ is bounded on $B_{1/2}$ by a constant depending only on $n$, $\sup_{B_1}|Du|$, and the Lipschitz norm of $\th$.
\end{corollary}
In Remark 1.3 in \cite{Ba1}, 
Bhattacharya pointed out that for each $\a\in(0,1)$, $|x|^{1+\a}$ is a convex solution to \eqref{LE*} with a H\"older supercritical phase, but Hessian of $|x|^{1+\a}$ is unbounded. 
Given any positive continuous function $\phi$ on $(0,2]$ with $\lim_{t\to0}\phi(t)=\infty$,
we can find a family of smooth non-convex solutions $u_\ep$ to \eqref{LE*} with critical and supercritical phases $\th_\ep$, where $\lim_{\ep\to0}|D^2u_\ep|(0^n)=\infty$, and $|D\th_\ep|(x)\le\phi(|x|+\ep)$ for each $x\in B_1$ (see Appendix II for details).
Therefore, in view of the counter-examples of Nadirashvili-Vl$\mathrm{\breve{a}}$dut \cite{NV} and Wang-Yuan \cite{WdY0}, the condition on the phase $\th$ in Corollary \ref{intro} is \textbf{sharp} and cannot be weakened further.

The proof of Theorem \ref{main} is divided into 4 main steps:
\begin{itemize}

  \item[i)] Firstly, we use Allard's regularity theorem to prove that a Lagrangian graph with the smallest Jordan angle close to $-\pi/2$ locally can be written as the graph of a graphic function $Dw$ satisfying $\p^2_{x_nx_n}w\ge0$. Then we establish a Harnack's inequality for $\p^2_{x_nx_n}w$ in $\S4$ using a (modified) De Giorgi-Nash-Moser iteration.

  \item[ii)] Secondly, we show \emph{a priori} Hessian estimate with constraints $\la_{m}/2>\la_{m+1}\ge1$ for some $1\le m\le n-2$ in $\S5$. This is achieved by carefully analyzing the Laplacian of $\prod_{i\le m}(1+\la_i^2)^{-1/2}$ on $G_{Du}$, which contains a selected divergence term. Here, $\la_1\ge\cdots\ge\la_n$ are eigenvalues of $D^2u$.

  \item[iii)] Thirdly, let $L$ be a smooth special Lagrangian manifold in a ball, which is a possible limit of Lagrangian graphs with mean curvature converging to zero and uniformly bounded volume. We show that $L$ is a graph unless $L$ is flat using i) and nonnegative superharmonic functions on the graphic part of $L$ in $\S6$.

  \item[iv)] Lastly, in $\S7$ we derive an Allard-type regularity theorem for the Lagrangian graph of $Du$ by contradiction with the help of iii) and geometric measure theory. Combining this theorem and ii), we can prove that if the maximal eigenvalue of $D^2u$ is close to $\infty$, then the minimal eigenvalue must be close to $-\infty$ (in a certain sense). Finally, we can complete the proof of Theorem \ref{main} by combining with i).

\end{itemize}
For convenience in proof, we introduce a notation as follows.
\begin{definition}
For all constants $\La,\k\ge0$, $r>0$ and each integer $n\ge2$, let $\mathbb{F}_n(\La,\k,r)$ be the set containing each pair $(u,\th)$ so that $u\in C^{4}(B_r)$ satisfies \eqref{LE*} on $B_r\subset\R^n$ with the phase $\th$, where $|Du|\le \k r$ on $B_r$, $\th\ge(n-2)\pi/2$, and the Lipschitz norm $\mathbf{Lip}\,\th\le\La/r$.
Put $\mathbb{F}_n(\La,\k)=\mathbb{F}_n(\La,\k,1)$ for short.
\end{definition}

\textbf{Notional convention.} 
For each integer $k>0$, let $\omega_k$ denote the volume of $k$-dimensional unit Euclidean ball, and $\mathcal{H}^k$ denote the $k$-dimensional Hausdorff measure. Let $0^k$ denote the origin of $\R^k$, and $\mathbf{0}$ denote the origin of $\R^{2n}$.
For each $R>0$ and $\mathbf{p}\in\R^{2n}$, let $\mathbf{B}_R(\mathbf{p})\subset\R^{2n}$ denote the ball of radius $R$ centered at $\mathbf{p}\in\R^{2n}$ (except in $\S$ 2.2). We denote $\mathbf{B}_R=\mathbf{B}_R(\mathbf{0})$ for short. 
When we write an integration on a subset of a Riemannian manifold w.r.t. a volume element, we always omit the volume element if it is associated with the standard
metric of the given manifold.

We will use the symbol $\ell$ to represent the subscript in sequences of functions $u_\ell$ or $\th_\ell$, rather than the derivative of $u$ or $\th$ w.r.t. $x_\ell$. We use $u_i,w_i$ to denote the derivatives of $u,w$  w.r.t. $x_i$ for English letter $i$ (and similar ones for $j,k,l$).
For an open $\Om\subset\R^n$ and a function $u\in C^{2}(\Om)$, we always use $G_{Du}$ to denote a (Lagrangian) graph of $Du$ over $\Om$ in $\R^n\times\R^n$.
For a function $f$ on $G_{Du}$, we will identify $f(x,Du(x))=f(x)$, and regard it as a function $f$ on $\Om$ in context, and vice versa.

\section{Preliminary}
\subsection{Geometry of Lagrangian graphs in Euclidean space}
Let $\Om$ be an open set in $\R^n$ ($n\ge2$), and $u$ be a smooth solution to \eqref{LE*} on $\Om$ for a function $\th\in C^\infty(\Om,(-\f n2\pi,\f n2\pi))$. Let $L$ denote the Lagrangian graph in $\R^n\times\R^n$ with the graphic function $Du$ on $\Om$. 
Let $J$ be the standard complex structure of $\C^n$ so that $JE_i=E_{n+i}$ and $JE_{n+i}=-E_{i}$ for each $i=1,\cdots,n$. 
Here, $\{E_i\}_{1\le i\le 2n}$ denotes a standard orthonormal basis of $\R^n\times\R^n$.
Let $dz=dz_1\wedge\cdots\wedge dz_n$ be the volume form of $\C^n$ with $dz_i=dx_i+\sqrt{-1}dy_i$ for each $i$.

Let $\th_L$ be the phase function of $L$, then the $n$-form $\mathrm{Re}(e^{-\sqrt{-1}\th_L}dz)$ has the comass 1 on $L$ (see III.1 in Harvey-Lawson \cite{HL}, or Chapter 5 in Xin \cite{X} for details). Namely, for any $\mathbf{p}\in L$ there holds
\begin{equation}\aligned\label{PhaseThDEF}
\mathrm{Re}(e^{-\sqrt{-1}\th_L}dz)(\mathbf{v})\le|\mathbf{v}|\qquad \mathrm{at}\ \mathbf{p}\ \mathrm{for\ any}\ n\mathrm{-vector}\ \mathbf{v},
\endaligned
\end{equation}
where the equality occurs if and only if $\mathbf{v}$ represents a Lagrangian plane with the phase $\th_L(\mathbf{p})$ mod $2\pi$. Hence, $\th=\th_L$ mod $2\pi$.
For the constant $\th$, $L$ is minimal, and $L$ is said to be a \emph{minimal (or special) Lagrangian graph} (over $\Om$).
This is also equivalent to that (the current associated with) $L$ is
minimizing in $\Om\times\R^n$ (see Theorem 2.3, Proposition 2.17 in \cite{HL}; or \cite{X}).

Let $(\cdots)^{TL}$ and $(\cdots)^{N}$ denote the projection into the tangent bundle $TL$ and the normal bundle $NL$, respectively. 
Let $\bn$ and $\na^L$ denote the Levi-Civita connections of $\R^n\times\R^n$ and $L$, respectively.
The second fundamental form $A_L$ of $L$ is a bi-linear form defined by 
$$A_L(V_1,V_2)=(\bn_{V_1}V_2)^N=\bn_{V_1}V_2-\na^L_{V_1}V_2$$ 
for any $C^1$-vector fields $V_1,V_2$ along $L$. It's clear that $A_L(V_1,V_2)=A_L(V_2,V_1)$, and
$$\lan A_L(V_1,V_2),\nu\ran=-\lan V_2,\bn_{V_1}\nu\ran$$
for any local normal vector field $\nu$ on $L$.
This means that $A_L(V_1,V_2)(\mathbf{p})$ only depends the values of $V_1(\mathbf{p})$ and $V_2(\mathbf{p})$ at any considered point $\mathbf{p}\in L$. 
In other words, $A_L(W_1,W_2)$ is well-defined for any vector $W_1,W_2$ on $T_\mathbf{p}L$.
The mean curvature $H_L$ of $L$ is defined by
$$H_L=\sum_{i=1}^nA_L(e_i,e_i),$$
where $\{e_i\}$ is an orthonormal basis of $T_\mathbf{p}L$.

Taking the derivative of \eqref{LE*} w.r.t. $x_k$, we get
\begin{equation}\aligned\label{uijkThk}
\sum_{i,j=1}^ng^{ij}u_{ijk}=\th_k\qquad \mathrm{for\ each}\ k=1,\cdots,n.
\endaligned
\end{equation}
Here, $\th_k=\p_{x_k}\th$, $u_{ij}=\p_{x_i}\p_{x_j}u$, $u_{ijk}=\p_{x_i}\p_{x_j}\p_{x_k}u$, and $(g^{ij})$ is the inverse matrix of $(g_{ij})$ with $g_{ij}:=\de_{ij}+\sum_l u_{il}u_{jl}$. 
Let $\De_L$ denote the Laplacian on $L$. Let $\mathbf{x}=(x_1,\cdots,x_{2n})$ denote the position vector in $\R^n\times\R^n$. 
Set $v=\sqrt{\det{g_{ij}}}$. Then
\begin{equation}\aligned
\De_L \mathbf{x}=\f1{v}\sum_{i,j=1}^n\p_{x_i}\left(g^{ij}v\p_{x_j}\mathbf{x}\right)=\sum_{i,j=1}^ng^{ij}\p_{x_ix_j}\mathbf{x}+\f1v\sum_{i,j=1}^n\p_{x_i}\left(g^{ij}v\right)\p_{x_j}\mathbf{x}.
\endaligned
\end{equation}
Since $\p_{x_j}\mathbf{x}$ is a tangent vector field on $L$ (see p. 1096 in Osserman \cite{O} for more details), with \eqref{uijkThk} we have
\begin{equation}\aligned\label{HLDeLx}
H_L=\De_L\mathbf{x}=(\De_L\mathbf{x})^N=\left(\sum_{i,j=1}^ng^{ij}\p_{x_ix_j}\mathbf{x}\right)^N=(0,\cdots,0,\th_1,\cdots,\th_n)^N.
\endaligned
\end{equation}
Since $Y^N=-J((JY)^{TL})$ for every vector $Y\in T_{\mathbf{p}}\R^{2n}$ with $\mathbf{p}\in L$,
from \eqref{HLDeLx} we have
\begin{equation}\aligned\label{HL-JnaTh}
H_L=J\left((\th_1,\cdots,\th_n,0,\cdots,0)^{TL}\right)=J\na^L\th.
\endaligned
\end{equation}

Let $\la_1,\cdots,\la_n$ be the eigenvalues of $D^2u$ on $\Om$ with $\la_1\ge\cdots\ge\la_n$. 
There is an orthonormal matrix-valued function $\g=(\g_1,\cdots,\g_n)$ on $\Om$ so that
\begin{equation}\aligned\label{gg1n}
(D^2u) \g_i=\la_i \g_i\quad\mathrm{with}\ \g_i=(\g_{1i},\cdots,\g_{ni})^T
\endaligned
\end{equation}
on $\Om$. Here, $T$ denotes transpose.
Then 
$$\la_i=D^2u(\g_i,\g_i)=\lan\g_i,(D^2u) \g_i\ran=\sum_{j,k=1}^n\g_{ji}u_{jk}\g_{ki}.$$ 
Noting that $\la_i$ may not be $C^1$, and $\g$ may not be $C^0$.
Let $\th_{\g_i}=D_{\g_i}\th=\lan\g_i,D\th\ran$ and 
$$\th_{\g_i\g_j}=D^2\th(\g_i,\g_j)=\lan\g_i,(D^2\th)\g_j\ran=\sum_{k,l=1}^n\g_{ki}\th_{kl}\g_{lj}.$$
We choose two orthonormal bases
\begin{equation}\aligned\label{einui}
e_i=\f1{\sqrt{1+\la_i^2}}\left(\g_i+\la_i J\g_{i}\right),\qquad \nu_i=\f1{\sqrt{1+\la_i^2}}\left(-\la_i \g_i+J\g_{i}\right)
\endaligned
\end{equation}
for each $i=1,\cdots,n$. 
At any considered point $\mathbf{p}=(p,Du(p))$, by using tangent vector fields $\g_i(p)+\sum_{j=1}^n\lan\g_i(p),(D^2u)\g_j(p)\ran J\g_j(p)$ on $L$,
we get
\begin{equation}\aligned
A_L(e_i,e_j)\big|_{\mathbf{p}}=(1+\la_i^2)^{-\f12}(1+\la_j^2)^{-\f12}\sum_{l=1}^nD^3u(\g_i,\g_j,\g_l)J\g_l\big|_p.
\endaligned
\end{equation}
Let $h_{ijk}=\lan A_L(e_i,e_j),\nu_k\ran$ denote the components of $A_L$ for each $i,j,k=1,\cdots,n$, then
\begin{equation}\aligned\label{hijk}
h_{ijk}=(1+\la_i^2)^{-\f12}(1+\la_j^2)^{-\f12}(1+\la_k^2)^{-\f12}D^3u(\g_i,\g_j,\g_k).
\endaligned
\end{equation}

Let $\si_k(\la_1,\cdots,\la_n)=\sum_{1\le i_1<\cdots<i_k\le n}\la_{i_1}\cdots\la_{i_k}$ denote the $k$-th elementary symmetric polynomial of $\la_1,\cdots,\la_n$. We suppose the phase $\th=\sum_i\arctan\la_i\in[\f{n-2}2\pi,\f n2\pi)$.
From Lemma 2.1 in Wang-Yuan \cite{WdY}, 
\begin{equation}\aligned\label{sik***}
\si_k(\la_1,\cdots,\la_n)\ge0\qquad\qquad \mathrm{for\ each}\ k=1,\cdots,n-1.
\endaligned
\end{equation}
Analog to \eqref{sik***} (with dimension $n$ replaced by $n-1$),
\begin{equation}\aligned\label{sikplai***}
\f{\p}{\p\la_i}\si_k(\la_1,\cdots,\la_n)\ge0\qquad\qquad \mathrm{for\ each}\ k\in\{1,\cdots,n-1\},\ i\in\{1,\cdots,n\}.
\endaligned
\end{equation}
We further suppose 
$\la_m>\la_{m+1}$ on $B_\de(p)$ for some ball $B_\de(p)\subset\Om$ and some integer $1\le m\le n-1$, then the function 
\begin{equation}\aligned\label{vm*}
v_m:=\prod_{i=1}^{m}\sqrt{1+\la_i^2}
\endaligned
\end{equation}
is smooth on $B_\de(p)$.
For $n=2,m=1$, from Lemma 2.1 in Warren-Yuan \cite{WmY3} and Lemma 4.1 in Bhattacharya \cite{Ba1} there holds
\begin{equation}\aligned\label{DeMla_1n2}
\De_L\log v_1=&(1+\la_1^2)h_{111}^2+\f{2\la_1(1+\la_1\la_2)}{\la_1-\la_2}h_{122}^2+\f{3\la_1-\la_2+\la_1^2(\la_1+\la_2)}{\la_1-\la_2}h_{112}^2\\
&+\f{\la_1}{1+\la_1^2}\th_{\g_1\g_1}-\f{\la_1}{1+\la_1^2}\th_{\g_1}D_{\g_1}\log v_1.
\endaligned
\end{equation}
In general, from Lemma 4.1 (as well as (4.7) both) in Bhattacharya \cite{Ba1} (see p. 487 in Wang-Yuan \cite{WdY}  for constant phases), we have the following formula.
\begin{lemma}\label{DeMla_1=}
For $n\ge2$, on $B_\de(p)$ there holds
\begin{equation}\aligned
\De_L\log v_m&=\sum_{k=1}^{m}(1+\la_k^2)h_{kkk}^2+\sum_{\stackrel{i,k=1}{i\neq k}}^{m}(3+\la_i^2+2\la_i\la_k)h_{iik}^2\\
&+\sum_{k\le m<l}\f{2\la_k(1+\la_k\la_l)}{\la_k-\la_l}h_{llk}^2+\sum_{l\le m<k}\f{3\la_l-\la_k+\la_l^2(\la_l+\la_k)}{\la_l-\la_k}h_{llk}^2\\
&+2\sum_{i<j<k<m}\left(3+\la_i\la_j+\la_j\la_k+\la_k\la_i\right)h_{ijk}^2\\
&+2\sum_{i<j\le m<k}\left(1+\la_i\la_j+\la_i\la_k+\la_j\la_k+\la_i\f{1+\la_k^2}{\la_i-\la_k}+\la_j\f{1+\la_k^2}{\la_j-\la_k}\right)h_{ijk}^2\\
&+2\sum_{i\le m<j<k}\la_i\left(\la_j+\la_k+\f{1+\la_j^2}{\la_i-\la_j}+\f{1+\la_k^2}{\la_j-\la_k}\right)h_{ijk}^2\\
&+\sum_{i=1}^{m}\f{\la_i}{1+\la_i^2}\th_{\g_i\g_i}-\sum_{i=1}^{m}\f{\la_i}{1+\la_i^2}\th_{\g_i}D_{\g_i}\log v_m.
\endaligned
\end{equation}
\end{lemma}

\subsection{Varifolds and currents in geometric measure theory}

Let us recall Almgren's notion of {\it varifolds} from  geometric measure
theory (see Lin-Yang \cite{LYa} and Simon \cite{S} for more details), which is a
generalization of submanifolds.
For a set $S$ in $\R^{N}$ with $N>n$, we call $S$ \emph{countably $n$-rectifiable} if $S\subset S_0\cup\bigcup_{\ell=1}^\infty F_\ell(\R^n)$, where $\mathcal{H}^n(S_0)=0$, and $F_\ell:\, \R^n\rightarrow\R^{N}$ are Lipschitz mappings for all integers $\ell\ge1$.
Suppose $\mathcal{H}^n(S)<\infty$. Let $\vartheta$ be a positive locally $\mathcal{H}^n$ integrable function on $S$.
Let $|S|$ be the varifold associated with the set $S$.
The associated varifold $V=\vartheta|S|$ is called a \emph{rectifiable $n$-varifold}. Here, $\vartheta$ is called the \emph{multiplicity} function of $V$, and the multiplicity of $|S|$ is equal to one on $S$.
If $\vartheta$ is integer-valued, then $V$ is said to be an \emph{integral varifold}. Associated to $V$, there is a Radon measure $\mu_V$ defined by $\mu_V=\mathcal{H}^n\llcorner\vartheta$, namely,
$$\mu_V(W)=\int_{W\cap S}\vartheta(y)d\mathcal{H}^n(y)\qquad \mathrm{for\ each\ open}\ W\subset\R^{N}.$$
For an open set $U\subset\R^{N}$, $V$ is said to have \emph{the generalized mean curvature} $H_V$  in $U$ if
\begin{equation}\aligned\label{divSYHV}
\int \mathrm{div}_S Yd\mu_V=-\int \lan Y,H_V\ran d\mu_V
\endaligned
\end{equation}
for each $Y\in C^\infty_c(U,\R^{N})$. Here, $\mathrm{div}_S Y$ is the divergence of $Y$ restricted on $S$. For $H_V=0$, $V$ is said to be \emph{stationary}.
For $V=|S|$, we also say that $S$ has the generalized mean curvature $H_V$.

For each $R>0$, let $\mathbf{B}_R(\mathbf{p})\subset\R^{N}$ denote the ball of radius $R$ centered at $\mathbf{p}\in\R^{N}$. Denote $\mathbf{B}_R=\mathbf{B}_R(\mathbf{0})$ for short. 
From (the formula) 18.1 in \cite{S} (or Lemma 1.18 in Colding-Minicozzi \cite{CM1}), for $V=|S|$, $\phi\in C^1(\mathbf{B}_r(\mathbf{p}),\R^+)$ and almost all $r\in(0,R)$
\begin{equation}\aligned\label{MVder}
\f{\p}{\p r}\left(r^{-n}\int_{S\cap \mathbf{B}_r(\mathbf{p})}\phi\right)\ge r^{-n-1}\int_{S\cap\mathbf{B}_r(\mathbf{p})}\lan \mathbf{x}-\mathbf{p},\na^S \phi+H_V\phi\ran,
\endaligned
\end{equation}
whenever $\mathbf{B}_R(\mathbf{p})\subset U$. Here, $\na^S$ denotes the derivative restricted on $S$ a.e..
Let $\mathbf{H}=\sup_S|H_V|$.
From \eqref{MVder} with $\phi\equiv1$, there is an 'almost' monotonicity formula:
\begin{equation}\aligned\label{almostmonoto}
e^{\mathbf{H} r}r^{-n}\mathcal{H}^n(S\cap \mathbf{B}_r(\mathbf{p}))\le e^{\mathbf{H} R}R^{-n}\mathcal{H}^n(S\cap \mathbf{B}_R(\mathbf{p}))
\endaligned
\end{equation}
for each $r\in(0,R)$.

Let us recall Allard's regularity theorem (see Allard \cite{A} or 24.2 in \cite{S}). 
\begin{lemma}\label{AllardregThm}
There are constants $\tau_*,\varrho_*\in(0,1/2)$ depending only on $n,N$ such that
if $V$ is an integral $n$-varifold in $\R^N$ with $\mathbf{0}\in\mathrm{spt} V$ and $\mathrm{spt} V\subset\overline{\mathbf{B}}_r$ for some $r>0$ such that its generalized mean curvature $H_V$ in $\mathbf{B}_r$ satisfies $|H_V|\le\tau_*/r$ a.e., and the Radon measure $\mu_V$ associated to $V$ satisfies
\begin{equation}\aligned\label{DEFepn}
\mu_V(\mathbf{B}_r)\le(1+\tau_*)\omega_nr^n,
\endaligned
\end{equation}
then there is a function $f=(f_1,\cdots,f_{N-n})$ on a ball $\mathbf{B}_{\varrho_*r}\cap P$ for some $n$-subspace $P\subset\R^N$ so that $\{(x,f(x)):\, x\in \mathbf{B}_{\varrho_*r}\cap P\}\subset\mathrm{spt}V$ and
\begin{equation}\aligned\label{DEFepn*}
\sup_{\mathbf{B}_{\varrho_*r}\cap P}|Df|+r^{\a}\sup_{x,y\in \mathbf{B}_{\varrho_*r}\cap P,x\neq y}|x-y|^{-\a}|Df(x)-Df(y)|\le c(\a,n,N)
\endaligned
\end{equation}
for each $\a\in(0,1)$, where $c(\a,n,N)$ is a constant depending only on $\a,n,N$.
\end{lemma}
\begin{remark}
From Allard's regularity theorem, we first choose $\tau_*,\varrho_*$ depending only on $n,N$ so that \eqref{DEFepn*} holds with $\varrho_*$ replaced by $2\varrho_*$ for $\a=\f12$. Then using the standard Schauder estimates of elliptic equations, we can get \eqref{DEFepn*} for each $\a\in(0,1)$.
\end{remark}


Let $\mathcal{D}^n(U)$ denote the set including all smooth $n$-forms on $U$ with compact supports in $U$. Denote $\mathcal{D}_n(U)$ be the set of $n$-currents in $U$, which are continuous linear functionals
on $\mathcal{D}^n(U)$. For each $T\in \mathcal{D}_n(U)$ and each open set $W$ in $U$, one defines the mass of $T$ on $W$ by
\begin{equation*}\aligned
\mathbf{M}(T\llcorner W)=\sup_{|\omega|_U\le1,\omega\in\mathcal{D}^n(U),\mathrm{spt}\omega\subset W}T(\omega)
\endaligned
\end{equation*}
with $|\omega|_U=\sup_{x\in U}\lan\omega(x),\omega(x)\ran^{1/2}$.
Let $\p T$ be the boundary of $T$ defined by $\p T(\omega')=T(d\omega')$ for any $\omega'\in\mathcal{D}^{n-1}(U)$.
$T\in\mathcal{D}_n(U)$ is said to be an \emph{integer multiplicity current} if it can be expressed as
$$T(\omega)=\int_S\vartheta\lan \omega,\vec{\tau}\ran,\qquad \mathrm{for\ each}\ \omega\in \mathcal{D}^n(U),$$
where $S$ is a countably $n$-rectifiable subset of $U$, $\vartheta$ is a locally $\mathcal{H}^n$-integrable positive integer-valued function, and $\vec{\tau}$ is an orientation on $S$, i.e.,  $\vec{\tau}(x)$ is an $n$-vector representing the approximate tangent space $T_xS$ for $\mathcal{H}^n$-a.e. $x$.
We call $[|S|]$ the current associated with $S$, and write $T=\vartheta[|S|]$ with the multiplicity $\vartheta$.
Let $f:\ U\rightarrow \R^N$ be a $C^1$-mapping, and $f_*\vec{\tau}$ denote the push-forward of $\vec{\tau}$, which is an orientation of $f(S)$ in $\R^N$. We define $f_\sharp(T)\in \mathcal{D}_n(\R^N)$ by letting
\begin{equation}\label{PushforwfT}
f_\sharp(T)(\omega)=\int_S\vartheta\lan \omega\circ f,f_*\vec{\tau}\ran
\end{equation}
for each $\omega\in \mathcal{D}^n(\R^N)$.
It's clear that $f_\sharp(T)$ is an integer multiplicity current in $\R^N$.

If both $T$ and $\p T$ are integer multiplicity rectifiable currents, then $T$ is called an \emph{integral current}.
Federer and Fleming \cite{FF} proved a compactness theorem (or referred to as a closure theorem):
a sequence of integral currents $\{T_j\}\subset\mathcal{D}_n(U)$ with $\mathbf{M}(T_j)$ and $\mathbf{M}(\p T_j)$ uniformly
bounded admits a subsequence that converges weakly to an integral current.

The integer multiplicity $n$-current $T$ is said to be \emph{minimizing} in $U$ if
 $\mathbf{M}(T\llcorner W)\le \mathbf{M}(T'\llcorner W)$ whenever $W\subset\subset U$, $\p T = \p T'$ in $U$, spt$(T-T')$ is compact in $W$ (see p. 208 in \cite{LYa} or $\S$ 33 in \cite{S} for instance). 
A point $\mathbf{p}\in \mathrm{spt}T\setminus\mathrm{spt}(\p T)$ is an interior regular point 
if there is a constant $r_{\mathbf{p}}>0$, a smooth embedded submanifold $\G\subset\R^N$ and an integer $k_{\mathbf{p}}>0$ such that $T\llcorner\mathbf{B}_{r_{\mathbf{p}}}(\mathbf{p})=k_{\mathbf{p}}[|\G|]$. The set of interior regular points, which is relatively open in $\mathrm{spt}T\setminus\mathrm{spt}(\p T)$, is called \emph{the regular part}, denoted by $\mathcal{R}(T)$. Its complement $\mathrm{spt}(T)\setminus(\mathrm{spt}(\p T)\cup\mathcal{R}(T))$, the interior
 singular set of $T$, is denoted by $\mathcal{S}(T)$.
From Almgren \cite{Am} or De Lellis-Spadaro \cite{ds1,ds2,ds3}, the (Hausdorff) dimension of $\mathcal{S}(T)\le n-2$ for the minimizing current $T$ in $U$.

\section{Integral estimates for Lagrangian graphs}

In this section, we will establish two integral estimates for Lagrangian graphs mainly based on Wang-Yuan's results in \cite{WdY}.
Let $n\ge2$, and $(u,\th)\in\mathbb{F}_n(\La,\k,r)$ for some $\La\ge0$, $\k\ge1$ and $r>0$.
Let $\de=\inf_{B_{r/2}}\th-\f{n-2}2\pi$. For $\de>0$, let us consider Lewy-Yuan rotation in \cite{Y0} as follows.
Let $\bar{x},\bar{y}:\ B_{r}\to\R^n$ be smooth mapping defined by
\begin{equation}\label{barxbaryx}
\left\{\begin{split}
\bar{x}(x)=&\f1{\sqrt{\si^2+1}}(\si x+Du(x))\\
\bar{y}(x)=&\f1{\sqrt{\si^2+1}}(-x+\si Du(x))
\end{split}\right.\qquad\mathrm{for\ each}\ x\in B_r,
\end{equation}
where $\si:=\cot\f{\de}{n}$. 
(For $\de=0$, we do not need  Lewy-Yuan rotation here, i.e., $\si=\infty$.)
Then
\begin{equation}\aligned\label{LowD2u}
D^2u>-\cot\left(\th-(n-2)\f{\pi}2\right)\ge-\cot\de>-\f1n\cot\f{\de}n\ge-\f{\si}n,
\endaligned
\end{equation}
and the Jacobi of the mapping $\bar{x}$:
\begin{equation}\aligned\label{DEFJ}
J_u=\left(\f{\p\bar{x}_i}{\p x_j}\right)=\f1{\sqrt{\si^2+1}}(\si I+D^2u(x))
\endaligned
\end{equation}
is reversible.
So there is a function $\bar{u}$ on $\bar{x}(B_r)$ such that (see \eqref{bJJb-1App} in Appendix I for instance)
\begin{equation*}
D\bar{u}\big|_{\bar{x}(x)}=\bar{y}(x)=\f1{\sqrt{\si^2+1}}(-x+\si Du(x)).
\end{equation*}

\begin{lemma}\label{VolGDu}
For $n\ge2$, there is a constant $c_n>0$ depending only on $n$ such that 
\begin{equation}\aligned\label{VolGD2u}
\int_{B_{r/2}}\sqrt{\mathrm{det}(I_n+(D^2\bar{u})^2)}\le c_n\left(1+\k\sup_{B_{r}}|D\th|\right)\k^nr^n.
\endaligned
\end{equation}
\end{lemma}
\begin{proof}
We can only consider the case $r=12$ by scaling.
We first assume $\de<\pi/2$.  
From \eqref{D2barubebxx} in Appendix I, it follows that
\begin{equation}\aligned
\arctan(D^2\bar{u})=\arctan D^2u-n\arctan\si^{-1}=\arctan D^2u-\de.
\endaligned
\end{equation}
Let $\bar{\th}=\arctan(D^2\bar{u})$, then $\bar{\th}(\bar{x}(x))=\th(x)-\de$, and $\inf_{\bar{x}(B_6)}\bar{\th}=\f{n-2}2\pi$ by the definition of $\de$. 
Noting $\si=\cot\f{\de}{n}\ge1$. From \eqref{LowD2u} and \eqref{DEFJ}, we get
\begin{equation}\aligned\label{Julowb}
J_u\ge\f1{\sqrt{\si^2+1}}\left(\si-\f{\si}n\right)I\ge\f{n-1}{\sqrt{2}n}I\ge\f13I.
\endaligned
\end{equation}
From $D\bar{\th}\big|_{\bar{x}(x)}=J_u^{-1}(x)D\th(x)$, we get
\begin{equation}\aligned\label{DTh}
\sup_{\bar{x}(B_6)}|D\bar{\th}|\le\sup_{B_6}\left(\overline{\la}(J_u^{-1})|D\th|\right)\le3\sup_{B_6}|D\th|.
\endaligned
\end{equation}
Here, $\overline{\la}(J_u^{-1})$ denotes the maximal eigenvalue of $J_u^{-1}$.
Moreover, \eqref{Julowb} implies $B_2\subset\bar{x}(B_{6}).$

Let $\e$ be a Lipschitz function satisfying $\e=1$ on $B_1$, $\e=0$ outside $B_{1+1/n}$ and $|D\e|\le n$.
Let $\mu_1,\cdots,\mu_n$ be eigenvalues of $D^2\bar{u}$, and $\bar{\si}_k$ denote the $k$-th elementary symmetric polynomial of $D^2\bar{u}$ so that $\bar{\si}_k=\bar{\si}_k(\mu_1,\cdots,\mu_n)$ with $\bar{\si}_0=1$. 
Recalling a divergence formula (see p. 495 in \cite{WdY}):
\begin{equation}\aligned\label{divStruc}
\int_\Om\sum_{i=1}^n\f{\p\bar{\si}_{k}}{\p \bar{u}_{ij}}\f{\p\varphi}{\p x_i}=0\qquad \mathrm{for\ each}\ j=1,\cdots,n,
\endaligned
\end{equation}
where $\varphi$ is a Lipschitz function on $\Om$ with compact support in $\Om$.
Combining with \eqref{sik***} and \eqref{divStruc}, we have
\begin{equation}\aligned\label{n-1bsisinbth}
&(n-1)\int_{B_1}\bar{\si}_{n-1}\sin\bar{\th}\le(n-1)\int_{B_2}\e\bar{\si}_{n-1}\sin\bar{\th}
=\int_{B_2}\e\sum_{i,j=1}^n\f{\p\bar{\si}_{n-1}}{\p \bar{u}_{ij}}\f{\p^2\bar{u}}{\p x_i\p x_j}\sin\bar{\th}\\
=&-\int_{B_2}\sin\bar{\th}\sum_{i,j=1}^n\f{\p\e}{\p x_i}\f{\p\bar{\si}_{n-1}}{\p \bar{u}_{ij}}\f{\p \bar{u}}{\p x_j}-\int_{B_2}\e\cos\bar{\th}\sum_{i,j=1}^n\f{\p\bar{\th}}{\p x_i}\f{\p\bar{\si}_{n-1}}{\p \bar{u}_{ij}}\f{\p \bar{u}}{\p x_j}.
\endaligned
\end{equation}
From
\begin{equation}\aligned
\sup_{\bar{x}(B_6)}\left|\sin\bar{\th}\right|\le\mathrm{osc}_{_{B_6}}\th\le6\sup_{B_6}|D\th|\qquad\mathrm{for\ even}\ n,
\endaligned
\end{equation}
 \eqref{sikplai***} and \eqref{DTh}, \eqref{n-1bsisinbth} gives
\begin{equation}\aligned\label{sinhatTh}
(n-1)\left|\int_{B_1}\bar{\si}_{n-1}\sin\bar{\th}\right|\le c_n\k\sup_{B_6}|D\th|\int_{B_{1+1/n}}\bar{\si}_{n-2}.
\endaligned
\end{equation}
Here, $c_n>0$ denotes a general constant depending only on $n$, which may change from line to line.
Similarly, for odd $n$
 \begin{equation}\aligned\label{coshatTh}
(n-1)\left|\int_{B_1}\bar{\si}_{n-1}\cos\bar{\th}\right|\le c_n\k\sup_{B_6}|D\th|\int_{B_{1+1/n}}\bar{\si}_{n-2}.
\endaligned
\end{equation}

On the other hand, using \eqref{divStruc} we clearly have
\begin{equation}\aligned\label{sin-1n-2}
(n-1)\left|\int_{B_1}\bar{\si}_{n-1}\right|\le c_n\k\int_{B_{1+1/n}}\bar{\si}_{n-2}.
\endaligned
\end{equation}
By induction, we get
\begin{equation}\aligned\label{sin-2induction}
(n-2)\int_{B_{1+1/n}}\bar{\si}_{n-2}\le c_n(n-3)\k\int_{B_{1+2/n}}\bar{\si}_{n-3}\le\cdots\le c_n^{n-1}\k^{n-2}.
\endaligned
\end{equation}
Let 
$$\bar{v}=\sqrt{\det\left(I+(D^2\bar{u})^2\right)}=\prod_{i=1}^n\sqrt{1+\mu_i^2}.$$ 
From (3.2) in \cite{WdY}, we have
\begin{equation}\aligned\label{vlai}
\sum_{i=1}^n\f{\bar{v}}{1+\mu_i^2}=\cos\bar{\th}\sum_{1\le2k+1\le n}(-1)^k(n-2k)\bar{\si}_{2k}-\sin\bar{\th}\sum_{1\le2k\le n}(-1)^k(n-2k+1)\bar{\si}_{2k-1}.
\endaligned
\end{equation}
Substituting \eqref{sinhatTh}\eqref{coshatTh}\eqref{sin-2induction} into \eqref{vlai}, it follows that
\begin{equation}\aligned\label{Estv}
\int_{B_1}\sum_{i=1}^n\f{\bar{v}}{1+\mu_i^2}\le c_n\left(1+\k\sup_{B_6}|D\th|\right)\k^{n-2}.
\endaligned
\end{equation}
Moreover, with \eqref{sin-1n-2}\eqref{sin-2induction} we have
\begin{equation}\aligned\label{Estv*}
\int_{B_1}\sum_{i=1}^n\f{\bar{v}}{1+\mu_i^2}\le c_n\k^{n-1}.
\endaligned
\end{equation}

Let $\Si$ denote the graph of $D\bar{u}$ over $\bar{x}(B_{6})\supset B_2$ with metric $(\bar{g}_{ij})=I_n+(D^2\bar{u})^2$. Let $(\bar{g}^{ij})$ be the inverse of $(\bar{g}_{ij})$.
Let $\xi$ be a Lipschitz function satisfying $\xi=1$ on $B_{1/2}$, $\xi=0$ outside $B_{1}$ and $|D\xi|\le 2$.
We also regard $\xi$ as a function on $\Si$ by identifying $\xi(x)=\xi(x,D\bar{u}(x))$.
Integrating by parts in conjuction with Cauchy-Schwartz inequality gives
\begin{equation}\aligned
-\int_\Si\bar{u}_i\xi^2\De_\Si\bar{u}_i=&\int_\Si|\na^\Si\bar{u}_i|^2\xi^2+2\int_\Si\bar{u}_i\xi\lan\na^\Si\bar{u}_i,\na^\Si\xi\ran \\
\ge&\f12\int_\Si|\na^\Si\bar{u}_i|^2\xi^2-2\int_\Si\bar{u}_i^2|\na^\Si\xi|^2.
\endaligned
\end{equation}
From \eqref{HL-JnaTh},
\begin{equation}\aligned\label{DeSibui}
\De_\Si\bar{u}_i=\lan H_\Si,E_{n+i}\ran=\lan J\na^\Si\bar{\th},E_{n+i}\ran=\lan \na^\Si\bar{\th},E_i\ran=\sum_{j=1}^n\bar{g}^{ij}\f{\p\bar{\th}}{\p x_j}
\endaligned
\end{equation}
for each $i$, we get
\begin{equation}\aligned\label{bnbui2xi2*}
\int_\Si\sum_{i=1}^n|\na^\Si\bar{u}_i|^2\xi^2\le&4\int_\Si\sum_{i=1}^n\bar{u}_i^2|\na^\Si\xi|^2 -2\int_\Si\xi^2\sum_{i,j=1}^n\bar{g}^{ij}\f{\p\bar{\th}}{\p x_j}\bar{u}_i\\
\le& c_n\k^2\int_{B_1}\sum_{i=1}^n\f{\bar{v}}{1+\mu_i^2}+c_n\k\sup_{B_6}|D\th|\int_{B_1}\sum_{i=1}^n\f{\bar{v}}{1+\mu_i^2}.
\endaligned
\end{equation}
Combining with \eqref{Estv}\eqref{Estv*}, we get
\begin{equation}\aligned
\int_{B_{1/2}}\sum_{i=1}^n\f{\mu_i^2}{1+\mu_i^2}\bar{v}\le c_n\left(1+\k\sup_{B_6}|D\th|\right)\k^n,
\endaligned
\end{equation}
and then
\begin{equation}\aligned\label{barvupb}
n\int_{B_{1/2}}\bar{v}\le\int_{B_{1/2}}\sum_{i=1}^n\f{\mu_i^2}{1+\mu_i^2}\bar{v}+\int_{B_1}\sum_{i=1}^n\f{\bar{v}}{1+\mu_i^2}\le c_n\left(1+\k\sup_{B_6}|D\th|\right)\k^n.
\endaligned
\end{equation}
If $\de\ge\pi/2$, then $u$ is convex, and $-\tan\f{\pi}{2n}\le D^2\bar{u}\le\cot\f{\pi}{2n}$ from \eqref{D2barubetan} in Appendix I, which also gives \eqref{barvupb}. 
By considering a suitable finite covering $\{B_{1/2}(p_i)\}$ of $B_6$, we complete the proof.
\end{proof}
\begin{remark}
The order $n$ of $\k^n$ in \eqref{VolGD2u} is sharp by the example $u(x)=\f{\k}2\sum_{i=1}^nx_i^2$ with $\mathrm{tr}(\arctan D^2u)=n\arctan\k$.
\end{remark}

From Lemma \ref{VolGDu}, we immediately have
\begin{equation}\aligned\label{VolGDur}
\mathcal{H}^n(G_{Du}\cap\mathbf{B}_{r/2})=\mathcal{H}^n(\Si\cap\mathbf{B}_{r/2})\le c_n\left(1+\k\sup_{B_{r}}|D\th|\right)\k^nr^n.
\endaligned
\end{equation}
Let $\phi$ be a nonnegative function in the Sobolev space $W^{1,2}(\overline{B}_2)$.
We have the following integral estimate.
\begin{lemma}\label{Intfv}
Let $(u,\th)\in\mathbb{F}_n(\La,\k,r)$ for some $\La,\k\ge0$ and $v=\sqrt{\det\left(I_n+(D^2u)^2\right)}$.
There is a constant $c_{n,\k,\La}>0$ depending only on $n,\k$ and $\La$ such that 
\begin{equation}\aligned
\int_{B_{r/2}}\phi v\le c_{n,\k,\La}\left(\f1\ep\int_{B_{r}}v+\int_{B_{r}}\phi+\f{\ep}{r^2}\int_{B_{r}}|\na^\Si \phi|^2v\right)\qquad\mathrm{for\ each}\ \ep>0.
\endaligned
\end{equation}
\end{lemma}
\begin{proof}
By scaling, we can only consider the case $r=4$.
Let $\e,\xi$ be the functions defined in Lemma \ref{VolGDu}.
Let $\si_k$ denote the $k$-th elementary symmetric polynomial of $D^2u$ with $\si_0=1$ so that $\si_k=\si_k(\la_1,\cdots,\la_n)$. Here, $\la_1,\cdots,\la_n$ are eigenvalues of $D^2u$.
For each $k=1,\cdots,n-1$ and $\ep>0$, together with (3.5) in  \cite{WdY} (the function $b$ there instead by $\ep \phi$ here) and \eqref{sik***}\eqref{divStruc}, we can get
\begin{equation}\aligned
&\ep k\int \phi\e\si_k=\ep\int \phi\e\sum_{i,j=1}^n\f{\p\si_{k}}{\p u_{ij}}\f{\p^2u}{\p x_i\p x_j}\\
=&-\ep\int \e\sum_{i,j=1}^n\f{\p \phi}{\p x_i}\f{\p\si_{k}}{\p u_{ij}}\f{\p u}{\p x_j}-\ep\int \phi\sum_{i,j=1}^n\f{\p\e}{\p x_i}\f{\p\si_{k}}{\p u_{ij}}\f{\p u}{\p x_j}\\
\le& c_n\k\left(\ep\int_{B_{1+1/n}} \phi\si_{k-1}+\int_{B_{1+1/n}}\left(\ep^2|\na^\Si \phi|^2+n\right)v\right).
\endaligned
\end{equation}
By induction (on $k$) and \eqref{vlai}, we have (see also the Step 4 in \cite{WdY})
\begin{equation}\aligned\label{fvlai11}
\int_{B_1}\sum_{i=1}^n\f{\phi v}{1+\la_i^2}\le c_n\k^{n-1}\int_{B_2}\phi+c_n\k\int_{B_{2}}\left(\ep|\na^\Si \phi|^2+\f n\ep\right)v.
\endaligned
\end{equation}
Integrating by parts in conjuction with Cauchy-Schwartz inequality implies
\begin{equation}\aligned
-&\int_\Si u_i\xi^2\phi\De_\Si u_i=\int_\Si|\na^\Si u_i|^2\xi^2\phi+2\int_\Si u_i\xi\lan\na^\Si u_i,\na^\Si\xi\ran \phi+\int_\Si u_i\xi^2\lan\na^\Si u_i,\na^\Si \phi\ran\\
\ge&\f12\int_\Si|\na^\Si u_i|^2\xi^2\phi-2\int_\Si u_i^2|\na^\Si\xi|^2\phi-\f1{2\ep}\int_\Si|\na^\Si u_i|^2\xi^2-\f{\ep}2\int_\Si u_i^2|\na^\Si \phi|^2\xi^2.
\endaligned
\end{equation}
Analog to the argument in \eqref{bnbui2xi2*}, with \eqref{DeSibui} we get
\begin{equation}\aligned\label{fvlai22}
\int_\Si\sum_{i=1}^n|\na^\Si u_i|^2\xi^2\phi\le
&c_n\k^2\int_{B_1}\sum_{i=1}^n\f{v\phi}{1+\la_i^2}+\f{n}{\ep}\int_{B_1}v+c_n\k^2\ep\int_{B_1}|\na^\Si \phi|^2v\\
&+c_n\k\sup_{B_1}|D\th|\int_{B_1}\sum_{i=1}^n\f{v\phi}{1+\la_i^2}.
\endaligned
\end{equation}
Analog to \eqref{barvupb}, combining \eqref{fvlai11}\eqref{fvlai22} we can obtain
\begin{equation}\aligned
\int_{B_{1/2}}\phi v\le c_{n,\k,\La}\left(\f1\ep\int_{B_2}v+\int_{B_2}\phi+\ep\int_{B_2}|\na^\Si \phi|^2v\right).
\endaligned
\end{equation}
We complete the proof by a suitable covering of $B_2$ with finite balls $\{B_{1/2}(p_i)\}$.
\end{proof}

\section{Allard's regularity theorem and Harnack's inequality}

Recalling that $\tau_{*},\varrho_*$ are the constants in Lemma \ref{AllardregThm} with $N=2n$ there.
Let $(u,\th)\in\mathbb{F}_n(\tau_{*},\k)$ for some $\k\ge0$, $n\ge2$, and $L=G_{Du}$.
We further assume 
\begin{equation}\aligned\label{Allard-req}
|Du(0^n)|=0\qquad\mathrm{and}\qquad\mathcal{H}^n(L\cap\mathbf{B}_1)\le(1+\tau_{*})\omega_n.
\endaligned
\end{equation}
For two $n$-planes $P_1,P_2$, they are two points in Grassmannian manifolds $\mathbf{G}_{n,n}$ (see $\S$ 2 in \cite{JX} or $\S$ 7.1 in \cite{X} for more details). We use $d_{\mathbf{G}}(P_1,P_2)$ to denote the distance of $P_1,P_2$ in $\mathbf{G}_{n,n}$.
From Lemma \ref{AllardregThm}, there is a constant $c^*_{n}$ depending only on $n$ so that
\begin{equation}\aligned\label{Equi-cont}
\sup_{\mathbf{x},\mathbf{y}\in L\cap\mathbf{B}_{\varrho_*},\mathbf{x}\neq\mathbf{y}}|\mathbf{x}-\mathbf{y}|^{-\f12}d_{\mathbf{G}}(T_{\mathbf{x}}L,T_{\mathbf{y}}L)\le c^*_{n}.
\endaligned
\end{equation}
Let $\mathscr{A}$ be a matrix-valued function on $L$ defined by $\mathscr{A}(\mathbf{x})=\arctan D^2u(x)$ for any $\mathbf{x}=(x,Du(x))$ with $x\in B_1$. 
For each $\mathbf{p}\in L$, the tangent space of $L$ at $\mathbf{p}$ is
\begin{equation}\aligned
T_{\mathbf{p}}L=\{(x,y)\in\R^{n}\times\R^n|\ (\cos\mathscr{A}(\mathbf{p}))y=(\sin\mathscr{A}(\mathbf{p}))x,\ x\in\R^n\}.
\endaligned
\end{equation}
Clearly, \eqref{Equi-cont} implies that $\mathscr{A}$ is equicontinuous on $L\cap\mathbf{B}_{\varrho_*}$. Namely, for any $\ep>0$, there is a constant $\de(n,\ep)>0$ depending only on $n,\ep$ so that for any $\mathbf{x},\mathbf{x'}\in L\cap\mathbf{B}_{\varrho_*}$ with $|\mathbf{x}-\mathbf{x'}|<\de(n,\ep)$ there holds
\begin{equation}\aligned\label{AAxx'ep}
\left|\mathscr{A}(\mathbf{x})-\mathscr{A}(\mathbf{x'})\right|<\ep.
\endaligned
\end{equation}

Let
$\Th_{1},\cdots,\Th_{n}$ denote eigenvalues of $\mathscr{A}$ with $\f{\pi}2>\Th_{1}\ge\cdots\ge\Th_{n}>-\f{\pi}2$. There exists a real orthonormal matrix $Q$ so that $\mathscr{A}=Q\Xi Q^T$, where $\Xi=\mathrm{diag}\{\Th_{1},\cdots,\Th_{n}\}$.
Let $\mathbf{E}_i(\mathbf{x})$ denote the eigenspace of $\Th_i(\mathbf{x})$ w.r.t. $\mathscr{A}(\mathbf{x})$. In particular, dim$\mathbf{E}_i(\mathbf{x})>1$ if $\Th_i(\mathbf{x})$ has multiplicity $>1$. For two linear subspaces $P,P'$ in $\R^n$, we define
$$d_{\mathbb{S}}\left(P,P'\right):=\inf_{\mathbf{v}\in P\cap\p B_1,\mathbf{v}'\in P'\cap\p B_1}|\mathbf{v}-\mathbf{v}'|.$$

\begin{proposition}\label{ThiQxx'}
For any $\ep>0$, there is a constant $\de\in(0,\varrho_*)$ depending only on $\ep,n$ so that for every $\mathbf{x},\mathbf{x}'\in L\cap\mathbf{B}_{\varrho_*}$ with $|\mathbf{x}-\mathbf{x}'|<\de$ there holds
\begin{equation}\aligned
\sup_{1\le i\le n}\left(\left|\Th_i(\mathbf{x})-\Th_i(\mathbf{x}')\right|+d_{\mathbb{S}}\left(\mathbf{E}_i(\mathbf{x}),\mathbf{E}_i(\mathbf{x}')\right)\right)<\ep.
\endaligned
\end{equation}
\end{proposition}
\begin{proof}
Let us prove it by contradiction. We suppose there are sequences of $\{(u_\ell,\th_\ell)\}_{\ell\ge1}\subset\mathbb{F}_n(\tau_*,\k_i)$ with $\mathcal{H}^n(G_{Du_\ell}\cap\mathbf{B}_1)\le(1+\tau_{*})\omega_n$, $\mathbf{x}_\ell,\mathbf{x}_\ell'\in G_{Du_\ell}\cap\mathbf{B}_{\varrho_*}$, $\k_i>0$ so that $\lim_{\ell\to\infty}|\mathbf{x}_\ell-\mathbf{x}_\ell'|=0$ and  
\begin{equation}\aligned\label{Xixi'xx'}
\lim_{\ell\to\infty}\left|\Xi_\ell(\mathbf{x}_\ell)-\Xi'_\ell(\mathbf{x}_\ell')\right|+\lim_{\ell\to\infty}\sup_i d_{\mathbb{S}}\left(\mathbf{E}_{i,\ell}(\mathbf{x}_\ell),\mathbf{E}_{i,\ell}(\mathbf{x}_\ell')\right)>0,
\endaligned
\end{equation}
where $\Xi_\ell=\mathrm{diag}\{\Th_{1,\ell},\cdots,\Th_{n,\ell}\}$ with $\Th_{1,\ell},\cdots,\Th_{n,\ell}$ denoting eigenvalues of $\mathscr{A}_\ell:=\arctan D^2u_\ell$ with $\f{\pi}2>\Th_{1,\ell}\ge\cdots\ge\Th_{n,\ell}>-\f{\pi}2$, and $\mathbf{E}_{i,\ell}$ denotes the eigenspace of $\Th_{i,\ell}$ w.r.t. $\mathscr{A}_\ell$.
Let $Q_\ell,Q_\ell'$ be real orthonormal matrice so that 
$\mathscr{A}_\ell(\mathbf{x}_\ell)=Q_\ell\Xi_\ell(\mathbf{x}_\ell) Q_\ell^T$ and $\mathscr{A}_\ell(\mathbf{x}_\ell')=Q'_\ell\Xi_\ell(\mathbf{x}_\ell') (Q'_\ell)^T$.
Up to choosing subsequences, we can assume that $Q_\ell,Q_\ell',\Xi_\ell(\mathbf{x}_\ell),\Xi'_\ell(\mathbf{x}_\ell')$ converges to $Q_\infty,Q'_\infty,\Xi_\infty,\Xi'_\infty$, respectively. Hence,
\begin{equation}\aligned\label{AixiA'}
\left|Q_\infty\Xi_\infty Q_\infty^T-Q_\infty'\Xi'_\infty (Q_\infty')^T\right|=&\lim_{\ell\to\infty}\left|Q_\ell\Xi_\ell(\mathbf{x}_\ell) Q_\ell^T-Q_\ell'\Xi'_\ell(\mathbf{x}_\ell') (Q_\ell')^T\right|\\
=&\lim_{\ell\to\infty}\left|\mathscr{A}_\ell(\mathbf{x}_\ell)-\mathscr{A}_\ell(\mathbf{x}_\ell')\right|.
\endaligned
\end{equation}
Let $\Th_{1,\infty}\ge\cdots\ge\Th_{n,\infty}$ be the eigenvalues of $\Xi_\infty$, and $\Th'_{1,\infty}\ge\cdots\ge\Th'_{n,\infty}$ be the eigenvalues of $\Xi'_\infty$.
Let $\mathbf{E}_{i,\infty}$ and $\mathbf{E}'_{i,\infty}$ denote the eigenspace of $\Th_{i,\infty}$ and $\Th'_{i,\infty}$ w.r.t. $Q_\infty\Xi_\infty Q_\infty^T$ and $Q_\infty'\Xi'_\infty (Q_\infty')^T$, respectively. In this situation, we have
\begin{equation}\aligned\label{EiEi'infty}
\sup_i d_{\mathbb{S}}\left(\mathbf{E}_{i,\infty},\mathbf{E}'_{i,\infty}\right)=\lim_{\ell\to\infty}\sup_i d_{\mathbb{S}}\left(\mathbf{E}_{i,\ell}(\mathbf{x}_\ell),\mathbf{E}_{i,\ell}(\mathbf{x}_\ell')\right)=0.
\endaligned
\end{equation}
From \eqref{AAxx'ep}, it follows that
\begin{equation}\aligned
\lim_{\ell\to\infty}\left|\mathscr{A}_\ell(\mathbf{x}_\ell')-\mathscr{A}_\ell(\mathbf{x}_\ell')\right|=0.
\endaligned
\end{equation}
Hence, \eqref{AixiA'} gives
\begin{equation}\aligned
Q_\infty\Xi_\infty Q_\infty^T=Q_\infty'\Xi'_\infty (Q_\infty')^T,
\endaligned
\end{equation}
i.e.,
\begin{equation}\aligned
\Xi_\infty=Q_\infty^TQ_\infty'\Xi'_\infty (Q_\infty^TQ_\infty')^T.
\endaligned
\end{equation}
This means that $\Xi_\infty$ and $\Xi'_\infty$ have the same eigenvalues, i.e., $\Xi_\infty=\Xi'_\infty$. 
From \eqref{Xixi'xx'}\eqref{EiEi'infty}, it follows that
\begin{equation}\aligned
&\sup_i d_{\mathbb{S}}\left(\mathbf{E}_{i,\infty},\mathbf{E}'_{i,\infty}\right)=\left|\Xi_\infty-\Xi'_\infty\right|+\sup_i d_{\mathbb{S}}\left(\mathbf{E}_{i,\infty},\mathbf{E}'_{i,\infty}\right)\\
=&\lim_{\ell\to\infty}\left|\Xi_\ell(\mathbf{x}_\ell)-\Xi'_\ell(\mathbf{x}_\ell)\right|+\lim_{\ell\to\infty}\sup_i d_{\mathbb{S}}\left(\mathbf{E}_{i,\ell}(\mathbf{x}_\ell),\mathbf{E}_{i,\ell}(\mathbf{x}_\ell')\right)>0.
\endaligned
\end{equation}
This contradicts to \eqref{EiEi'infty}. We complete the proof.
\end{proof}

Let $\be^*=(\be^*_1,\cdots,\be^*_n)=(\pi/2,\cdots,\pi/2,-\pi/2)\in\R^n$, and $\mathfrak{R}_{\be^*}$ be a rotation defined by $\mathfrak{R}_{\be^*}(x,y)=(\hat{x}_1,\cdots,\hat{x}_n,\hat{y}_1,\cdots,\hat{y}_n)$ with
\begin{equation}\label{Rbe**}
\left\{\begin{split}
&\hat{x}_i=(\cos\be^*_i) x_i+(\sin\be^*_i) y_i\\
&\hat{y}_j=-(\sin\be^*_j) x_j+(\cos\be^*_j) y_j\\
\end{split}\right.\qquad\mathrm{for\ each}\ 1\le i,j\le n.
\end{equation}
Denote $\Si=\mathfrak{R}_{\be^*}(L)$. 
Let $\bar{x}=(\bar{x}_1,\cdots,\bar{x}_n):\, B_1\to\R^n$ be a mapping defined by $\bar{x}_i=(\cos\be^*_i) x_i+(\sin\be^*_i)u_i(x)$.
We suppose
\begin{equation}\aligned\label{A0vert}
\left|\mathscr{A}(\mathbf{0})-\mathrm{diag}\{\be^*_1,\cdots,\be^*_n\}\right|<\f1{10n}.
\endaligned
\end{equation}
Noting $|D\th|\le\tau_*$ on $B_1$.
Combining \eqref{Allard-req}\eqref{A0vert} and Allard's regularity theorem, there is a constant $\de\in(0,1/3)$ depending only on $n$ so that $\Si\cap\mathbf{B}_{2\de}$ is a connected Lagrangian graph over some open set $W$ of $\R^n$ with the graphic function $Dw$ for some $w\in C^4(W)$ and $B_\de\subset W$ so that $|w(0^n)|=|Dw(0^n)|=0$, $|D^2w|\le1$ on $B_\de$ and 
\begin{equation}\aligned\label{W3q}
|D^2w|_{C^\a(B_\de)}< c(n,\a)\qquad \mathrm{for\ each}\ \a\in(0,1).
\endaligned
\end{equation}
Here, $c(n,\a)$ is a positive constant depending only on $n,\a$.
From \eqref{Dbarubebxx} in Appendix I, 
\begin{equation}\aligned\label{Dwbxxbe**}
Dw(\bar{x}(x))=-x\sin\mathbf{S}_{\be^*}=-x\,\mathrm{diag}\{1,\cdots,1,-1\}=(-x_1,\cdots,-x_{n-1},x_n).
\endaligned
\end{equation}
Let $w_i:=\p_{x_i}w$, $w_{ij}:=\p^2_{x_ix_j}w$ and so on. 
Then \eqref{Dwbxxbe**} infers $x_n=w_n(\bar{x}(x))$, and $x_i=-w_i(\bar{x}(x))$ for $1\le i\le n-1$.
Combining \eqref{D2barubebxx} in Appendix I, there holds
$$\mathrm{tr}(\arctan D^2w)\Big|_{\bar{x}(x)}=\mathrm{tr}(\arctan D^2u)(x)-\sum_{i=1}^n\be^*_i=\th(-w_1,\cdots,-w_{n-1},w_n)\Big|_{\bar{x}(x)}-\f{n-2}2\pi$$
on $B_\de$.
We set 
\begin{equation}\aligned\label{tildeth}
\hat{\th}(y_1,\cdots,y_n):=\th(-y_1,\cdots,-y_{n-1},y_n)-\f{n-2}2\pi\qquad \mathrm{for\ every}\ (y_1,\cdots,y_n)\in B_\de. 
\endaligned
\end{equation}
Then
\begin{equation}\aligned
\mathrm{tr}(\arctan D^2w)=\hat{\th}(Dw)\qquad \mathrm{on}\ B_\de,
\endaligned
\end{equation}
and
\begin{equation}\aligned\label{gijwijk}
\sum_{i,j=1}^n\hat{g}^{ij}w_{ijk}=\sum_{i=1}^n\hat{\th}_iw_{ik}\qquad \mathrm{on}\ B_\de,
\endaligned
\end{equation}
where $\hat{g}^{-1}=(\hat{g}^{ij})$ is the inverse matrix of $\hat{g}:=I_n+(D^2w)^2$.
From $\f{n-2}2\pi\le \th<\f{n}2\pi$, we get $\hat{\th}\in[0,\pi)$.
Let $\la_1\ge\cdots\ge\la_n$ denote the eigenvalues of $D^2u$. 
In particular, $\la_n<0<\la_{n-1}$ on $\bar{x}^{-1}(B_\de)$ from \eqref{A0vert}.
There is an orthonormal basis $\g=(\g_1,\cdots,\g_n)$ 
so that 
\begin{equation}\aligned\label{D2gilayXi**}
(D^2u)\g_i=\la_i\g_i\qquad \mathrm{and}\qquad D^2u=\g\Xi \g^T
\endaligned
\end{equation}
with $\Xi=\mathrm{diag}\{\la_1,\cdots,\la_n\}$. We suppose $\g_n(0^n)=E_n^T$ up to a rotation.
Let $\hat{\la}_1\ge\cdots\ge\hat{\la}_n$ denote the eigenvalues of $D^2w$. 
\begin{lemma}\label{wnnge0}
For the suitably small $\de>0$, there holds
\begin{equation}\aligned\label{wnnbla1}
w_{nn}\ge\hat{\la}_1/2>0\qquad\qquad \mathrm{on}\ B_\de.
\endaligned
\end{equation}
\end{lemma}
\begin{proof}
From \eqref{D2uSinSv} in Appendix I, and $\la_n<0$ on $\bar{x}^{-1}(B_\de)$, we have
\begin{equation}\aligned\label{hatlala}
\hat{\la}_1(\bar{x}(x))=-\f1{\la_n(x)}>0,\qquad \hat{\la}_{i+1}(\bar{x}(x))=-\f1{\la_i(x)} \quad \mathrm{for\ each}\ i\in\{1,\cdots,n-1\}.
\endaligned
\end{equation}
Combining with \eqref{D2gilayXi**}, it follows that
\begin{equation}\aligned
(D^2w)\hat{\g}_n=\hat{\la}_{1}\hat{\g}_n,\qquad (D^2w)\hat{\g}_i=\hat{\la}_{i+1}\hat{\g}_i,
\endaligned
\end{equation}
where $\hat{\g}_i(\bar{x}(x))=\g_i(x)$ for each $i\in\{1,\cdots,n\}$.
Hence, there is a function $a=(a_1,\cdots,a_n)$ on $B_\de$ with $|a|=1$ and $E_n^T=\sum_{i=1}^na_i\hat{\g}_i$. Set $\hat{\la}_{n+1}=\hat{\la}_1$.
Then
\begin{equation}\aligned
w_{nn}=\lan E_n^T,(D^2w)E_n^T\ran=\sum_{i=1}^n\lan E_n^T,a_i\hat{\la}_{i+1}\hat{\g}_i\ran=\sum_{i=1}^n\hat{\la}_{i+1}a_i^2\qquad  \mathrm{on}\ B_\de.
\endaligned
\end{equation}
From \eqref{A0vert}, we get that $\hat{\la}_1$ has multiplicity one w.r.t. $D^2w$ on $B_\de$.
From Proposition \ref{ThiQxx'} and $\g_n(0^n)=E_n^T$, we can choose the suitably small $\de>0$ so that $|\g_n-E_n|$ is sufficiently small on $B_\de$. Then it can give
\begin{equation}\aligned
\sum_{i=1}^{n-1}a_i^2+|1-a_n|^2<\f1{16}\qquad  \mathrm{on}\ B_\de.
\endaligned
\end{equation}
From $|\la_n|\le\la_i$ for each $i$, and \eqref{hatlala}, we get $|\hat{\la}_i|\le\hat{\la}_1$.
Therefore,
\begin{equation}\aligned
w_{nn}\ge\hat{\la}_1a_n^2-\sum_{i=1}^{n-1}|\hat{\la}_i|a_i^2\ge\hat{\la}_1\left(1-\f14\right)^2-\hat{\la}_1\sum_{i=2}^na_i^2\ge\hat{\la}_1/2\qquad  \mathrm{on}\ B_\de,
\endaligned
\end{equation}
as desired.
\end{proof}
Let $\na=\na^\Si$, and $\De_\Si$ be the Laplacian of $\Si$. Then
\begin{equation}\aligned
\De_{\Si}\mathbf{x}=J\na\hat{\th}=&\sum_{j,k=1}^n\hat{g}^{jk}(\na_{E_j+\sum_iw_{ij}E_{n+i}}\hat{\th})J\left(E_k+\sum_{k'=1}^nw_{kk'}E_{n+k'}\right)\\
=&\sum_{i,j,k=1}^n\hat{g}^{jk}\hat{\th}_iw_{ij}\left(E_{n+k}-\sum_{k'=1}^nw_{kk'}E_{k'}\right).
\endaligned
\end{equation}
In particular, for each $j\in\{1,\cdots,n\}$
\begin{equation}\aligned\label{DeSixj}
\f1{\sqrt{\det \hat{g}}}\sum_{i=1}^n\p_i(\sqrt{\det \hat{g}}\hat{g}^{ij})=\De_{\Si}x_j=-\sum_{i,k,l=1}^n\hat{g}^{kl}\hat{\th}_iw_{il}w_{jk}.
\endaligned
\end{equation}
We take the derivative of
\begin{equation}\aligned\label{DeSiwn**}
\De_{\Si}w_n=\sum_{i,j=1}^n\hat{g}^{ij}w_{ijn}+\sum_{j=1}^nw_{jn}\De x_j
\endaligned
\end{equation}
w.r.t. $x_n$, and get
\begin{equation}\aligned\label{pxnDeSiwn=}
\p_{x_n}(\De_{\Si}w_n)=&\sum_{i,j=1}^n\left(\hat{g}^{ij}w_{ijnn}+\left(\p_{x_n}\hat{g}^{ij}\right)w_{ijn}\right)+\sum_{j=1}^n\left(w_{jnn}\De x_j+w_{jn}\p_{x_n}(\De x_j)\right)\\
=&\De_{\Si}w_{nn}+\sum_{i,j=1}^n\left(\p_{x_n}\hat{g}^{ij}\right)w_{ijn}+\sum_{j=1}^nw_{jn}\p_{x_n}(\De x_j).
\endaligned
\end{equation}
From \eqref{gijwijk} and \eqref{DeSiwn**}, we get
\begin{equation}\aligned\label{DeSiwn***}
\De_{\Si}w_n=\sum_{i=1}^n\hat{\th}_iw_{in}+\sum_{j=1}^nw_{jn}\De x_j.
\endaligned
\end{equation}
Substituting \eqref{DeSiwn***} into the left hand of \eqref{pxnDeSiwn=} gives
\begin{equation}\aligned\label{DeSiwnn}
\De_{\Si}w_{nn}=\p_{x_n}\left(\sum_{i=1}^n\hat{\th}_iw_{in}\right)-\sum_{i,j=1}^n\left(\p_{x_n}\hat{g}^{ij}\right)w_{ijn}+\sum_{j=1}^nw_{jnn}\De_\Si x_j.
\endaligned
\end{equation}
Let $A_\Si$ denote the second fundamental form of $\Si$ in $\R^n\times\R^n$. We set 
$$ f=w_{nn},\qquad\mathbf{h}=\sum_{i=1}^n\hat{\th}_iw_{in},\qquad \Psi=|A_\Si|+|A_\Si|^2$$ 
and 
$$\Phi=\f1{w_{nn}}\left(\sum_{j=1}^nw_{jnn}\De_\Si x_j-\sum_{i,j=1}^n\p_{x_n}\left(\hat{g}^{ij}\right)w_{ijn}\right).$$
From \eqref{W3q}, Lemma \ref{wnnge0} and \eqref{DeSixj}, we have 
\begin{equation}\aligned\label{hPhiB}
|\mathbf{h}|\le c f,\qquad |\Phi|\le c f|A_\Si|+c|A_\Si|^2\le c\Psi.
\endaligned
\end{equation}
Here, $c\ge1$ is a general constant depending only on $n$,
which may change from line to line.
We rewrite the equation \eqref{DeSiwnn} as follows:
\begin{equation}\aligned\label{DeSibfw}
\De_{\Si} f=\p_{x_n}\mathbf{h}+\Phi f.
\endaligned
\end{equation}
We fix a point $z=(z',z'')\in\Si\cap\mathbf{B}_{\de/4}$. For each $q\ge0$, $r\in(0,1)$ and $F\in L^q(\Si\cap \mathbf{B}_r(z))$, we define
\begin{equation}\aligned
||F||_{q,r}=\left(\fint_{\Si\cap \mathbf{B}_r(z)}|F|^q\right)^{1/q}:=\left(\f1{\mathcal{H}^n(\Si\cap \mathbf{B}_r(z))}\int_{\Si\cap \mathbf{B}_r(z)}|F|^q\right)^{1/q}.
\endaligned
\end{equation}
Noting $|Dw(0^n)|=0$ and $|D^2w|\le1$ on $B_\de$. From \eqref{hijk} and \eqref{gijwijk}, for $q>1$ we have a $L^q$-estimate (see (9.40) in \cite{GT} for instance):
\begin{equation}\aligned\label{W3q*}
\r^q\fint_{\Si\cap \mathbf{B}_{\r}(z)}|A_\Si|^q\le c\r^q\fint_{B_\r(z')}|D^3w|^q\le c(n,q)\qquad \mathrm{for\ any}\ \r\in(0,\de/2].
\endaligned
\end{equation}

Let us carry out a (modified) De Giorgi-Nash-Moser iteration (see \cite{M} for an early version).
\begin{lemma}\label{SubmeanV}
For each $k\ge1$, there is a constant $c_{k}$ depending only on $n,k$ so that
\begin{equation}\aligned
|| f||_{\infty,r/2}\le c_{k,\k,\La}|| f||_{k,r}\qquad \mathrm{for\ each}\ 0<r<\de/4.
\endaligned
\end{equation}
\end{lemma}
\begin{proof}
Let $\e$ be a nonnegative Lipschitz function on $\Si$ with compact support in $\Si\cap\mathbf{B}_{\de/2}(z)$. For $\ell\ge1$, from \eqref{DeSibfw}
\begin{equation}\aligned\label{2ell-1wdWe}
&(2\ell-1)\int_\Si f^{2\ell-2}|\na f|^2\e^2=-\int_\Si f^{2\ell-1}\e^2\De_\Si f-2\int_\Si f^{2\ell-1}\e\lan\na\e,\na f\ran\\
=&-\int_\Si f^{2\ell-1}\e^2\p_{x_n}\mathbf{h}-\int_\Si f^{2\ell}\e^2\Phi-2\int_\Si f^{2\ell-1}\e\lan\na\e,\na f\ran\\
=&\int_{\R^n}\p_{x_n}\left( f^{2\ell-1}\e^2\sqrt{\det \hat{g}}\right)\mathbf{h}-\int_\Si f^{2\ell}\e^2\Phi-2\int_\Si f^{2\ell-1}\e\lan\na\e,\na f\ran.
\endaligned
\end{equation}
Since $|D^2w|\le1$ on $B_\de$, we get $|\p_{x_n}\sqrt{\det \hat{g}}|\le c f|A_\Si|$ on $B_\de$.
From \eqref{hPhiB}\eqref{2ell-1wdWe}, we have
\begin{equation}\aligned
(2\ell-1)\int_\Si f^{2\ell-2}&|\na f|^2\e^2\le c(2\ell-1)\int_{\Si} f^{2\ell-1}|\na f|\e^2+c\int_\Si f^{2\ell}\e|\na\e|\\
&+c\int_\Si f^{2\ell+1}\e^2|A_\Si|-\int_\Si f^{2\ell}\e^2\Phi-2\int_\Si f^{2\ell-1}\e\lan\na\e,\na f\ran.
\endaligned
\end{equation}
With Cauchy-Schwartz inequality, we have
\begin{equation}\aligned\label{w2l-2CSI}
\f{2\ell-1}4\int_\Si f^{2\ell-2}&|\na f|^2\e^2\le c\ell\int_{\Si} f^{2\ell}\e^2+\f{c}{\ell}\int_\Si f^{2\ell}|\na\e|^2+c\int_\Si f^{2\ell}\e^2\Psi.
\endaligned
\end{equation}
Hence,
\begin{equation}\aligned
\int_\Si\left|\na( f^{\ell}\e)\right|^2\le& 2\ell^2\int_\Si f^{2\ell-2}|\na f|^2\e^2+2\int_\Si f^{2\ell}|\na\e|^2\\
\le& c\ell^2\int_{\Si} f^{2\ell}\e^2+c\int_\Si f^{2\ell}|\na\e|^2+c\ell\int_\Si f^{2\ell}\e^2\Psi.
\endaligned
\end{equation}
Suppose $n\ge3$.
With Sobolev inequality, for each $\ep\in(0,1]$ we have
\begin{equation}\aligned
&\left(\int_\Si( f^{\ell}\e)^{\f{2n}{n-2}}\right)^{\f{n-2}{n}}\le c\ell^2\int_{\Si} f^{2\ell}\e^2+c\int_\Si f^{2\ell}|\na\e|^2\\
&\qquad\qquad\qquad\qquad\quad+\f{4c^2\ell^2}{\ep}\int_\Si f^{2\ell}\e^2+\ep\int_\Si f^{2\ell}\e^2\Psi^2\\
&\le
\f{c\ell^2}\ep \int_{\Si} f^{2\ell}\e^2+c\int_\Si f^{2\ell}|\na\e|^2+\ep\left(\int_\Si( f^{\ell}\e)^{\f{2n}{n-2}}\right)^{\f{n-2}{n}}\left(\int_{\Si\cap\mathrm{supp}\e}\Psi^n\right)^{\f 2n}.
\endaligned
\end{equation}
From \eqref{W3q*} and the definition of $\Psi$, we fix a suitably small $\ep>0$ so that
\begin{equation}\aligned
\left(\int_\Si( f^{\ell}\e)^{\f{2n}{n-2}}\right)^{\f{n-2}{n}}\le c\ell^2\int_{\Si} f^{2\ell}\e^2+c\int_\Si f^{2\ell}|\na\e|^2.
\endaligned
\end{equation}
For $0<\tau\le r$,
let $\e$ be a Lipschitz function defined by $\e=1$ on $\mathbf{B}_{r}(z)$, $\e=0$ outside $\mathbf{B}_{r+\tau}(z)$ so that $|\bn\e|\le\f1{\tau}$ on $\mathbf{B}_{r+\tau}(z)\setminus \mathbf{B}_{r}(z)$. Then we have
\begin{equation}\aligned
\f1{r^2}\left(\fint_{\Si\cap \mathbf{B}_r(z)} f^{\f{2n\ell}{n-2}}\right)^{\f{n-2}{n}}\le c\ell^2\fint_{\Si\cap \mathbf{B}_{r+\tau}(z)} f^{2\ell}+c\tau^{-2}\fint_{\Si\cap \mathbf{B}_{r+\tau}(z)} f^{2\ell},
\endaligned
\end{equation}
which implies
\begin{equation}\aligned\label{ellrbfw}
\left(\fint_{\Si\cap \mathbf{B}_r(z)} f^{\f{2n\ell}{n-2}}\right)^{\f{n-2}{2n\ell}}\le c^{\f1{\ell}}\left(\ell r+\f r\tau\right)^{\f1{\ell}}\left(\fint_{\Si\cap \mathbf{B}_{r+\tau}(z)} f^{2\ell}\right)^{\f1{2\ell}}.
\endaligned
\end{equation}

Fix a constant $k\ge1$. Let $\ell_i=2k\left(\f n{n-2}\right)^i$, $\tau_i=2^{-i-2}r$, and $r_i=r-\sum_{j=0}^i\tau_j\ge r/2$.
Let $\e_i$ be a Lipschitz function defined by $\e_i=1$ on $\mathbf{B}_{r_i}(z)$, $\e_i=0$ outside $\mathbf{B}_{r_i+\tau_i}(z)$ so that $|\bn\e_i|\le\f1{\tau_i}$ on $\mathbf{B}_{r_i+\tau_i}(z)\setminus  \mathbf{B}_{r_i}(z)$.
From \eqref{ellrbfw}, we have
\begin{equation}\aligned
|| f||_{\ell_i,r_i}\le c^{\f1{2\ell_{i-1}}}\left(\ell_{i-1}r_i+\f{r_i}{\tau_i}\right)^{\f1{\ell_{i-1}}}|| f||_{\ell_{i-1},r_{i-1}}.
\endaligned
\end{equation}
Let $n_*=\max\{2,n/(n-2)\}$. Then
\begin{equation}\aligned
|| f||_{\ell_i,r_i}\le c^{\f1{2\ell_{i-1}}}(4+kr)^{\f1{2\ell_{i-1}}}n_*^{\f i{\ell_{i-1}}}|| f||_{\ell_{i-1},r_{i-1}}.
\endaligned
\end{equation}
Now we iterate the above inequality, and complete the proof of $n\ge3$ (see (4.7)-(4.10) in \cite{D0} for instance). For $n=2$, the index $\f{2n}{n-2}$ can be replaced by any large finite constant in the Sobolev inequality, and the above argument still works.
\end{proof}
We follow the proof in Theorem 4.2 in \cite{D0} step by step with (4.3) in \cite{D0} replaced by Lemma \ref{SubmeanV}, then immediately get the following result.
\begin{corollary}\label{MVsub}
For any $q>0$, there is a constant $c_{q}$ depending only on $n,q$ so that
\begin{equation}\aligned
|| f||_{\infty,r/2}\le c_{q}|| f||_{q,r}\qquad \mathrm{for\ each}\ 0<r<\de/4.
\endaligned
\end{equation}
\end{corollary}
Moreover, we derive another mean value inequality, which will be used later.
\begin{lemma}
There is a constant $c$ depending only on $n$ so that
\begin{equation}\aligned
|| f^{-\si}||_{\infty,r/2}\le c|| f^{-\si}||_{1,r}\qquad \mathrm{for\ each}\ 0<r<\de/4,\ \si\in(0,1].
\endaligned
\end{equation}
\end{lemma}
\begin{proof}
Given the constant $\si\in(0,1]$, let $\mathfrak{S}= f^{-\si}$. Then $|\mathbf{h}|\le c\mathfrak{S}^{-1/\si}$ and
\begin{equation}\aligned\label{DemfSge}
&\De_\Si\mathfrak{S}=-\si f^{-\si-1}\De_\Si f+\si(\si+1) f^{-\si-2}|\na f|^2\\
=&-\si f^{-\si-1}\p_{x_n}\mathbf{h}-\si f^{-\si}\Phi+\si(\si+1) f^{-\si-2}|\na f|^2\\
\ge&-\si\mathfrak{S}^{1+1/\si}\p_{x_n}\mathbf{h}-\si\mathfrak{S}\Phi.
\endaligned
\end{equation}
Let $\e$ be a nonnegative Lipschitz function on $\Si$ with compact support in $\Si\cap\mathbf{B}_{\de/2}(z)$. For $\ell\ge1$, from \eqref{DemfSge}
\begin{equation}\aligned
&(2\ell-1)\int_\Si\mathfrak{S}^{2\ell-2}|\na\mathfrak{S}|^2\e^2=-\int_\Si\mathfrak{S}^{2\ell-1}\e^2\De_\Si\mathfrak{S}-2\int_\Si\mathfrak{S}^{2\ell-1}\e\lan\na\e,\na\mathfrak{S}\ran\\
\le&-\si\int_\Si\mathfrak{S}^{2\ell+1/\si}\e^2\p_{x_n}\mathbf{h}-\si\int_\Si\mathfrak{S}^{2\ell}\e^2\Phi-2\int_\Si\mathfrak{S}^{2\ell-1}\e\lan\na\e,\na\mathfrak{S}\ran\\
=&\si\int_{\R^n}\p_{x_n}\left(\mathfrak{S}^{2\ell+1/\si}\e^2\sqrt{\det \hat{g}}\right)\mathbf{h}-\si\int_\Si\mathfrak{S}^{2\ell}\e^2\Phi-2\int_\Si\mathfrak{S}^{2\ell-1}\e\lan\na\e,\na\mathfrak{S}\ran.
\endaligned
\end{equation}
From $|\p_{x_n}\sqrt{\det \hat{g}}|\le  c|A_\Si|$ and $|\mathbf{h}|\le c\mathfrak{S}^{-1/\si}$, we have
\begin{equation}\aligned
(2\ell-1)\int_\Si\mathfrak{S}^{2\ell-2}&|\na\mathfrak{S}|^2\e^2\le c(2\ell\si+1)\int_{\Si}\mathfrak{S}^{2\ell-1}|\na\mathfrak{S}|\e^2+c\int_\Si\mathfrak{S}^{2\ell}\e|\na\e|\\
&+c\int_\Si\mathfrak{S}^{2\ell}\e^2|A_\Si|-\si\int_\Si\mathfrak{S}^{2\ell}\e^2\Phi-2\int_\Si\mathfrak{S}^{2\ell-1}\e\lan\na\e,\na\mathfrak{S}\ran.
\endaligned
\end{equation}
Then we follow the argument from \eqref{w2l-2CSI} in Lemma \ref{SubmeanV}, and complete the proof.
\end{proof}

Let us prove a mean value inequality for 'superharmonic' functions.
\begin{lemma}\label{MVsup}
There are constants $c,\si_*>0$ depending only on $n$ so that
\begin{equation}\aligned
\fint_{\Si\cap\mathbf{B}_\r(z)} f^{\si_*}\le c f^{\si_*}(z)\qquad \mathrm{for\ each}\ 0<\r<\de/8.
\endaligned
\end{equation}
\end{lemma}
\begin{proof}
Let 
$$\phi:=\log f-c_{ f}\quad \mathrm{with}\quad c_{ f}=\fint_{\Si\cap\mathbf{B}_\r(z)}\log f.$$
Then $ f=e^{\phi+c_{ f}}$, and
\begin{equation}\aligned
\De_\Si \phi=\f{\De_\Si f}{ f}-\f{|\na f|^2}{ f^2}=\f{\p_{x_n}\mathbf{h}}{ f}+\Phi-|\na \phi|^2=e^{-\phi-c_{ f}}\p_{x_n}\mathbf{h}+\Phi-|\na \phi|^2.
\endaligned
\end{equation}
Let $\xi$ be a Lipschitz function on $\Si$ with compact support in $\Si\cap\mathbf{B}_{\de/2}(z)$ to be defined later. Noting $|\mathbf{h}|\le c f=c e^{\phi+c_{ f}}$.
Using Cauchy-Schwartz inequality we have
\begin{equation}\aligned
&\int_\Si\xi^2\left(|\na \phi|^2-\Phi\right)=\int_\Si\xi^2e^{-\phi-c_{ f}}\p_{x_n}\mathbf{h}-\int_\Si\xi^2\De_\Si \phi\\
=&-\int_{\R^n}\p_{x_n}\left(\xi^2e^{-\phi-c_{ f}}\sqrt{\det \hat{g}}\right)\mathbf{h}+\int_\Si\lan\na \phi,\na\xi^2\ran\\
\le&c\int_\Si\xi|\na\xi|+c\int_\Si\xi^2|\na \phi|+c\int_\Si\xi^2|A_\Si|+2\int_\Si\xi\lan\na \phi,\na\xi\ran\\
\le&c\int_\Si(\xi^2+|\na\xi|^2)+\f14\int_\Si|\na \phi|^2\xi^2+c\int_\Si\xi^2|A_\Si|+\f14\int_\Si\xi^2|\na \phi|^2+\int_\Si|\na\xi|^2.
\endaligned
\end{equation}
This implies
\begin{equation}\aligned
\int_\Si|\na \phi|^2\xi^2
\le c &\int_\Si(\xi^2+|\na\xi|^2)+c\int_\Si\xi^2\Psi.
\endaligned
\end{equation}
For $0<\r<\de/8$,
let $\xi$ be a Lipschitz function defined by $\xi=1$ on $\mathbf{B}_{3\r/2}(z)$, $\xi=0$ outside $\mathbf{B}_{2\r}(z)$ so that $|\bn\xi|\le\f2\r$ on $\mathbf{B}_{2\r}(z)\setminus \mathbf{B}_{3\r/2}(z)$. Then with \eqref{W3q*} we have
\begin{equation}\aligned\label{Dphi|2n-2}
\int_{\Si\cap \mathbf{B}_{3\r/2}(z)}|\na \phi|^2\le c\int_{\Si\cap \mathbf{B}_{2\r}(z)}(1+4\r^{-2})+c\int_{\Si\cap \mathbf{B}_{2\r}(z)}\Psi\le c\r^{n-2}.
\endaligned
\end{equation}
From Neumann-Poincar\'e inequality and the definition of $\phi$, we get
\begin{equation}\aligned
\fint_{\Si\cap \mathbf{B}_{3\r/2}(z)}|\phi|^2\le c\r^2\fint_{\Si\cap \mathbf{B}_{3\r/2}(z)}|\na \phi|^2\le c.
\endaligned
\end{equation}

For $q\ge1$, using Cauchy-Schwartz inequality we have
\begin{equation}\aligned\label{Bonn1}
&\int_\Si|\phi|^q\xi^2\left(|\na \phi|^2-\Phi\right)=\int_\Si|\phi|^q\xi^2e^{-\phi-c_{ f}}\p_{x_n}\mathbf{h}-\int_\Si|\phi|^q\xi^2\De_\Si \phi\\
=&-\int_{\R^n}\p_{x_n}\left(|\phi|^q\xi^2e^{-\phi-c_{ f}}\sqrt{\det \hat{g}}\right)\mathbf{h}+\int_\Si\lan\na \phi,\na(|\phi|^q\xi^2)\ran\\
\le&c q\int_\Si|\phi|^{q-1}|\na \phi|\xi^2+c\int_\Si|\phi|^{q}\xi|\na\xi|+c\int_\Si|\phi|^{q}\xi^2|\na \phi|+c\int_\Si|\phi|^{q}\xi^2|A_\Si|\\
&+2\int_\Si|\phi|^q\xi\lan\na \phi,\na\xi\ran+q\int_\Si|\phi|^{q-1}\xi^2|\na \phi|^2\\
\le&c q\int_\Si|\phi|^{q-1}(1+|\na \phi|^2)\xi^2+c\int_\Si|\phi|^q(\xi^2+|\na\xi|^2)+\f14\int_\Si|\phi|^q|\na \phi|^2\xi^2\\
&+c\int_\Si|\phi|^{q}\xi^2|A_\Si|+\f14\int_\Si|\phi|^q\xi^2|\na \phi|^2+\int_\Si|\phi|^q|\na\xi|^2+q\int_\Si|\phi|^{q-1}|\na \phi|^2\xi^2.
\endaligned
\end{equation}
This implies
\begin{equation*}\aligned
\int_\Si|\phi|^q|\na \phi|^2\xi^2\le c &q\int_\Si|\phi|^{q-1}(1+|\na \phi|^2)\xi^2+c \int_\Si|\phi|^{q}(\xi^2+|\na\xi|^2)+c\int_\Si|\phi|^{q}\xi^2\Psi.
\endaligned
\end{equation*}
From the Young's inequality: $q|\phi|^{q-1}\le\f{q-1}q|\phi|^q+q^{q-1}$ and
\begin{equation}\aligned
c q|\phi|^{q-1}\le&\f{\left((\f q{2(q-1)})^{\f{q-1}q}|\phi|^{q-1}\right)^{\f q{q-1}}}{\f q{q-1}}+\f{\left(c q(\f {2(q-1)}q)^{\f{q-1}q}\right)^{q}}{q}\\
=&\f12|\phi|^q+c^{q}(2(q-1))^{q-1},
\endaligned
\end{equation}
we get
\begin{equation}\aligned
\int_\Si|\phi|^{q}|\na \phi|^2\xi^2\le& c\int_\Si|\phi|^{q}(\xi^2+|\na\xi|^2)+c q^{q-1}\int_\Si\xi^2\\
&+c^q q^{q}\int_\Si|\na \phi|^2\xi^2+c\int_\Si|\phi|^{q}\xi^2\Psi.
\endaligned
\end{equation}
By Cauchy-Schwartz inequality, we further have
\begin{equation}\aligned
&\int_\Si|\phi|^{q}|\na \phi|\xi^2\le \f {\r}4\int_\Si|\phi|^{q}|\na \phi|^2\xi^2+\f1{\r}\int_\Si|\phi|^{q}\xi^2\\
\le& c\int_\Si|\phi|^{q}(\r^{-1}\xi^2+\r|\na\xi|^2)+c q^{q-1}\int_\Si\xi^2+c^q q^{q}\r\int_\Si|\na \phi|^2\xi^2+c\int_\Si|\phi|^{q}\xi^2\Psi.
\endaligned
\end{equation}
With Sobolev inequality, for all $q\ge2$ 
\begin{equation}\aligned
\left(\int_\Si|\phi^{q}\xi^2|^{\f{n}{n-1}}\right)^{\f{n-1}{n}}\le& c\int_\Si|\na(\phi^{q}\xi^2)|\le cq\int_\Si|\phi|^{q-1}|\na \phi|\xi^2+2c\int_\Si|\phi|^{q}|\na \xi|\xi\\
\le& c q\int_\Si|\phi|^{q-1}(\r^{-1}\xi^2+\r|\na\xi|^2)+c q^{q}\int_\Si\xi^2\\
&+c^q q^{q}\r\int_\Si|\na \phi|^2\xi^2+c q\int_\Si|\phi|^{q-1}\xi^2\Psi+c\int_\Si|\phi|^{q}|\na \xi|\xi.
\endaligned
\end{equation}
For each $\ep\in(0,1]$ and $\tau>0$, there holds
\begin{equation}\aligned
c q|\phi|^{q-1}\tau\le&\f{\left(\ep^{\f{q-1}q}|\phi|^{q-1}\tau\right)^{\f q{q-1}}}{\f q{q-1}}+\f{\left(\ep^{\f{1-q}q}c q\right)^{q}}{q}=\ep\f{q-1}q|\phi|^{q}\tau^{\f{q}{q-1}}+\ep^{1-q}c^q q^{q-1},
\endaligned
\end{equation}
then with Cauchy-Schwartz inequality we have
\begin{equation}\aligned
\left(\int_\Si|\phi^{q}\xi^2|^{\f{n}{n-1}}\right)^{\f{n-1}{n}}
\le & c\int_\Si|\phi|^{q}(\r^{-1}\xi^2+\r|\na\xi|^2)+c^q q^{q}\int_\Si(\r^{-1}\xi^2+\r|\na\xi|^2)\\
&+c^q q^{q}\r\int_\Si|\na \phi|^2\xi^2+\ep\int_\Si|\phi|^{q}\xi^2\Psi^{\f{q}{q-1}}+\ep^{1-q}c^q q^{q}\int_\Si\xi^2.
\endaligned
\end{equation}
From \eqref{W3q*} and H\"older inequality for $\int_\Si|\phi|^{q}\xi^2\Psi^{\f{q}{q-1}}$ with the suitably small (fixed) $\ep>0$, we get
\begin{equation}\aligned\label{phiqxinn-1d}
\left(\int_\Si|\phi^{q}\xi^2|^{\f{n}{n-1}}\right)^{\f{n-1}{n}}
\le& c \int_\Si|\phi|^{q}(\r^{-1}\xi^2+\r|\na\xi|^2)\\
&+c^q q^{q}\int_\Si\left(\r|\na \phi|^2+\r^{-1})\xi^2+\r|\na\xi|^2\right).
\endaligned
\end{equation}

Let $\r_j=(1+2^{-j-1})\r$ for each integer $j\ge0$.
Let $\xi_j$ be a cut-off function on $\R^n\times\R^n$ defined by $\xi_j=1$ on $\mathbf{B}_{\r_{j+1}}(z)$, $\xi_j=\f {\r_{j}-|\mathbf{x}|}{\r_{j}-\r_{j+1}}$ on $\mathbf{B}_{\r_{j}}(z)\setminus \mathbf{B}_{\r_{j+1}}(z)$,
$\xi_j=0$ outside $\mathbf{B}_{\r_j}(z)$.
Then $|\bn\xi_j|\le 2^{j+2}/\r$. From \eqref{Dphi|2n-2} and \eqref{phiqxinn-1d}, one has
\begin{equation}\aligned
\left(\int_{\Si\cap \mathbf{B}_{\r_{j+1}}(z)}|\phi|^{\f{nq}{n-1}}\right)^{\f{n-1}{n}}
\le \f{2^{2j}c}{2\r} \int_{\Si\cap \mathbf{B}_{\r_j}(z)}|\phi|^{q}+c^q q^{q}2^{2j}\r^{n-1}.
\endaligned
\end{equation}
Let $q_j=2(\f{n}{n-1})^j$ for each integer $j\ge0$. Then 
\begin{equation}\aligned
||\phi||_{q_{j+1},\r_{j+1}}^{q_{j}}\le 2^{2j}\left(c||\phi||_{q_{j},\r_{j}}^{q_{j}}+c^{q_j}\, q_j^{q_j}\right),
\endaligned
\end{equation}
which implies
\begin{equation}\aligned\label{||phi||j}
||\phi||_{q_{j+1},\r_{j+1}}\le 2^{\f{2j}{q_j}}c^{\f1{q_j}}||\phi||_{q_{j},\r_{j}}+c q_j.
\endaligned
\end{equation}
Let $a_0=||\phi||_{2,\r_0}/2$, and 
$$a_j=\f{||\phi||_{q_j,\r_j}}{q_j}\prod_{0\le i\le j-1}2^{-\f{2i}{q_i}}c^{-\f1{q_i}}\qquad\mathrm{for\ each}\ j\ge1,$$
then \eqref{||phi||j} implies
\begin{equation}\aligned
a_{j+1}\le\f{1}{q_{j+1}}\left(2^{\f{2j}{q_j}}c^{\f1{q_j}}||\phi||_{q_{j},\r_{j}}+cq_j\right)\prod_{0\le i\le j}2^{-\f{2i}{q_i}}c^{-\f1{q_i}}\le \left(1-\f1n\right)a_j+c.
\endaligned
\end{equation}
By iteration,
\begin{equation}\aligned
&a_{j+1}-nc\le  \left(1-\f1n\right)(a_j-nc)\le\cdots\le\left(1-\f1n\right)^{j+1}(a_0-nc).
\endaligned
\end{equation}
Since $\prod_{i=0}^{\infty}2^{-\f{2i}{q_i}}c^{-\f1{q_i}}>0$, we get
$$||\phi||_{q_j,\r_j}\le cq_j.$$
Noting we have allowed that $c$ may change from line to line, but $c$ depends only on $n$.
For each integer $k\ge 1$, there is an integer $j_k\ge 0$ such that $q_{j_k}\le k\le q_{j_{k+1}}$.
Noting $\r_j\ge\r$.
With H\"older inequality, we have
\begin{equation}\aligned
||\phi||_{k,\r}\le||\phi||_{k,\r_{j_k+1}}\le||\phi||_{q_{j_k+1},\r_{j_k+1}}\le cq_{j_k+1}\le \f {nc}{n-1}k.
\endaligned
\end{equation}
By Taylor's expansion (for $e^t$), there is a constant $\si_*\in(0,1)$ depending only on $n$ so that (see (3.32)-(3.35) in \cite{D1} for instance)
\begin{equation}\aligned
\fint_{\Si\cap\mathbf{B}_\r(z)} f^{\si_*}\fint_{\Si\cap\mathbf{B}_\r(z)} f^{-\si_*}=\fint_{\Si\cap\mathbf{B}_\r(z)}e^{\si_*\phi}\fint_{\Si\cap\mathbf{B}_\r(z)}e^{-\si_*\phi}\le c.
\endaligned
\end{equation}
From Lemma \ref{MVsup}, we get
\begin{equation}\aligned
 f^{-\si_*}(z)\fint_{\Si\cap\mathbf{B}_\r(z)} f^{\si_*}\le c\fint_{\Si\cap\mathbf{B}_\r(z)} f^{\si_*}\fint_{\Si\cap\mathbf{B}_\r(z)} f^{-\si_*}\le c^2,
\endaligned
\end{equation}
which implies
\begin{equation}\aligned
\fint_{\Si\cap\mathbf{B}_\r(z)} f^{\si_*}\le c^2 f^{\si_*}(z).
\endaligned
\end{equation}
This completes the proof.
\end{proof}

Combining Corollary \ref{MVsub} and Lemma \ref{MVsup}, we get the following Harnack's inequality. 
\begin{theorem}\label{Harnack}
There is a constant $c>0$ depending only on $n$ so that
\begin{equation}\aligned
\sup_{\Si\cap\mathbf{B}_r(z)} f\le c\inf_{\Si\cap\mathbf{B}_r(z)} f\qquad\qquad \mathrm{for\ each}\ 0<r<\de/8,\ z\in\Si\cap\mathbf{B}_{\de/4}.
\endaligned
\end{equation}
\end{theorem}

\section{Hessian estimates with constraints}

Let $(u,\th)\in\mathbb{F}_n(\La,\k)$ for some $\La,\k\ge0$ and $n\ge3$.
We assume that $\la_1,\cdots,\la_n$ are the distinct eigenvalues of $D^2u$ on $B_1$ satisfying 
\begin{equation}\aligned\label{distincteachla}
\la_1>\la_2>\cdots>\la_n.
\endaligned
\end{equation}
Let $L=G_{Du}$.
We define $\g=(\g_1,\cdots,\g_n)$, $\{e_i\}_{i=1}^n$, $\{\nu_i\}_{i=1}^n$ and $h_{ijk}$ as in \eqref{gg1n}, \eqref{einui} and \eqref{hijk}, respectively.
For an integer $m\in\{1,\cdots, n-2\}$, we define a function $v_m$ on $B_1$ as in \eqref{vm*}.
From $\si_2\ge0$ in \eqref{sik***}, using the case of dimension 3, we have
\begin{equation}\aligned\label{laijkge0}
\la_i\la_j+\la_j\la_k+\la_k\la_i\ge0\qquad \mathrm{for\ each}\ i<j<k.
\endaligned
\end{equation}
From Lemma \ref{DeMla_1=} and \eqref{laijkge0}, we have
\begin{equation}\aligned\label{DeLlogv*}
\De_L\log v_m\ge&\sum_{k=1}^{m}(1+\la_k^2)h_{kkk}^2+\sum_{\stackrel{i,k=1}{i\neq k}}^{m}(3+\la_i^2+2\la_i\la_k)h_{iik}^2\\
&+\sum_{k\le m<l}\f{2\la_k(1+\la_k\la_l)}{\la_k-\la_l}h_{llk}^2
+\sum_{l\le m<k}\f{3\la_l-\la_k+\la_l^2(\la_l+\la_k)}{\la_l-\la_k}h_{llk}^2\\
&+\sum_{i=1}^{m}\f{\la_i}{1+\la_i^2}\th_{\g_i\g_i}-\sum_{i=1}^{m}\f{\la_i}{1+\la_i^2}\th_{\g_i}D_{\g_i}\log v_m.
\endaligned
\end{equation}
For $1\le k\le m$, let 
\begin{equation}\aligned\label{Tklen-2*}
T_k:=(1+\la_k^2)h_{kkk}^2+\sum_{i\le m,i\neq k}(3+\la_i^2+2\la_i\la_k)h_{iik}^2+\sum_{l>m}\f{2\la_k(1+\la_k\la_l)}{\la_k-\la_l}h_{llk}^2.
\endaligned
\end{equation}
For $m<k\le n$, let
\begin{equation}\aligned\label{Tklen-2*2}
T_k:=\sum_{i\le m}\f{3\la_i-\la_k+\la_i^2(\la_i+\la_k)}{\la_i-\la_k}h_{iik}^2.
\endaligned
\end{equation}
From Lemma 2.1 in Wang-Yuan \cite{WdY} (the case of dimension 3), it follows that
\begin{equation}\aligned\label{lan-2n**}
0\le2\la_n+\la_{n-2}.
\endaligned
\end{equation}
Then it's clear that
\begin{equation}\aligned\label{Tm*}
T_k\ge\f12\sum_{i\le m}(1+\la_i^2)h_{iik}^2\qquad\mathrm{for}\ k=m+1,\cdots,n.
\endaligned
\end{equation}
\begin{lemma}
For $1\le k\le m$, there holds
\begin{equation}\aligned\label{Tk*}
T_k\ge\f17\sum_{i=1}^n\la_i^2h_{iik}^2-21\th_{\g_k}^2 \qquad \ \ \mathrm{on}\ B_{1}.
\endaligned
\end{equation} 
\end{lemma}
\begin{proof}
For $\la_n\ge0$, \eqref{Tk*} clearly holds from \eqref{Tklen-2*}. Hence, we only consider the case $\la_n\le0$.
For $l=m+1,\cdots,n$, there holds
\begin{equation}\aligned
\f23\f{\la_k^2\la_l^2}{(\la_k-\la_l)(\la_k-2\la_l/3)}+\f{2\la_k^2\la_l}{\la_k-2\la_l/3}=\f{2\la_k^2\la_l}{\la_k-\la_l}.
\endaligned
\end{equation}
Noting $n-m\ge2$.
For each $1\le k\le m$, from \eqref{Tklen-2*} with Cauchy-Schwartz inequality we have
\begin{equation}\aligned\label{Tk-1/7*}
&T_k-\f17\sum_{i\le m}\la_i^2h_{iik}^2-\f23\sum_{l>m}\f{\la_k^2\la_l^2}{(\la_k-\la_l)(\la_k-2\la_l/3)}h_{llk}^2\\
\ge&\f6{7}\la_k^2h_{kkk}+2\sum_{i\le m,i\neq k}\la_i\la_kh_{iik}^2+\sum_{l>m}\f{2\la_k^2\la_l}{\la_k-2\la_l/3}h_{llk}^2\\
\ge&\la_k^2\left(\sum_{i=1}^{n-1}h_{iik}\right)^2\left(\f{7}6+\sum_{i\le m,i\neq k}\f{\la_k}{2\la_i}+\sum_{m<l<n}\f{\la_k-2\la_{l}/3}{2\la_{l}}\right)^{-1}+\f{2\la_k^2\la_nh_{nnk}^2}{\la_k-2\la_n/3}\\
\ge&\la_k^2\left(\sum_{i=1}^{n-1}h_{iik}\right)^2\left(\sum_{i\le n-1}\f{\la_k}{2\la_i}+\f13\right)^{-1}+\f{\la_k^2}{\f{\la_k}{2\la_n}-\f13}h_{nnk}^2.
\endaligned
\end{equation}
With \eqref{lan-2n**} and $\la_n\le0$, we have
\begin{equation}\aligned
\f23\f{\la_k^2}{(\la_k-\la_n)(\la_k-2\la_n/3)}\ge\f23\left(1-\f{\la_n}{\la_{n-2}}\right)^{-1}\left(1-\f{2\la_n}{3\la_{n-2}}\right)^{-1}\ge\f13=\f17+\f4{21}.
\endaligned
\end{equation}
Noting $\sum_{1\le i\le n}\la_i^{-1}\le0$ from \eqref{sik***} and $\la_n\le0$. Hence, \eqref{Tk-1/7*} gives
\begin{equation}\aligned\label{Tklaihiik2}
T_k-\f17\sum_{i\le n}\la_i^2h_{iik}^2-\f4{21}\la_n^2h_{nnk}^2\ge\la_k^2\left(\sum_{i=1}^{n-1}h_{iik}\right)^2\left(\f13-\f{\la_k}{2\la_n}\right)^{-1}+\f{\la_k^2}{\f{\la_k}{2\la_n}-\f13}h_{nnk}^2.
\endaligned
\end{equation}
From
\begin{equation}\aligned
\sum_{i=1}^nh_{iik}=\lan H_{L},\nu_k\ran=\f{\th_{\g_k}}{\sqrt{1+\la_k^2}},
\endaligned
\end{equation}
we get
\begin{equation}\aligned\label{sumhiiki*}
\left(\sum_{i=1}^{n-1}h_{iik}\right)^2=\left(h_{nnk}-\f{\th_{\g_k}}{\sqrt{1+\la_k^2}}\right)^2\ge h_{nnk}^2-\f{2h_{nnk}\th_{\g_k}}{\sqrt{1+\la_k^2}}.
\endaligned
\end{equation}
Hence, combining with Cauchy-Schwartz inequality and \eqref{Tklaihiik2} gives
\begin{equation}\aligned
T_k-\f17\sum_{i\le n}\la_i^2h_{iik}^2\ge&\f4{21}\la_n^2h_{nnk}^2-\f{2\la_k}{\sqrt{1+\la_k^2}}|h_{nnk}\th_{\g_k}|\f{\la_k}{\f13-\f{\la_k}{2\la_n}}\\
\ge&\f4{21}\la_n^2h_{nnk}^2+4\la_n|h_{nnk}\th_{\g_k}|\ge-21\th_{\g_k}^2,
\endaligned
\end{equation}
as desired.
\end{proof}

Let $g_{ij}=\de_{ij}+\sum_ku_{ik}u_{jk}$, and $\tilde{g}_{ij}=\de_{ij}+\sum_ku_{\g_i\g_k}u_{\g_j\g_k}=\de_{ij}+\la_i^2\de_{ij}$ with $\g_i=(\g_{1i},\cdots,\g_{ni})^T$ as in \eqref{gg1n}. Let $g^{-1}=(g^{ij})$ and $\tilde{g}^{-1}=(\tilde{g}^{ij})$ be the inverse matrice of $g=(g_{ij})$ and $\tilde{g}=(\tilde{g}_{ij})$, respectively.
Let $\Xi=\mathrm{diag}\{\la_1,\cdots,\la_n\}$, and $v=\sqrt{\det g}$.
From $(D^2u)\g=\g \Xi$ with $\g=(\g_{ij})$, we have $\g^T(D^2u)^2\g=\Xi^2$, $\g^Tg\g=\tilde{g}$, 
\begin{equation}\aligned
\tilde{g}^{-1}=(\g^Tg\g)^{-1}=\g^Tg^{-1}\g \quad\mathrm{and}\quad g^{-1}=\g\tilde{g}^{-1}\g^T.
\endaligned
\end{equation}
Hence, 
\begin{equation}\aligned
(D^2u) g^{-1}=(D^2u)\g\tilde{g}^{-1}\g^T=\g \Xi\tilde{g}^{-1}\g^T.
\endaligned
\end{equation}
In particular, for each $j,k\in\{1,\cdots,n\}$
\begin{equation}\aligned\label{laithgigi}
\sum_{i=1}^n u_{ik}g^{ij}=\sum_{i=1}^{n}\f{\la_i}{1+\la_i^2}\g_{ki}\g_{ji}.
\endaligned
\end{equation}

\begin{lemma}\label{divthk*}
If we further assume $\la_{m}/2\ge\la_{m+1}\ge1$ on $B_{1}$, then there holds
\begin{equation*}\aligned
\left|\sum_{i=1}^{m}\sum_{j,k=1}^n\f{\th_k}v\p_{j}\left(\f{\la_iv\g_{ji}\g_{ki}}{1+\la_i^2}\right)\right|
\le
&\sum_{i=1}^n\sum_{k=1}^{m}\f{|h_{iik}\th_{\g_k}|}{\sqrt{1+\la_k^2}}+3\sum_{\min\{i,k\}\le m}|\la_ih_{iik}\th_{\g_k}|.
\endaligned
\end{equation*}
\end{lemma}
\begin{proof}
The distinct eigenvalues in \eqref{distincteachla} infer that all $\la_l$ and $\g_l$ are differentiable on $B_1$. More precisely,
from matrix analysis we have
\begin{equation}\aligned\label{pjlam*}
\p_{j}\la_l=\lan\g_l,\p_j(D^2u)\g_l\ran,
\endaligned
\end{equation}
and
\begin{equation}\aligned\label{pjgm*}
\p_{j}\g_l=\sum_{i\neq l}\f1{\la_l-\la_i}\lan\g_i,\p_j(D^2u)\g_l\ran\g_i.
\endaligned
\end{equation}
In fact, \eqref{pjlam*} can be obtained by taking derivative of $\la_l=D^2u(\g_l,\g_l)$. If we take derivative of $D^2u=\g\Xi\g^T$ w.r.t. $x_j$, i.e.,
$$\p_j(D^2u)=(\p_j\g)\Xi\g^T+\g(\p_j\Xi)\g^T+\g\Xi(\p_j\g)^T,$$
then \eqref{pjgm*} can be derived by computing $\lan\g_i,\p_j(D^2u)\g_l\ran$ using the above equality.
From \eqref{hijk}, \eqref{pjlam*} and \eqref{pjgm*}, for each $l$ we have
\begin{equation}\aligned\label{gjmjlam}
\sum_{j=1}^n\g_{jl}\p_{j}\la_l=\lan\g_l,D_{\g_l}(D^2u)\g_l\ran=(1+\la_l^2)^{3/2}h_{lll},
\endaligned
\end{equation}
\begin{equation}\aligned
\sum_{j=1}^n\g_{jl}\p_{j}\g_{kl}=\sum_{i\neq l}\f{\lan\g_i,D_{\g_l}(D^2u)\g_l\ran}{\la_l-\la_i}\g_{ki}=\sum_{i\neq l}\f{1+\la_l^2}{\la_l-\la_i}\sqrt{1+\la_i^2}h_{lli}\g_{ki},
\endaligned
\end{equation}
and
\begin{equation}\aligned\label{pjgjm}
\sum_{j=1}^n\p_{j}\g_{jl}=\sum_{i\neq l}\f1{\la_l-\la_i}\sum_j\lan\g_i,\p_{j}(D^2u)\g_l\ran\g_{ji}=\sum_{i\neq l}\f{1+\la_i^2}{\la_l-\la_i}\sqrt{1+\la_l^2}h_{iil}.
\endaligned
\end{equation}
From $v=\sqrt{\det g}=\prod_{i=1}^n\sqrt{1+\la_i^2}$, it follows that
\begin{equation}\aligned\label{pjvvgjm}
\sum_{j=1}^n\p_jv\g_{jl}=&D_{\g_l}e^{\f12\sum_{i=1}^n\log(1+\la_i^2)}=\sum_{i=1}^n\f{v\la_i}{1+\la_i^2}D_{\g_l}\la_i
=\sum_{i=1}^n\sqrt{1+\la_l^2}\la_ih_{iil}v.
\endaligned
\end{equation}
Substituting \eqref{gjmjlam}-\eqref{pjvvgjm} into 
\begin{equation}\aligned
\p_{j}\left(\f{\la_l}{1+\la_l^2}v\g_{jl}\g_{kl}\right)=&\f{1-\la_l^2}{(1+\la_l^2)^2}v\g_{jl}\g_{kl}\p_{j}\la_l+\f{\la_l}{1+\la_l^2}v\p_j\g_{jl}\g_{kl}\\
&+\f{\la_l}{1+\la_l^2}\p_jv\g_{jl}\g_{kl}+\f{\la_l}{1+\la_l^2}v\g_{jl}\p_j\g_{kl}
\endaligned
\end{equation}
with summation on $j$ gives
\begin{equation}\aligned\label{pjlalvggl}
\sum_{j=1}^n\p_{j}\left(\f{\la_lv\g_{jl}\g_{kl}}{1+\la_l^2}\right)=&\f{1-\la_l^2}{\sqrt{1+\la_l^2}}v\g_{kl}h_{lll}+\f{\la_lv\g_{kl}}{\sqrt{1+\la_l^2}}\sum_{i\neq l}\f{1+\la_i^2}{\la_l-\la_i}h_{iil}\\
&+\f{\la_lv\g_{kl}}{\sqrt{1+\la_l^2}}\sum_{i=1}^n\la_ih_{iil}+\sum_{i\neq l}\f{\la_l v}{\la_l-\la_i}\sqrt{1+\la_i^2}h_{lli}\g_{ki}.\\
\endaligned
\end{equation}
Moreover, we multiply $\th_k$ on the both sides of \eqref{pjlalvggl} with summation on $k$, and get
\begin{equation}\aligned
\sum_{j,k=1}^n\f{\th_k}v\p_{j}&\left(\f{\la_lv\g_{jl}\g_{kl}}{1+\la_l^2}\right)=\f{1-\la_l^2}{\sqrt{1+\la_l^2}}h_{lll}\th_{\g_l}+\f{\la_l\th_{\g_l}}{\sqrt{1+\la_l^2}}\sum_{i\neq l}\f{1+\la_i^2}{\la_l-\la_i}h_{iil}\\
&\qquad\qquad\quad+\f{\la_l\th_{\g_l}}{\sqrt{1+\la_l^2}}\sum_{i=1}^n\la_{i}h_{iil}+\sum_{i\neq l}\f{\la_l}{\la_l-\la_i}\sqrt{1+\la_i^2}h_{lli}\th_{\g_i}\\
=\f{h_{lll}\th_{\g_l}}{\sqrt{1+\la_l^2}}&+\f{\la_l\th_{\g_l}}{\sqrt{1+\la_l^2}}\sum_{i\neq l}\f{1+\la_i\la_l}{\la_l-\la_i}h_{iil}+\sum_{i\neq l}\f{\la_l}{\la_l-\la_i}\sqrt{1+\la_i^2}h_{lli}\th_{\g_i}.
\endaligned
\end{equation}
Since
\begin{equation}\aligned
\f{\la_l}{\sqrt{1+\la_l^2}}\f{1+\la_i\la_l}{\la_l-\la_i}+\f{\la_i}{\la_i-\la_l}\sqrt{1+\la_l^2}=\f1{\sqrt{1+\la_l^2}}\qquad\mathrm{for\ all}\ i\neq l,
\endaligned
\end{equation}
we have
\begin{equation}\aligned\label{thkvlamvgg}
&\sum_{l=m+1}^n\sum_{j,k=1}^n\f{\th_k}v\p_{j}\left(\f{\la_lv\g_{jl}\g_{kl}}{1+\la_l^2}\right)\\
=&\sum_{l=m+1}^{n}\sum_{i=1}^n\f{h_{iil}\th_{\g_l}}{\sqrt{1+\la_l^2}}-\sum_{l=m+1}^{n}\sum_{i\neq l}\f{\la_i}{\la_i-\la_l}\sqrt{1+\la_l^2}h_{iil}\th_{\g_l}\\&+\sum_{i=m+1}^{n}\sum_{l\neq i}\f{\la_i}{\la_i-\la_l}\sqrt{1+\la_l^2}h_{iil}\th_{\g_l}\\
=&\sum_{l=m+1}^{n}\sum_{i=1}^n\f{h_{iil}\th_{\g_l}}{\sqrt{1+\la_l^2}}+\left(\sum_{i=m+1}^{n}\sum_{l=1}^m-\sum_{l=m+1}^{n}\sum_{i=1}^m\right)\f{\la_i}{\la_i-\la_l}\sqrt{1+\la_l^2}h_{iil}\th_{\g_l}.
\endaligned
\end{equation}
From \eqref{einui} and $\lan J\g_j,E_{n+k}\ran=\lan \g_j,E_k\ran=\g_{kj}$, we have
\begin{equation}\aligned\label{thkDeLuk}
\sum_{i,j,k=1}^n\f{\th_k}v\p_j&(vg^{ij}u_{ik})=\sum_{k=1}^n\th_k\De_Lu_k=\sum_{k=1}^n\th_k\lan H_L,E_{n+k}\ran=\sum_{i,j,k=1}^nh_{iij}\lan\nu_j,E_{n+k}\ran\th_k\\
=&\sum_{i,j,k=1}^n\f{h_{iij}\th_k}{\sqrt{1+\la_j^2}}\lan J\g_j,E_{n+k}\ran=\sum_{i,j,k=1}^n\f{h_{iij}\g_{kj}\th_k}{\sqrt{1+\la_j^2}}=\sum_{i,j=1}^n\f{h_{iij}\th_{\g_j}}{\sqrt{1+\la_j^2}}.
\endaligned
\end{equation}
Combining \eqref{thkvlamvgg}\eqref{thkDeLuk}, from \eqref{laithgigi} we obtain
\begin{equation}\aligned
&\sum_{i=1}^{m}\sum_{j,k=1}^n\f{\th_k}v\p_{j}\left(\f{\la_i}{1+\la_i^2}v\g_{ji}\g_{ki}\right)\\
=&\sum_{i,j,k=1}^n\f{\th_k}v\p_j(vg^{ij}u_{ik})-\sum_{l=m+1}^n\sum_{j,k=1}^n\f{\th_k}v\p_{j}\left(\f{\la_l}{1+\la_l^2}v\g_{jl}\g_{kl}\right)\\
=&\sum_{i=1}^n\sum_{k=1}^{m}\f{h_{iik}\th_{\g_k}}{\sqrt{1+\la_k^2}}+\left(\sum_{i=1}^{m}\sum_{l=m+1}^n-\sum_{l=1}^{m}\sum_{i=m+1}^n\right)\f{\la_i}{\la_i-\la_l}\sqrt{1+\la_l^2}h_{iil}\th_{\g_l}.
\endaligned
\end{equation}
We complete the proof by $\la_{m}/2\ge\la_{m+1}\ge1$ on $B_{1}$.
\end{proof}

Let $\na$ denote the Levi-Civita connection on $L$.
We have the following Jacobi inequality involving a divergence term of derivatives of $\th$ with constraints.
\begin{lemma}\label{Jacobiv*}
If we further assume $\la_{m}/2\ge\la_{m+1}\ge1$ on $B_{1}$, then there is a constant $\de_n>0$ depending only on $n$ so that 
\begin{equation}\aligned\label{DeLogvmlowdb*}
\De_L\log v_m-\de_n|\na\log v_m|^2\ge\f{1}v\sum_{i=1}^{m}\sum_{j,k=1}^n\p_{j}\left(\f{\la_i\th_k}{1+\la_i^2}v\g_{ji}\g_{ki}\right)-\f{\La^2}{\de_n}\qquad \mathrm{on}\ B_{1}.
\endaligned
\end{equation}
\end{lemma}
\begin{proof}
Noting
\begin{equation}\aligned\label{lailai2thgigi*}
&\sum_{i=1}^{m}\f{\la_i}{1+\la_i^2}\th_{\g_i\g_i}=\sum_{i=1}^{m}\sum_{j,k=1}^{n}\f{\la_i}{1+\la_i^2}\g_{ki}\th_{jk}\g_{ji}\\
=&\f{1}v\sum_{i=1}^{m}\sum_{j,k=1}^n\p_{j}\left(\f{\la_i\th_k}{1+\la_i^2}v\g_{ji}\g_{ki}\right)-\sum_{i=1}^{m}\sum_{j,k=1}^n\f{\th_k}v\p_{j}\left(\f{\la_iv\g_{ji}\g_{ki}}{1+\la_i^2}\right).
\endaligned
\end{equation}
Substituting \eqref{Tm*}\eqref{Tk*}\eqref{lailai2thgigi*} into \eqref{DeLlogv*} gives
\begin{equation}\aligned\label{DeLlogv*Lipth*}
\De_L\log v_m\ge&\f17\sum_{i=1}^n\sum_{k=1}^{m}\la_i^2h_{iik}^2-21\La^2+\f12\sum_{i=1}^{m}\sum_{k=m+1}^n\la_i^2h_{iik}^2-\sum_{i=1}^{m}\f{\la_i\th_{\g_i}}{1+\la_i^2}D_{\g_i}\log v_m\\
&+\f{1}v\sum_{i=1}^{m}\sum_{j,k=1}^n\p_{j}\left(\f{\la_i\th_k}{1+\la_i^2}v\g_{ji}\g_{ki}\right)-\sum_{i=1}^{m}\sum_{j,k=1}^n\f{\th_k}v\p_{j}\left(\f{\la_iv\g_{ji}\g_{ki}}{1+\la_i^2}\right).
\endaligned
\end{equation}
From \eqref{hijk}\eqref{pjlam*},
by Cauchy-Schwartz inequality
\begin{equation}\aligned\label{|nalogv*|2}
|\na\log v_m|^2=&\f14\left|\sum_{i=1}^{m}\na\log (1+\la_i^2)\right|^2=\sum_{j=1}^n\f1{1+\la_j^2}\left(\sum_{i=1}^{m}\f{\la_i}{1+\la_i^2}D^3u(\g_i,\g_i,\g_j)\right)^2\\
=&\sum_{j=1}^n\left(\sum_{i=1}^{m}\la_ih_{iij}\right)^2\le m\sum_{i=1}^{m}\sum_{j=1}^n\la_i^2h_{iij}^2.
\endaligned
\end{equation}
Noting $\la_{m}/2\ge\la_{m+1}\ge1$.
Combining Lemma \ref{divthk*} and \eqref{|nalogv*|2}, we finish the proof by using Cauchy-Schwartz inequality in \eqref{DeLlogv*Lipth*} several times.
\end{proof}

Now we use Lemma \ref{Jacobiv*} to derive a Hessian estimate with constraints.
\begin{theorem}\label{Hesslan-2}
Let $(u,\th)\in\mathbb{F}_n(\La,\k,r)$ for some $\La,\k\ge0$, $r>0$, $n\ge3$, and $\la_1\ge\cdots\ge\la_n$ denote eigenvalues of $D^2u$. If $\la_{m}/2>\la_{m+1}\ge1$ on $B_{r}$ with $1\le m\le n-2$, then there exists a constant $c(n,\k,\La)>0$ depending only on $n,\k,\La$ so that
$$|D^2u|(0^n)\le c(n,\k,\La).$$
\end{theorem}
\begin{proof}
Without loss of generality, we assume $r=1$. 
Given $x_0\in B_1$, up to a rotation we assume $\g_i(x_0)=E_i^T$ for each $i$, where $\g$ is defined in \eqref{gg1n}, and $\{E_i\}_{1\le i\le n}$ denotes a standard orthonormal basis of $\R^n$.
For each $\ep>0$, we consider a sequence of functions 
$$\tilde{u}_\ep(x):=u(x)+\f{\ep}2\sum_{i=1}^n(n+1-i)x_i^2.$$ 
Denote $\mathfrak{B}=\mathrm{diag}\{n,n-1,\cdots,1\}$. We have
\begin{equation}\aligned
D^2\tilde{u}_\ep=D^2u+\ep\mathfrak{B}=\g\Xi\g^T+\ep\g\tilde{\mathfrak{B}}\g^T+\ep\left(\mathfrak{B}-\g\tilde{\mathfrak{B}}\g^T\right),
\endaligned
\end{equation}
where $\Xi=\mathrm{diag}\{\la_1,\cdots,\la_n\}$, and $\tilde{\mathfrak{B}}$
is diagonal. It's clear that
$\lim_{x\to x_0}\tilde{\mathfrak{B}}(x)=\mathfrak{B}$.
Hence, there is a function $\psi=(\psi_1,\cdots,\psi_n)$ with $\lim_{x\to x_0}|\psi|(x)=0$ so that
$\la_i+\ep(n+1-i)(1-\psi_i)$ is the $i$-th eigenvalue of $D^2u+\ep\g\tilde{\mathfrak{B}}\g^T$. 
For each $t\in[0,1]$, we define
$$\mathfrak{M}_\ep(t)=\g\Xi\g^T+\ep\g\tilde{\mathfrak{B}}\g^T+t\ep\left(\mathfrak{B}-\g\tilde{\mathfrak{B}}\g^T\right).$$
Let $\mu_1(t)\ge\cdots\ge\mu_n(t)$ denote the eigenvalues of $\mathfrak{M}_\ep(t)$, and  $\g_i(t)$ be an eigenvector w.r.t. $\mu_i(t)$ so that $(\g_1(t),\cdots,\g_n(t))$ is orthonormal.
From \eqref{pjlam*}, it follows that
\begin{equation}\aligned\label{pjlam***}
\p_t\mu_i(t)=\left\lan\g_i(t),\ep\left(\mathfrak{B}-\g\tilde{\mathfrak{B}}\g^T\right)\g_i(t)\right\ran.
\endaligned
\end{equation}
From Newton-Leibniz formula for each $\mu_i(t)$ using \eqref{pjlam***},
we conclude that there is a function $\tilde{\psi}_{,\ep}=(\tilde{\psi}_{1,\ep},\cdots,\tilde{\psi}_{n,\ep})$ with $\lim_{x\to x_0}|\tilde{\psi}_{,\ep}|(x)=0$ uniformly independent of $\ep$ so that
$\tilde{\la}_{i,\ep}:=\la_i+\ep(n+1-i)(1-\tilde{\psi}_{i,\ep})$ is the $i$-th eigenvalue of $D^2\tilde{u}_\ep$. 
Hence, there is a constant $0<\de<1-|x_0|$ (depending only on $x_0$ and $u$) so that $|\tilde{\psi}_{,\ep}|<\f1{2n}$ on $B_\de(x_0)$. This implies
$$(n+1-i)(1-\tilde{\psi}_{i,\ep})>(n-i)(1-\tilde{\psi}_{i+1,\ep})\qquad \mathrm{on}\ B_\de(x_0)$$
for each $i=1,\cdots,n-1$. Then it follows that $\tilde{\la}_{1,\ep}>\tilde{\la}_{2,\ep}>\cdots>\tilde{\la}_{n,\ep}$ on $B_\de(x_0)$. The condition $\la_{m}/2>\la_{m+1}\ge1$ on $B_{1}$ infers that $\tilde{\la}_{m,\ep}/2\ge\tilde{\la}_{m+1,\ep}\ge1$ on $B_\de(x_0)$ for all sufficiently small $\ep>0$.
Moreover, the phase $\tilde{\th}_\ep$ of $G_{D\tilde{u}_\ep}$ satisfies $\tilde{\th}_\ep>\th$ on $B_\de(x_0)$ for each $\ep>0$.
From Lemma \ref{Jacobiv*} (for the case of $\tilde{u}_\ep$ instead of $u$) and the arbitrary $x_0$, we deduce that
\eqref{DeLogvmlowdb*} holds in the distribution sense on $B_1$ by letting $\ep\to0$.

Let $Y=\sum_{j=1}^nY_jE_j$ be a vector field on $B_1$ defined by
$$Y_j:=\sum_{i=1}^{m}\sum_{k=1}^n\f{\la_i\th_k}{1+\la_i^2}\g_{ji}\g_{ki}\qquad \mathrm{for\ each}\ j\in\{1,\cdots,n\}.$$
From $(D^2u)\g=\g \Xi$ with $\g=(\g_{ij})$, one has
\begin{equation}\aligned\label{Yjxjululj}
&\left|\sum_{j=1}^nY_j\left(x_j+\sum_{l=1}^nu_lu_{lj}\right)\right|\le c_n|x|\mathbf{Lip}\th+\left|\sum_{i=1}^{m}\sum_{j,k,l=1}^n\f{\la_i\th_k}{1+\la_i^2}\g_{ji}\g_{ki}u_lu_{lj}\right|\\
\le&c_n\La|x|+\left|\sum_{i=1}^{m}\sum_{k,l=1}^n\f{\la_i\th_k}{1+\la_i^2}\g_{ki}u_l\g_{li}\la_i\right|\le c_n\La|x|+c_n\La|Du|,
\endaligned
\end{equation}
where $c_n$ is a general constant depending only on $n$.
Let $c_{\k,\La}\ge1$ be a general constant depending only on $n,\La,\k$, which may change from line to line.
For each $r\in(0,1/2]$, let $\Om_r=\{x\in B_r: |x|^2+|Du(x)|^2<r^2\}$.
We multiply $r^2-|\mathbf{x}|^2$ on both sides of \eqref{DeLogvmlowdb*}, then with \eqref{VolGDur} \eqref{Yjxjululj} integrating (on $L\cap \mathbf{B}_r$) by parts  gives
\begin{equation}\aligned\label{r2-x2Delogvm}
&-\int_{L\cap \mathbf{B}_r}(r^2-|\mathbf{x}|^2)\De_L\log v_m\le -\int_{L\cap \mathbf{B}_r}\f{r^2-|\mathbf{x}|^2}v\mathrm{div}_{\R^n}(Yv)+\f{\La^2}{\de_n}\int_{L\cap \mathbf{B}_r}r^2\\
\le&
c_{\k,\La}r^{n+2}-\sum_{j=1}^n\int_{\Om_r}Y_jv\p_{x_j}|\mathbf{x}|^2\le c_{\k,\La}r^{n+2}+\int_{L\cap\mathbf{B}_r}c_\La r \le c_{\k,\La}r^{n+1}.
\endaligned
\end{equation}
Let $\e_r\ge0$ be a Lipschitz function on $[0,\infty)$ with compact support in $[0,r]$ so that $\e_r\equiv1$ on $[0,r/2]$, $\e_r(s)=2(r-s)/r$ for each $s\in(r/2,r]$. 
Set $\varphi_r(x,Du(x))=\e_r(|x|)$. Using \eqref{DeLogvmlowdb*} and Lemma \ref{Intfv} (with $\phi\equiv1$) we obtain
\begin{equation}\aligned\label{varrnalogvmDeL}
\de_n\int_L\varphi_r^2|\na\log v_m|^2\le& -\int_L\f{\varphi_r^2}v\mathrm{div}_{\R^n}(Yv)+\int_L\varphi_r^2\De_L\log v_m+\f{\La^2}{\de_n}\int_L\varphi_r^2\\
\le&
c_{\k,\La}r^n+\sum_{j=1}^n\int_{\R^n}Y_jv\p_{x_j}\varphi_r^2-2\int_L\varphi_r\lan\na\varphi_r,\na\log v_m\ran\\
\le& c_{\k,\La}r^n+\int_{B_r}\f{c_\La}rv+\f{\de_n}2\int_L\varphi_r^2|\na\log v_m|^2+\f1{\de_n}\int_L|\na\varphi_r|^2,
\endaligned
\end{equation}
which implies
\begin{equation}\aligned\label{nalogv*2}
\int_{B_{r/2}}|\na\log v_m|^2v\le c_{\k,\La}r^{n-2}.
\endaligned
\end{equation}
From \eqref{MVder} and \eqref{r2-x2Delogvm}, we have
\begin{equation}\aligned\label{AppMVI}
\log v_m(\mathbf{0})\le&\f{8^n}{\omega_n}\int_{L\cap \mathbf{B}_{1/8}}\log v_m-\int_0^{1/8} \f{1}{\omega_nr^{n+1}}\int_{L\cap\mathbf{B}_r}\lan\mathbf{x},\na\log v_m+H_L\log v_m\ran dr\\
\le&\f{8^n}{\omega_n}\int_{L\cap \mathbf{B}_{1/8}}\log v_m+\int_0^{1/8} \f{\La}{\omega_n r^n}\int_{L\cap\mathbf{B}_r}\log v_m\, dr\\
&-\int_0^{1/8}  \f{1}{2\omega_nr^{n+1}}\int_{L\cap\mathbf{B}_r}(r^2-|\mathbf{x}|^2)\De_L\log v_m\, dr.
\endaligned
\end{equation}
Combining Lemma \ref{Intfv}, \eqref{r2-x2Delogvm}, \eqref{nalogv*2} and \eqref{AppMVI}, we complete the proof.
\end{proof}
\begin{remark}
Let $(u,\th)\in\mathbb{F}_n(0,\k,r)$ for some $\k\ge0$, $r>0$, $n\ge3$, and $\la_1\ge\cdots\ge\la_n$ denote eigenvalues of $D^2u$. The Hessian estimate $|D^2u(0^n)|$ has already been obtained by Wang-Yuan \cite{WdY}.
Here, we can get the estimate $|D^2u(0^n)|$ through the proof of Theorem \ref{Hesslan-2} without the condition $\la_{m}/2>\la_{m+1}\ge1$ on $B_{r}$. Because the Jacobi inequality in the form of \eqref{DeLogvmlowdb*} can be derived without Lemma \ref{divthk*} when $\th$ is a constant.
\end{remark}

Moreover, we consider another situation as follows.
\begin{lemma}\label{DeLlogvB1ge}
Let $(u,\th)\in\mathbb{F}_n(\La,\k)$ for some $\La,\k\ge0$, $n\ge2$, and $\la_1\ge\cdots\ge\la_n$ denote eigenvalues of $D^2u$.
If $\la_{n-1}\ge 2n^2/\tau$ and $\la_{n}\ge-1/\tau$ on $B_{1}$ for some $\tau\in(0,1]$, there is a constant $\de^*_{n}\in(0,1)$ depending only on $n$ so that
\begin{equation}\aligned\label{Jacobiv*+1}
\De_L\log v-\de^*_{n}|\na\log v|^2\ge\f{1}v\sum_{i,j,k=1}^n\p_{j}\left(vg^{ij}u_{ik}\th_k\right)-\f{\La^2}{\de^*_{n}}\qquad \mathrm{on}\ B_{1}.
\endaligned
\end{equation}
\end{lemma}
\begin{proof}
We consider a function $\tilde{u}(x,x_{n+1})=u(x)-\f s2x_{n+1}^2$ on $B_{1}\times\R\subset\R^{n+1}$ with a large $s>0$ so that $sI_n+D^2u>0$ on $B_{1}$. From \eqref{distincteachla}, the Hessian of $\tilde{u}(x,x_{n+1})$ has $(n+1)$ eigenvalues $\la_1(x)>\la_2(x)\cdots>\la_n(x)>-s$. Noting that the graph $G_{D\tilde{u}}\subset\R^{2(n+1)}$ is a product of $G_{Du}$ and $\R$. Hence, if we define $\tilde{h}_{ijk}$ being the components of the second fundamental form of $G_{D\tilde{u}}$ in the manner of \eqref{hijk}, then $\tilde{h}_{ij\, n+1}$ is equal to zero for each $i,j=1,\cdots,n$ with $\g_{n+1}=(0,\cdots,0,1)^T$. Therefore, from Lemma \ref{DeMla_1=} (with $m=n$ for the case of $\tilde{u}$ on $B_{1}\times\R)$, with \eqref{laijkge0} we obtain
\begin{equation}\aligned
\De_L\log v\ge&\sum_{k=1}^{n}(1+\la_k^2)h_{kkk}^2+\sum_{\stackrel{i,k=1}{i\neq k}}^{n}(3+\la_i^2+2\la_i\la_k)h_{iik}^2\\
&+\sum_{i=1}^{n}\f{\la_i}{1+\la_i^2}\th_{\g_i\g_i}-\sum_{i=1}^{n}\f{\la_i}{1+\la_i^2}\th_{\g_i}D_{\g_i}\log v.
\endaligned
\end{equation}
For $1\le k\le n$, let 
\begin{equation}\aligned\label{Tklen-1**}
T_k^*:=(1+\la_k^2)h_{kkk}^2+\sum_{i\le n,i\neq k}(3+\la_i^2+2\la_i\la_k)h_{iik}^2.
\endaligned
\end{equation} 
From $\la_{n-1}\ge 2n^2/\tau$ and $\la_{n}\ge-1/\tau$, using Cauchy-Schwartz inequality we have 
\begin{equation}\aligned
&T_k^*-\f12\sum_i\la_i^2h_{iik}^2\ge\f12\sum_{i\le n-1}\la_i\la_kh_{iik}^2+2\la_k\la_nh_{nnk}^2\\
\ge&\f12\la_k\la_{n-1}\sum_{i\le n-1}h_{iik}^2-\f{\la_k\la_{n-1}}{n^2}h_{nnk}^2
\ge \f{\la_k\la_{n-1}}{2(n-1)}\left(\sum_{i\le n-1}h_{iik}\right)^2-\f{\la_k\la_{n-1}}{n^2}h_{nnk}^2
\endaligned
\end{equation} 
for $1\le k\le n-1$. From \eqref{sumhiiki*},  we get
\begin{equation}\aligned
T_k^*-\f12\sum_{i\le n}\la_i^2h_{iik}^2\ge \f{\la_k\la_{n-1}}{2n(n-1)}\left(\f{n^2-2n+2}{n}h_{nnk}^2-n\f{2h_{nnk}\th_{\g_k}}{\sqrt{1+\la_k^2}}\right)\ge- c_*\th_{\g_k}^2,
\endaligned
\end{equation}
where $c_*$ is an absolute constant.
Clearly, $T_n^*\ge\f12\sum_i\la_i^2h_{iin}^2$.
Combining \eqref{laithgigi}\eqref{thkDeLuk} and \eqref{|nalogv*|2} (with $v_m$ replaced by $v$), we can finish the proof using Cauchy-Schwartz inequality, compared with \eqref{DeLlogv*Lipth*}.
\end{proof}
\begin{remark}
In fact, the 'distinct' condition \eqref{distincteachla} is unnecessary in Lemma \ref{DeLlogvB1ge}.
\end{remark}

Analog to the proof of Theorem \ref{Hesslan-2} (with \eqref{DeLogvmlowdb*} replaced by \eqref{Jacobiv*+1}), we immediately have the following estimate.
\begin{corollary}\label{Hessvspe}
Let $(u,\th)\in\mathbb{F}_n(\La,\k,r)$ for some $\La,\k\ge0$, $r>0$, $n\ge2$, and $\la_1\ge\cdots\ge\la_n$ denote eigenvalues of $D^2u$. If $\la_{n-1}\ge 2n^2/\tau$ and $\la_{n}\ge-1/\tau$ on $B_{r}$ for some $\tau\in(0,1]$, then there exists a constant $c(n,\k,\La,\tau)>0$ depending only on $n,\k,\La,\tau$ so that
$$|D^2u|(0^n)\le c(n,\k,\La,\tau).$$
\end{corollary}

\section{A rigidity theorem for special Lagrangian submanifolds}

Let $(u,\th)\in\mathbb{F}_n(0,\k,r)$ for some $\k\ge0$, $r>0$ and $n\ge2$. Then $\th$ is a constant, and 
$M=\{(x,Du(x))\in\R^n\times\R^n|\ x\in B_r\}$ is a special Lagrangian graph. Let $\la_1\ge\cdots\ge\la_n$ denote the eigenvalues of $D^2u$.
We define $\{e_i\}_{i=1}^n$, $\{\nu_i\}_{i=1}^n$, $h_{ijk}$ as in \eqref{einui} and \eqref{hijk}, respectively.
Let $\De_M$ denote the Laplacian of $M$, and $\na$ denote the Levi-Civita connection of $M$. 
We define $v_m$ as in \eqref{vm*} for some $m\in\{1,\cdots, n-1\}$.

Compared with Lemma \ref{Jacobiv*}, we have a stronger estimate as follows under  constraints different from one in Lemma \ref{Jacobiv*}.
\begin{lemma}\label{DeMla_1a}
Given $\a\in(0,1)$ and $\la_*>0$, there exists a constant $\la^*>\la_*$ depending on $n,\a,\la_*$ so that if $\la_m\ge\la^*$, $|\la_{m+1}|\le\la_*$ on $B_r$, then $v_m^{-\tau}$ is smooth and
\begin{equation}\aligned
\De_M v_m^{-\tau}\le(1-\a)\tau\sum_{k=1}^m\la_k^2h_{kkk}^2v_m^{-\tau}+(m\tau-\a)\tau\sum_{k=1}^n\sum_{i=1}^{m}\la_i^2h_{iik}^2v_m^{-\tau}
\endaligned
\end{equation}
for each $\tau\in(0,1)$.
In particular, $v_m^{(1-2\a)/m}$ is superharmonic on $B_r$.
\end{lemma}
\begin{proof}
For $1\le k\le n$, we define $T_k$ as in \eqref{Tklen-2*}\eqref{Tklen-2*2}.
From Lemma \ref{DeMla_1=}, we have
\begin{equation}\aligned\label{W-Y}
\De_M\log v_m\ge&\sum_{k=1}^nT_k.
\endaligned
\end{equation}
Given $\a\in(0,1)$.
For each $1\le k\le m$, with Cauchy-Schwartz inequality we have
\begin{equation}\aligned\label{Tk1lahiik21-a}
&(1-\a)\la_k^2h_{kkk}^2+2\sum_{i\le m,i\neq k}\la_i\la_kh_{iik}^2+\sum_{m<l\le n-1}\f{2\la_k^2\la_l}{\la_k-\la_l}h_{llk}^2\\
\ge&\la_k\left(\sum_{i=1}^{n-1}h_{iik}\right)^2\left(\f {1}{(1-\a)\la_k}+\sum_{i\le m,i\neq k}\f{1}{2\la_i}+\sum_{m<l\le n-1}\f{\la_k-\la_l}{2\la_k\la_l}\right)^{-1}\\
\ge&\la_kh_{nnk}^2\left(\f {1}{(1-\a)\la_m}+\f{m-1}{2\la_m}+\sum_{m<l\le n-1}\f{1}{2\la_l}\right)^{-1},
\endaligned
\end{equation}
where we have used $\sum_{i=1}^nh_{iik}=0$ in the last inequality in \eqref{Tk1lahiik21-a} since $M$ is minimal.
From \eqref{sik***}, there holds $\sum_{i=m+1}^{n}\f{1}{2\la_i}<0$ for $\la_n<0$. By an argument of contradiction, there is a constant $\de_*>0$ depending only on $n-m$ and $\la_*$ so that for $|\la_{m+1}|\le\la_*$ there holds
\begin{equation}\aligned
\sum_{i=m+1}^{n}\f{1}{2\la_i}<-\de_*.
\endaligned
\end{equation}
Hence, there is a large $\la^*>\la_*$ so that for $\la_m\ge\la^*$ we have
\begin{equation}\aligned
\f {1}{(1-\a)\la_m}+\f{m-1}{2\la_m}+\sum_{m<l\le n-1}\f{1}{2\la_l}\le-\f1{2\la_n}\le-\f{\la_k-\la_n}{2\la_k\la_n}.
\endaligned
\end{equation}
This implies
\begin{equation}\aligned\label{Tkdeh111}
&T_k-\sum_{i\le m}(1+\la_i^2)h_{iik}^2+(1-\a)\la_k^2h_{kkk}^2\\
\ge&(1-\a)\la_k^2h_{kkk}^2+2\sum_{i\le m,i\neq k}\la_i\la_kh_{iik}^2+\sum_{m<l\le n-1}\f{2\la_k^2\la_l}{\la_k-\la_l}h_{llk}^2+\f{2\la_k^2\la_n}{\la_k-\la_n}h_{nnk}^2\ge0.
\endaligned
\end{equation}
Clearly, the above inequality \eqref{Tkdeh111} holds for $\la_n\ge0$.

For $m<k\le n$, up to choosing the large $\la^*>\la_*$, from $\la_m\ge\la^*$ we have
\begin{equation}\aligned\label{Tk111}
T_k\ge\sum_{i\le m}\left(1+\f{\la_i^2(\la_i+\la_n)}{\la_i-\la_n}\right)h_{iik}^2\ge\a\sum_{i\le m}(1+\la_i^2)h_{iik}^2.
\endaligned
\end{equation}
Substituting \eqref{Tkdeh111}\eqref{Tk111} into \eqref{W-Y} gives
\begin{equation}\aligned\label{DeMlogla1}
\De_M\log v_m\ge\a\sum_{k=1}^n\sum_{i=1}^m(1+\la_i^2)h_{iik}^2-(1-\a)\sum_{k=1}^m\la_k^2h_{kkk}^2.
\endaligned
\end{equation}
For each $\tau\in(0,1)$, with \eqref{|nalogv*|2} we have
\begin{equation}\aligned
&\De_M v_m^{-\tau}=\De_M e^{-\tau\log v_m}\\
=&-\tau e^{-\tau\log v_m}\De_M\log v_m+\tau^2 e^{-\tau\log v_m}\left|\na\log v_m\right|^2\\
\le&(1-\a)\tau\sum_{k=1}^m\la_k^2h_{kkk}^2v_m^{-\tau}-\tau\a\sum_{k=1}^n\sum_{i=1}^m\la_i^2h_{iik}^2v_m^{-\tau}+\tau^2v_m^{-\tau}\sum_{k=1}^n\left(\sum_{i=1}^{m}\la_ih_{iik}\right)^2\\
\le&(1-\a)\tau\sum_{k=1}^m\la_k^2h_{kkk}^2v_m^{-\tau}+(m\tau-\a)\tau\sum_{k=1}^n\sum_{i=1}^{m}\la_i^2h_{iik}^2v_m^{-\tau}.
\endaligned
\end{equation}
This completes the proof.
\end{proof}

For each $n\ge2$, 
let $\{M_\ell\}_{\ell\ge1}$ be a sequence of Lagrangian graphs in $\mathbf{B}_1\subset\R^{2n}$ with 
phases $\ge\f{n-2}2\pi$ so that the mean curvature $H_{M_\ell}$ satisfies $\lim_{\ell\to\infty}\sup_{M_\ell}|H_{M_\ell}|=0$.
We suppose that 
there is a smooth special Lagrangian submanifold $L$ in $\mathbf{B}_1$ so that
$M_\ell\cap K$ converges to $L\cap K$ in $C^{2,1/2}$-norm for any compact set $K$ in $\mathbf{B}_1$. 
\begin{theorem}\label{Fredholm}
Either $L$ is a graph, or $L$ is flat.
\end{theorem}
\begin{proof}
For each $z\in L$,
let $\Th_1(z),\cdots,\Th_n(z)$ denote the Jordan angle between $T_zL$ and $\R^n\times\{0^n\}$ with $\Th_1(z)\ge\cdots\ge\Th_n(z)$ (see $\S$ 7.1 in \cite{X} or $\S$ 2.1 in \cite{DJX} for more details). Noting $-\pi/2\le\Th_j\le\pi/2$ for each $j$ from $M_\ell\to L$. 
Let $M$ denote the subset of $L$ containing every point $z$ so that 
$\Th_1(z)<\pi/2,$
which implies $\Th_n(z)>-\pi/2$ from the assumption of the phase of $M_\ell$ $\ge(n-2)\pi/2$.
Thus, $M$ is a special Lagrangian graph. If $M=\emptyset$, then $L$ is clearly flat. Now we assume $M\neq\emptyset$  and $L\setminus M\neq\emptyset$.
By the equation of special Lagrangian submanifold, we conclude that $L=\overline{M}$.
Set 
$$\mathcal{C}:=L\setminus M.$$
For each $z\in L$, let $\mathfrak{m}(z)$ denote the multiplicity of $\Th_1(z)$, i.e.,
$$\mathfrak{m}(z)=\{i\in\{1,\cdots,n-1\}:\ \Th_1(z)=\cdots=\Th_i(z)>\Th_{i+1}(z)\},$$ 
which is a upper semicontinuous and integer-valued function.
Hence, the infimum of the multiplicity of $\Th_1$ on $\mathcal{C}$,
\begin{equation}\aligned
m:=\inf_{z\in \mathcal{C}}\mathfrak{m}(z),
\endaligned
\end{equation}
can be attained at some point $p\in \mathcal{C}$. Moreover, there is a constant $\de_p>0$ with $\mathbf{B}_{\de_p}(p)\subset\mathbf{B}_1$ so that $\mathfrak{m}(p)=m$, $\sup_{L\cap\mathbf{B}_{\de_p}(p)}\mathfrak{m}\le m$ and $\mathfrak{m}=m$ on $\mathcal{C}\cap\mathbf{B}_{\de_p}(p)$.

From the definition of $M$, $M\cap \mathbf{B}_{\de_p}(p)$ is a graph over some subset of $\R^n$ with the graphic function $Du$ for some smooth function $u$.
Let $\la_1,\cdots,\la_n$ denote the eigenvalues of $D^2u$ satisfying $\la_1\ge\la_2\ge\cdots\ge\la_n$.
Let $\g=(\g_1,\cdots,\g_n)$ be an orthonormal matrix-valued function so that $\la_i=D^2u(\g_i,\g_i)=\lan\g_i,(D^2u) \g_i\ran$.
Then 
$$\Th_i=\arctan\la_i\qquad\mathrm{and}\qquad \cos\Th_i=(1+\la_i^2)^{-1/2}\qquad \mathrm{on}\ M\cap \mathbf{B}_{\de_p}(p).$$
Let $\na$ denote the Levi-Civita connection of $L$, then we have
\begin{equation}\aligned\label{Deraiam}
\na\Th_i=\sum_{1\le j\le n} h_{iij}e_j\qquad\qquad \mathcal{H}^n-a.e.\ \mathrm{on}\ L\cap \mathbf{B}_{\de_p}(p),
\endaligned
\end{equation}
where $\{e_j\}$ is orthonormal defined in \eqref{einui}, and $h_{ijk}$ is defined in \eqref{hijk}.
In fact, at each differentiable point $z$ of $\Th_i$, from \eqref{einui}\eqref{hijk}\eqref{pjlam*} one has
\begin{equation*}\aligned
\na\Th_i=\sum_{j=1}^n(\na_{e_j}\Th_i) e_j=\sum_{j=1}^n\cos\Th_j(D_{\g_j}\Th_i) e_j=\sum_{j=1}^n\cos\Th_j\f{D^3u(\g_i,\g_i,\g_j)}{1+\la_i^2} e_j=\sum_{j=1}^n h_{iij}e_j.
\endaligned
\end{equation*}

From the definition of $m$, the function $\Th_*:=\sum_{i\le m}\Th_i$ is smooth on $L\cap \mathbf{B}_{\de_p}(p)$ with $\Th_*\equiv\f m2\pi$ on $\mathcal{C}\cap \mathbf{B}_{\de_p}(p)$. Hence, $|\na\Th_*|=0$ at each $z\in\mathcal{C}\cap \mathbf{B}_{\de_p}(p)$, which implies
\begin{equation}\aligned
0\ge\Th_*(z')-\Th_*(z)\ge-\psi_\Th(|z-z'|)|z-z'|\qquad\mathrm{for \ each}\ z'\in L\cap \mathbf{B}_{\de_p}(p).
\endaligned
\end{equation}
Here, $\psi_\Th$ is a function depending on $\Th_1,\cdots,\Th_n$ satisfying $\lim_{s\to0}\psi_\Th(s)=0$. Since $\Th_i(z')\le\Th_i(z)=\f12\pi$ on $L\cap \mathbf{B}_{\de_p}(p)$ for each $i=1,\cdots,m$, we get
\begin{equation}\aligned
0\ge\Th_i(z')-\Th_i(z)\ge\Th_*(z')-\Th_*(z)\ge-\psi_\Th(|z-z'|)|z-z'|.
\endaligned
\end{equation}
This gives
\begin{equation}\aligned
|\na\Th_i(z)|=0\qquad \mathrm{for\ each} \ i=1,\cdots,m.
\endaligned
\end{equation}
With \eqref{Deraiam}, we get
\begin{equation}\aligned\label{hiijx=0}
h_{iij}(z)=0\qquad \mathrm{for\ each} \ i\in\{1,\cdots,m\},\ j\in\{1,\cdots,n\}.
\endaligned
\end{equation}

We define a function
\begin{equation}\aligned
\z=\prod_{1\le i\le m}(\cos\Th_i)^{1/m},
\endaligned
\end{equation}
which is a Lipschitz function on $L\cap \mathbf{B}_{\de_p}(p)$. For $z_0\in\mathcal{C}\cap \mathbf{B}_{\de_p}(p)$, from Newton-Leibniz formula one has
$$\cos\Th_i(z)=\cos\Th_i(z)-\cos\f\pi2\le\f\pi2-\Th_i(z)=\Th_i(z_0)-\Th_i(z)\le\psi_\Th(|z-z_0|)|z-z_0|$$
for each $i\in\{1,\cdots,m\}$. So we have
\begin{equation}\aligned
\z(z)\le\prod_{1\le i\le m}(\psi_\Th(|z-z_0|)|z-z_0|)^{1/m}=\psi_\Th(|z-z_0|)|z-z_0|.
\endaligned
\end{equation}
This implies
\begin{equation}\aligned\label{naVz=0}
|\na\z|=0\qquad\qquad \mathrm{on} \ \mathcal{C}\cap \mathbf{B}_{\de_p}(p).
\endaligned
\end{equation}
If
\begin{equation}\aligned\label{TpL0n}
T_pL=\{(p',y)\in\R^n\times\R^n:\, y=\R^n\}\qquad\mathrm{with}\ p=(p',p''),
\endaligned
\end{equation}
then 
there exist a constant $0<r_*<\de_p$ and a smooth function $\mathbf{w}$ on $B_{r_*}(p'')\subset T_pL$ so that $G_{D\mathbf{w}}\cap \mathbf{B}_{r_*}(p)\subset L$, $D\mathbf{w}(p'')=p'$, $|D^2\mathbf{w}(p'')|=0$ and $|D^3\mathbf{w}|\le1$ on $B_{r_*}(p'')$. From \eqref{DeSiwnn}, 
\begin{equation}\aligned\label{bfwnn*}
\De_L\mathbf{w}_{nn}=-\sum_{i,j=1}^n\left(\p_{x_n}\hat{g}^{ij}\right)\mathbf{w}_{ijn}=2\sum_{i,j,k,k',l=1}^n\hat{g}^{ik}\mathbf{w}_{kk'}\mathbf{w}_{k'ln}\hat{g}^{lj}\mathbf{w}_{ijn},
\endaligned
\end{equation}
where $(\hat{g}^{ij})$ is the inverse matrix of $I_n+(D^2\mathbf{w})^2$. We also regard $\mathbf{w}_{nn}$ as a function on $L\cap \mathbf{B}_{r_*}(p)$.
From Lemma \ref{wnnge0} and Theorem \ref{Harnack}, there is a constant $c>0$ depending only on $n$ so that
\begin{equation}\aligned
\sup_{L\cap\mathbf{B}_r(p)} (\mathbf{w}_{nn}+\ep)\le c\inf_{L\cap\mathbf{B}_r(p)} (\mathbf{w}_{nn}+\ep)\qquad\qquad \mathrm{for\ each}\ 0<r<r_*/8,\ 0<\ep<1.
\endaligned
\end{equation}
Letting $\ep\to0$ implies $\mathbf{w}_{nn}\equiv0$ on $L\cap \mathbf{B}_{r_*/8}(p)$ as $|D^2\mathbf{w}(p'')|=0$. From Lemma \ref{wnnge0} again, $\z\equiv0$ on $L\cap \mathbf{B}_{r_*/8}(p)$.
With the equation of minimal submanifolds, we finish the proof in the case of \eqref{TpL0n}.

Now we assume
\begin{equation}\aligned\label{ClaimTpL}
T_pL\neq\{(p',y)\in\R^n\times\R^n:\, y=\R^n\}.
\endaligned
\end{equation}
In this case, $\z$ is not identically equal to zero on $L\cap \mathbf{B}_{\de_p}(p)$.
Noting $\z^m$ is smooth. Put $\Phi=\z^m\circ \exp_{p}$, where $\exp_{p}$ is the exponential mapping w.r.t. $p$ (with values in $L$). 
From Taylor's expansion for $\Phi(x)$ at the origin in $T_pL\cong\R^n$ with $x=(x_1\cdots,x_n)$, we write
\begin{equation}\aligned
\Phi(x)=\sum_{j=0}^{m'}\sum_{\be_1+\cdots+\be_n=j,\be_i\ge0}\f{\p_{x_1}^{\be_1}\cdots\p_{x_n}^{\be_n}\Phi(0^n)}{j!}x_1^{\be_1}\cdots x_n^{\be_n}+\mathcal{E}_\Phi(x),
\endaligned
\end{equation}
where
\begin{equation}\aligned\label{epPhix}
\sum_{i=0}^2|x|^i|D^i\mathcal{E}_\Phi(x)|\le c'|x|^{m'+1}\qquad \mathrm{on}\ \exp^{-1}_{p}(L\cap\mathbf{B}_{\de_p}(p)),
\endaligned
\end{equation}
and $c'$ is a constant depending only on $m',n$. 
Denote $\p^\be\Phi=\p_{x_1}^{\be_1}\cdots\p_{x_n}^{\be_n}\Phi$ and $x^\be=x_1^{\be_1}\cdots x_n^{\be_n}$ for $\be=(\be_1,\cdots,\be_n)$ with $\be_i\ge0$. Put $|\be|_1=\be_1+\cdots+\be_n$. 
From \eqref{ClaimTpL}, we may assume $\p^\be\Phi(0^n)=0$ for all $|\be|_1\le m'-1$, and $\p^{\be'}\Phi(0^n)\neq0$ for some $\be'$ with $|\be'|_1=m'$.
From \eqref{naVz=0}, we get $m'\ge m+1$. Let
\begin{equation}\aligned
P(x)=\sum_{|\be|_1=m'}\f{\p^\be\Phi(0^n)}{m'!}x^\be=\sum_{\be_1+\cdots+\be_n=m',\be_i\ge0}\f{\p_{x_1}^{\be_1}\cdots\p_{x_n}^{\be_n}\Phi(0^n)}{m'!}x_1^{\be_1}\cdots x_n^{\be_n}.
\endaligned
\end{equation}
From $\Phi\ge0$ and \eqref{epPhix}, for any $\ep>0$ there is a constant $\r_\ep>0$ so that $P(x)\ge-\ep|x|^{m'}$ for all $|x|<\r_\ep$.
Since $P$ is homogeneous, letting $\ep\to0$ implies $P\ge0$ on $\R^n$, and then $m'$ is even. 
Up to choosing the small $\de_p>0$, we get 
$$\De_L \z^{\a'}\le0\qquad \mathrm{on}\ M\cap \mathbf{B}_{\de_p}(p)$$ 
from Lemma \ref{DeMla_1a} with $\a'=\f{m}{m+1/2}$. Set $\a_m=\a'/m$. This implies
\begin{equation}\aligned
\z^m\De_L\z^m-\left(1-\a_m\right)|\na^L\z^m|^2\le0\qquad\mathrm{on}\ M\cap \mathbf{B}_{\de_p}(p).
\endaligned
\end{equation}
Thus, for any $\de>0$, there is a small constant $r_\de>0$ so that 
\begin{equation}\aligned\label{PhiDePde}
\Phi(x)\De_{\R^n}\Phi(x)-\left(1-\a_m\right)|D\Phi|^2(x)\le\de|x|^{2m'-2}\qquad\mathrm{on}\ B_{r_\de}.
\endaligned
\end{equation}
Up to choosing the constant $r_\de>0$, from \eqref{epPhix} and \eqref{PhiDePde} we have
\begin{equation}\aligned\label{PDePde}
P(x)\De_{\R^n}P(x)-\left(1-\a_m\right)|DP|^2(x)\le2\de|x|^{2m'-2}\qquad\mathrm{on}\ B_{r_\de}.
\endaligned
\end{equation}
Since $P,\De_{\R^n}P,|DP|^2$ are homogeneous, from \eqref{PDePde} it follows that
\begin{equation}\aligned
P(x)\De_{\R^n}P(x)-\left(1-\a_m\right)|DP|^2(x)\le0\qquad\mathrm{on}\ \R^n.
\endaligned
\end{equation}
This infers 
\begin{equation}\aligned\label{DePam*}
\De_{\R^n} P^{\a_m}\le0\qquad\mathrm{on}\ \R^n\setminus\mathcal{C}_{P}.
\endaligned
\end{equation}
Here, $\mathcal{C}_P$ denotes the nodal set of $P$, i.e., $\mathcal{C}_P=\{x\in \R^n:\, P(x)=0\}$. 

Let $B_r(\mathcal{C}_{P})$ denote the $r$-neighborhood of $\mathcal{C}_{P}$ defined by $\{x\in\R^n:\, d(x,\mathcal{C}_{P})<r\}$.
Let $\e_\ep$ be a Lipschitz function defined by $\e_\ep\equiv1$ on $\R^n\setminus B_\ep(\mathcal{C}_{P})$, $\e_\ep=\f2\ep d(x,\mathcal{C}_{P})-1$ on $B_\ep(\mathcal{C}_{P})\setminus B_{\ep/2}(\mathcal{C}_{P})$ and $\e_\ep\equiv0$ on $B_{\ep/2}(\mathcal{C}_{P})$.
Since $m'\a_m>1$ and $\mathcal{C}_P$ has Hausdorff dimension $\le n-1$, for any $R>0$ we get
\begin{equation}\aligned\label{MB3neep}
\lim_{\ep\to0}\int_{B_R} |DP^{\a_m}|\cdot|D \e_\ep| =0.
\endaligned
\end{equation}
Let $\varphi$ be a nonnegative Lipschitz function with compact support in $\R^n$. 
Integrating by parts infers
\begin{equation}\aligned
0\le&-\int_{\R^n}\varphi\e_\ep\De_{\R^n} P^{\a_m} =\int_{\R^n}\lan D(\varphi\e_\ep),D P^{\a_m}\ran \\
=&\int_{\R^n}\e_\ep\lan D\varphi,D P^{\a_m}\ran +\int_{\R^n}\varphi\lan D\e_\ep,D P^{\a_m}\ran .
\endaligned
\end{equation}
Combining \eqref{MB3neep}, letting $\ep\to0$ in the above inequality infers
\begin{equation}\aligned
0\le\int_{\R^n}\lan D\varphi,D P^{\a_m}\ran,
\endaligned
\end{equation}
which means that $P^{\a_m}$ is superharmonic on $\R^n$ in the distribution sense.
Hence, the mean value inequality holds for $P^{\a_m}$, which implies $P^{\a_m}\equiv0$ as $P^{\a_m}\ge0$ and $P^{\a_m}(0^n)=0$. 
It contradicts to the definition of $P$. We complete the proof.
\end{proof}
\begin{remark}
The reason of considering the polynominal $P$ instead of $\xi^m$ is that the nodal set of $P$ is simple, while the nodal set of a smooth function may be very complicated.
\end{remark}

\section{Hessian estimates via geometric measure theory}

Let $\tau_{*}$ is the constant in Lemma \ref{AllardregThm} with $N=2n$ there.
\begin{lemma}\label{Key}
Let $\{(u_\ell,\th_\ell)\}_{\ell\ge1}$ be a sequence in $\mathbb{F}_n(\La,\k,r_\ell)$ for some $\k,\La\ge0$ and $n\ge2$ so that $Du_\ell(0^n)=0^n$, $\lim_{\ell\to\infty}|D^2u_\ell(0^n)|=\infty$ and $\lim_{\ell\to\infty}r_\ell\to\infty$. If $\mathcal{H}^n(G_{Du_\ell}\cap \mathbf{B}_1)\le(1+\tau_{*})\omega_n$
and the current $[|G_{Du_\ell}|]$ converges as $\ell\to\infty$ to an integral current $T$ in $\R^{2n}$,
then $\mathrm{spt}T$ is flat and $T$ has multiplicity one on $\mathrm{spt}T$.
\end{lemma}
\begin{proof}
Denote $M_\ell=G_{Du_\ell}$ for each $\ell$.
Let $[|M_\ell|]$ denote the current associated with $M_\ell$ with the orientation $\vec{\tau}_\ell=e_1^\ell\wedge\cdots\wedge e_n^\ell$, where $\{e_1^\ell,\cdots,e_n^\ell\}$ is an orthonormal tangent basis on $M_\ell$. 
Let $\omega_\ell$ be the $n$-form defined on $B_{r_\ell}\times\R^n$ by
$$\omega_\ell=\mathrm{Re}(e^{-\sqrt{-1}\th_\ell}dz_1\wedge\cdots\wedge dz_n).$$
Since $G_{Du_\ell}$ is Lagrangian, from \eqref{PhaseThDEF} (up to a choice of the orientation $\vec{\tau}_\ell$) we have
\begin{equation}\aligned
\,[|M_\ell|](\omega_\ell)=\int_{M_\ell}\lan\omega_\ell,\vec{\tau}_\ell\ran=\mathcal{H}^n(M_\ell).
\endaligned
\end{equation}
Let $\mathbf{W}$ be an integer multiplicity one $(n+1)$-current in $\R^{2n}$ with compact support, and 
$W_\ell=[|M_\ell|]+\p \mathbf{W}$ for each $\ell$.
Then we have
\begin{equation}\aligned
\mathbf{W}(d\omega_\ell)=\p\mathbf{W}(\omega_\ell)=W_\ell(\omega_\ell)-[|M_\ell|](\omega_\ell)=W_\ell(\omega_\ell)-\mathcal{H}^n(M_\ell),
\endaligned
\end{equation}
which implies (with \eqref{PhaseThDEF})
\begin{equation}\aligned
&\mathcal{H}^n(M_\ell))\le\mathbf{M}(W_\ell)-\mathbf{W}(d\omega_\ell)\\
\le&\mathbf{M}(W_\ell)+\sup_{B_{r_\ell}}|D\th_\ell|\mathbf{M}(\mathbf{W})\le\mathbf{M}(W_\ell)+\f{\La}{r_\ell}\mathbf{M}(\mathbf{W}).
\endaligned
\end{equation}
From \eqref{almostmonoto} and \eqref{VolGDur}, $\r^{-n}\mathbf{M}(T\llcorner\mathbf{B}_\r)$ is uniformly bounded on $\r$.
Since $[|M_\ell|]$ converges as $\ell\to\infty$ to $T$, it's clear that $W_\ell$ converges as $\ell\to\infty$ to $T+\p \mathbf{W}$. Letting $\ell\to\infty$ in the above inequality gives
\begin{equation}\aligned
\mathbf{M}(T)\le\mathbf{M}(T+\p \mathbf{W}),
\endaligned
\end{equation}
which means that $T$ is minimizing in $\R^{2n}$ with $\p T=0$. In particular, if we write $T=\vth[|S|]$ with the multiplicity $\vartheta$ and $S=\mathrm{spt}T$, then the varifold $\vth|S|$ is stationary clearly (from the first variational formula of varifolds).
The condition $\mathcal{H}^n(G_{Du_\ell}\cap \mathbf{B}_1)\le(1+\tau_{*})\omega_n$ implies
\begin{equation}\aligned\label{MassTB1}
\mathbf{M}(T\llcorner\overline{\mathbf{B}}_1)\le(1+\tau_{*})\omega_n.
\endaligned
\end{equation}
By Allard's regularity theorem (see Lemma \ref{AllardregThm}), $S\cap\mathbf{B}_{\varrho_*}$ is smooth, and
\begin{equation}\aligned
\sup_{\mathbf{x},\mathbf{y}\in M_\ell\cap\mathbf{B}_{\varrho_*}(\mathbf{x}),\mathbf{x}\neq\mathbf{y}}|\mathbf{x}-\mathbf{y}|^{-1/2}d_{\mathbf{G}}(T_{\mathbf{x}}M_\ell,T_{\mathbf{y}}M_\ell)\le c_{n}.
\endaligned
\end{equation}
Here, $c_{n}$ is a positive constant depending only on $n$, and $d_{\mathbf{G}}$ is the distance function in $\mathbf{G}_{n,n}$ defined in $\S$4.
Then $M_\ell\cap\mathbf{B}_{\varrho_*/2}$ converges to $S\cap\mathbf{B}_{\varrho_*/2}$ smoothly since $T\llcorner\mathbf{B}_{\varrho_*/2}$ has multiplicity one from \eqref{MassTB1}.
From Theorem \ref{Fredholm} and $\lim_{\ell\to\infty}|D^2u_\ell(0^n)|=\infty$, we deduce that $S\cap\mathbf{B}_{\varrho_*/2}$ is flat.

Since the minimizing current $T$ has singularities of dimension $\le n-2$ from \cite{Am} or Theorem 1.1 in \cite{ds1},
the regular part $\mathcal{R}(T)$ of $T$ is (relatively) open and connected. 
For a considered regular point $\mathbf{y}$ of $T$,
let $P_{\mathbf{y}}$ denote the tangent space of $S$ at $\mathbf{y}$, then the projection $\Pi:\,S\cap\mathbf{B}_\de(\mathbf{y})\to \Pi(S\cap\mathbf{B}_\de(\mathbf{y}))\subset P_{\mathbf{y}}$ is bijective for the small $\de>0$.
By \eqref{PushforwfT} and Constancy Theorem (see 26.27 in \cite{S} or 7.1.9 in \cite{LYa}), we conclude that  the multiplicity function of $T$ is a constant,
and equal to one on $S$ from \eqref{MassTB1}. Then the flatness of $S\cap\mathbf{B}_{\varrho_*/2}$ implies that $S$ is flat everywhere.
\end{proof}

Now we use Lemma \ref{AllardregThm} and Lemma \ref{Key} to show the following H\"older estimates on tangent planes of Lagrangian graphs.
\begin{theorem}\label{CaTanPlane}
Let $(u,\th)\in\mathbb{F}_n(\La,\k,2r)$ for some $\La,\k\ge0$, $n\ge2$ and $r>0$. Put $L=G_{Du}$.
For any $\a\in(0,1)$ and $x\in B_{r/2}$, there is a constant $c_{n,\a,\La,\k}>0$ depending only on $n,\a,\La$ and $\k$ so that 
\begin{equation}\aligned
r^\a\sup_{\mathbf{x},\mathbf{y}\in L\cap\mathbf{B}_{r}(\mathbf{x}),\mathbf{x}\neq\mathbf{y}}|\mathbf{x}-\mathbf{y}|^{-\a}d_{\mathbf{G}}(T_{\mathbf{x}}L,T_{\mathbf{y}}L)\le c_{n,\a,\La,\k}.
\endaligned
\end{equation}
\end{theorem}
\begin{proof}
Without loss of generality, we can assume $r=1$ by scaling. According to Allard's regularity theorem, we only need to prove that there is a constant $\de_{n,\La,\k}>0$ so that
\begin{equation}\aligned\label{Volsmall}
\mathcal{H}^n(G_{Du}\cap \mathbf{B}_{\de_{n,\La,\k}}(\mathbf{x}))\le (1+\tau_{*})\omega_n\de_{n,\La,\k}^n\qquad\mathrm{for\ any}\ \mathbf{x}\in G_{Du}\cap (B_1\times\R^n).
\endaligned
\end{equation}
Let us prove \eqref{Volsmall} by contradiction.
We suppose that there are a sequence of positive numbers $\de_\ell\to0$, a sequence of points $\mathbf{x}_\ell\in B_1\times\R^n$ and a sequence of $\{(u_\ell,\th_\ell)\}_{\ell\ge1}$ in $\mathbb{F}_n(\La,\k,2)$ with $\mathbf{x}_\ell=(x_\ell,Du_\ell(x_\ell))\in G_{D u_\ell}$
such that 
\begin{equation}\aligned\label{Contraxellu}
\mathcal{H}^n(G_{D u_\ell}\cap \mathbf{B}_{\de_\ell}(\mathbf{x}_\ell))>(1+\tau_{*})\omega_n\de_\ell^n.
\endaligned
\end{equation}
Then there is a sequence of points $\mathbf{p}_\ell=(p_\ell,Du_\ell(p_\ell))$ for some $p_\ell\in B_1$ so that $|\mathbf{x}_\ell-\mathbf{p}_\ell|\to0$ and
\begin{equation}\aligned
\lim_{\ell\to\infty}|D^2u_\ell|(p_\ell)=\infty.
\endaligned
\end{equation}
Or else, we have uniform $C^{2,1/2}$-estimates for $u_\ell$ on $B_{s_0}(x_\ell)$ for some $s_0>0$ independent of $\ell$, which violates \eqref{Contraxellu} and $\de_\ell\to0$.

We claim that there is a sequence of positive numbers $\vep_\ell\to0$ so that
\begin{equation}\aligned\label{volGDui}
\mathcal{H}^n(G_{D u_\ell}\cap \mathbf{B}_{\vep_\ell}(\mathbf{p}_\ell))=(1+\tau_{*}/2)\omega_n\vep_\ell^n\qquad\mathrm{for\ each}\ \ell.
\endaligned
\end{equation}
If the claim \eqref{volGDui} fails, then we can assume $\vep_\ell\ge2\vep_0$ for some constant $\vep_0>0$ up to a choice of  a subsequence, and
\begin{equation}\aligned\label{GDBrpelltau*2}
\mathcal{H}^n(G_{D u_\ell}\cap \mathbf{B}_{r}(\mathbf{p}_\ell))<(1+\tau_{*}/2)\omega_nr^n\qquad\mathrm{for\ each}\ \ell,\ \mathrm{and}\ r\in(0,\vep_\ell].
\endaligned
\end{equation}
Up to a choice of the smaller $\vep_0$,
we can further assume
\begin{equation}\aligned\label{tauvep0}
e^{\La\vep_0}\le\left(1+\f{\tau_*}2\right)^{-1}\left(1+\f{2\tau_*}3\right).
\endaligned
\end{equation}
Set $s_\ell=|\mathbf{x}_\ell-\mathbf{p}_\ell|$.
Combining \eqref{almostmonoto}, \eqref{Contraxellu}, \eqref{GDBrpelltau*2} and \eqref{tauvep0}, for every large $\ell$
\begin{equation}\aligned
e^{\La\de_\ell}(1+\tau_*)<&\f{e^{\La\de_\ell}}{\omega_n\de_\ell^n}\mathcal{H}^n(G_{D u_\ell}\cap \mathbf{B}_{\de_\ell}(\mathbf{x}_\ell))\le\f{e^{\La\vep_0}}{\omega_n\vep_0^n}\mathcal{H}^n(G_{D u_\ell}\cap \mathbf{B}_{\vep_0}(\mathbf{x}_\ell))\\
\le&\f{e^{\La\vep_0}}{\omega_n\vep_0^n}\mathcal{H}^n(G_{D u_\ell}\cap \mathbf{B}_{\vep_0+s_\ell}(\mathbf{p}_\ell))\le\f{e^{\La\vep_0}}{\omega_n\vep_0^n}(1+\tau_{*}/2)\omega_n(\vep_0+s_\ell)^n\\
=&\left(1+\f{\tau_*}2\right)e^{\La\vep_0}\left(1+\f{s_\ell}{\vep_0}\right)^n\le\left(1+\f{2\tau_*}3\right)\left(1+\f{s_\ell}{\vep_0}\right)^n.
\endaligned
\end{equation}
This above inequality is impossible for large $\ell$. Hence, we have proven the claim \eqref{volGDui}.

From \eqref{VolGDur}, there is a constant $c_n>0$ such that 
\begin{equation}\aligned
\mathcal{H}^n\left(G_{Du_\ell}\cap\mathbf{B}_{1/2}(\mathbf{p}_\ell)\right)\le c_n\left(1+\La (1+\k)\right)(1+\k)^n.
\endaligned
\end{equation}
By \eqref{almostmonoto}, there is a constant $c_{n,\La,\k}$ depending only on $n,\La,\k$ so that
\begin{equation}\aligned
\mathcal{H}^n\left(G_{Du_\ell}\cap\mathbf{B}_{r}(\mathbf{p}_\ell)\right)\le c_{n,\La,\k}r^n \qquad\mathrm{for\ any}\ r\in(0,1/2].
\endaligned
\end{equation}
Let $\tilde{u}_\ell(x)=\vep_\ell^{-2}u_\ell(\vep_\ell x+p_\ell)-\vep_\ell^{-1} x\cdot Du_\ell(p_\ell)$, and $\tilde{\th}_\ell(x)=\th_\ell(\vep_\ell x+p_\ell)$.
Then $\tilde{u}_\ell$ satisfies $|D\tilde{u}_\ell(0^n)|=0$ and
$$\mathrm{tr}(\arctan D^2\tilde{u}_\ell)=\tilde{\th}_\ell\qquad \mathrm{on} \ B_{(2\vep_\ell)^{-1}}.$$ 
Hence, $|D\tilde{u}_\ell|\le2\vep_\ell^{-1}\sup_{B_2}|Du_\ell|\le4\k\vep_\ell^{-1}$, and
$(\tilde{u}_\ell,\tilde{\th}_\ell)\in\mathbb{F}_n(\La/4,8\k,(2\vep_\ell)^{-1})$.
Denote $L_\ell=G_{D\tilde{u}_\ell}\cap(B_{(2\vep_\ell)^{-1}}\times \R^n)$. From \eqref{volGDui}, it follows that
\begin{equation}\aligned\label{volGDtui}
\mathcal{H}^n(L_\ell\cap \mathbf{B}_1)=(1+\tau_{*}/2)\omega_n.
\endaligned
\end{equation}
From Federer-Fleming's compactness theorem, there is an integral current $T_\infty$ in $\R^{2n}$ so that 
$[|L_\ell|]\llcorner U$ converges as $\ell\to\infty$ to $T_\infty\llcorner U$ for any bounded open set $U\subset\R^{2n}$.
From Lemma \ref{Key}, spt$T_\infty$ is flat and $T_\infty$ has multiplicity one on $\mathrm{spt}T_\infty$. As a consequence, $\mathcal{H}^n(L_\ell  \cap \mathbf{B}_1)$ converges as $\ell\to\infty$ to $\mathbf{M}(T_\infty\llcorner\mathbf{B}_1)=\omega_n$. This contradicts to \eqref{volGDtui}. Hence, \eqref{Volsmall} is true. We complete the proof.
\end{proof}

\begin{lemma}\label{blalul}
Let $\{(u_\ell,\th_\ell)\}_{\ell\ge1}$ be an infinite sequence in $\mathbb{F}_n(\La,\k)$ for some $\La,\k\ge0$ and $n\ge2$.
Let $\overline{\la}_{\ell},\underline{\la}_{\ell}$ denote the maximal and minimal eigenvalues of $D^2u_\ell$. If $\lim_{\ell\to\infty}\overline{\la}_{\ell}(0^n)=\infty$, then $\lim_{\ell\to\infty}\underline{\la}_{\ell}(x_\ell)=-\infty$ for some sequence $x_\ell\to0^n$.
\end{lemma}
\begin{proof}
Let us prove it by contradiction.
We suppose that there is a sequence $\{(u_\ell,\th_\ell)\}_{\ell\ge1}$ in $\mathbb{F}_n(\La,\k)$ with $\lim_{\ell\to\infty}\overline{\la}_{\ell}(0^n)=\infty$ and $\lim_{\ell\to\infty}\inf_{B_{\de_0}}\underline{\la}_{\ell}(0^n)\ge-1/\tau$ for some $\tau\in(0,1]$ and a small $\de_0>0$. 
We claim that
\begin{center}
for any $\ep>0$ there exists a constant $\de>0$ so that
$\liminf_{\ell\to\infty}\inf_{B_\de}\overline{\la}_{\ell}\ge1/\ep.$
\end{center}
Or else, there are a constant $\ep_0\in(0,1)$ and a sequence of points $x_\ell\to0$ so that 
$$\liminf_{\ell\to\infty}\overline{\la}_{\ell}(x_\ell)\le1/\ep_0.$$
Set $\mathbf{x}_\ell=(x_\ell,Du_\ell(x_\ell))$ and $\Si_\ell=G_{Du_\ell}$. 
From Theorem \ref{CaTanPlane},
there is a constant $r_0>0$ so that $\Si_\ell\cap\mathbf{B}_{r_0}(\mathbf{x}_\ell)$ can be written as a graph over some subset of $T_{\mathbf{x}_\ell}\Si_\ell$ with slope $\le \ep_0/2$.
Noting the formula 
$$\tan(\a_1+\a_2)=\f{\tan\a_1+\tan\a_2}{1-\tan\a_1\tan\a_2}\qquad \mathrm{for\ each}\ 0\le\a_1\le\a_2\le \a_1+\a_2<\pi/2.$$
Hence, for any $(x,Du_\ell(x))\in\mathbf{B}_{r_0}(\mathbf{x}_\ell)$,
\begin{equation}\aligned
|\underline{\la}_\ell(x)|\le\overline{\la}_\ell(x)\le\f{\f1{\ep_0}+\f{\ep_0}2}{1-\f1{\ep_0}\f{\ep_0}2}=2\left(\f1{\ep_0}+\f{\ep_0}2\right).
\endaligned
\end{equation}
In other words, $|D^2u_\ell(x)|$ is uniformly bounded whenever $(x,Du_\ell(x))\in\mathbf{B}_{r_0}(\mathbf{x}_\ell)$. Hence, there is a positive constant $r_0'<r_0$ (independent of $\ell$) so that
$$\Si_\ell\cap\{(x,y):\, |x-x_\ell|<r_0',y\in\R^n\}\subset \mathbf{B}_{r_0}(\mathbf{x}_\ell).$$
This implies $(0^n,Du_\ell(0^n))\in\mathbf{B}_{r_0}(\mathbf{x}_\ell)$ for all the large $\ell$ as $x_\ell\to0$. However, this violates $\lim_{\ell\to\infty}\overline{\la}_{\ell}(0^n)=\infty$. Hence, the above claim is true.

Let $\la_{1,\ell},\cdots,\la_{n,\ell}$ denote the eigenvalues of $D^2u_\ell$ with $\la_{1,\ell}\ge\cdots\ge\la_{n,\ell}$. From Theorem \ref{CaTanPlane}, there is a suitably small $\de\in(0,\de_0)$ so that (at least) one of following occurs:
\begin{itemize}
  \item[i)] for each large $\ell$ there is a $1\le m\le n-2$ so that
$\la_{m,\ell}>2\la_{m+1,\ell}\ge1$ on $B_\de$;
  \item[ii)] $\la_{n-1,\ell}\ge2n^2/\tau$ on $B_\de$ for all the large $\ell$.
\end{itemize}
In fact, if i) falis, then combining the above claim and Theorem \ref{CaTanPlane}, all $(n-1)$-eigenvalues $\la_{1,\ell},\cdots,\la_{n-1,\ell}$ should be large on $B_\de$ for the small $\de>0$ and all the large $\ell$, which means that ii) holds.
From Theorem \ref{Hesslan-2} or Corollary \ref{Hessvspe}, it follows that $\limsup_{\ell\to\infty}\overline{\la}_{\ell}(0^n)<\infty$. It's a contradiction. We complete the proof.
\end{proof}

Now, it's the position to prove Theorem \ref{main}.
\begin{proof}
Let us prove it by contradiction.
Suppose Theorem \ref{main} fails. There are an infinite sequence $\{(u_\ell,\th_\ell)\}_{\ell\ge1}$ in $\mathbb{F}_n(\La,\k)$ and $\{x'_\ell\}\subset B_{1/2}$ so that
$$\lim_{\ell\to\infty}\overline{\la}_{\ell}(x'_\ell)=\infty,$$
where $\overline{\la}_{\ell}$ denotes the maximal eigenvalue of $D^2u_\ell$. 
From Lemma \ref{blalul}, there is a sequence $\{x_\ell\}\subset B_{1}$ with $\limsup_{\ell\to\infty}|x_\ell|\le1/2$ so that
$$\lim_{\ell\to\infty}\underline{\la}_{\ell}(x_\ell)=-\infty,$$
where $\underline{\la}_{\ell}$ denotes the minimal eigenvalues of $D^2u_\ell$.
Put $\mathbf{x}_\ell=(x_\ell,Du_\ell(x_\ell))$ and $L_\ell=G_{Du_\ell}$.

Without loss of generality, we suppose that 
$\mathbf{x}_\ell$ converges to a point $\mathbf{x}_*=(x_*,y_*)\in\overline{B}_{1/2}\times\overline{B}_{\k}$.
Given $\a\in(0,1)$.
From Theorem \ref{CaTanPlane}, up to a choice of the subsequence, there is a $C^{1,\a}$ submanifold $L$ in $\mathbf{B}_{1/4}(\mathbf{x}_*)$ with $\mathbf{x}_*\in L$ so that $L_\ell\cap\mathbf{B}_{1/4}(\mathbf{x}_\ell)$ converges to $L$ in $C^{1,\a}$-sense and
$$T_{\mathbf{x}_*}L=\{(x,y)\in\R^n\times\R^n:\, x=x_*\}.$$
Let 
$$L^*=\{\mathbf{x}\in \R^{2n}:\, \mathbf{x}+\mathbf{x}_*\in L\},$$
then $T_{\mathbf{0}}L^*=\{0^n\}\times\R^n$, and $\mathfrak{R}_{\be^*}(T_{\mathbf{0}}L^*)=\R^n\times\{0^n\}$ with $\be^*=(\pi/2,\cdots,\pi/2,-\pi/2)\in\R^n$ and $\mathfrak{R}_{\be^*}$ defined in \eqref{Rbe**}.
Put 
$$\Si=\mathfrak{R}_{\be^*}(L^*),\qquad\mathrm{and}\qquad L^*_\ell=\{\mathbf{x}\in \R^{2n}:\, \mathbf{x}+\mathbf{x}_\ell\in L_\ell\}.$$
From Theorem \ref{CaTanPlane} again, there is a small constant $\de\in(0,1/8)$ depending only on $n,\k,\La$ so that $\Si\cap\mathbf{B}_{2\de}$ is a connected Lagrangian graph over some open set $W$ of $\R^n$ with the graphic function $Dw$ and $B_\de\subset W$ such that $|w(0^n)|=|Dw(0^n)|=|D^2w(0^n)|=0$, $|D^2w|<1$ on $\overline{B}_\de$ and \eqref{W3q} holds.
Moreover, for each large $\ell$, $\mathfrak{R}_{\be^*}(L^*_\ell)\cap\mathbf{B}_{2\de}$ is a connected Lagrangian graph over $W$ with the graphic function $Dw_\ell$ such that $|w_\ell(0^n)|=|Dw_\ell(0^n)|=0$, $\lim_{\ell\to\infty}|D^2w_\ell(0^n)|=0$, $|D^2w_\ell|\le1$ on $B_\de$ and \eqref{W3q} holds for $w_\ell$ instead of $w$.
Combining Lemma \ref{wnnge0} and Theorem \ref{Harnack} for $w_\ell$ instead of $w$, we conclude that $w_{nn}(x)=0$ for each $(x,Dw(x))\in\Si\cap\overline{\mathbf{B}}_{\de/2}$ by letting $\ell\to\infty$. In other words,
\begin{equation}\aligned
\Si\cap\overline{\mathbf{B}}_{\de/2}=\{(x,y)\in\R^n\times\R^n:\, x\in\overline{B}_{\de/2},\ y=0^n\},
\endaligned
\end{equation}
and
\begin{equation}\aligned
L\cap\overline{\mathbf{B}}_{\de/2}(\mathbf{x}_*)=\{(x,y)\in\R^n\times\R^n:\, x=x_*,\ y\in\overline{B}_{\de/2}(y_*)\}.
\endaligned
\end{equation}
Let $\p(L\setminus(\{x_*\}\times\R^n))$ denote the boundary of $L\setminus(\{x_*\}\times\R^n)$ in $L$.
If $\p(L\setminus(\{x_*\}\times\R^n))\neq\emptyset$, then we consider a point $\mathbf{z}_*=(x_*,z_*)\in\p(L\setminus(\{x_*\}\times\R^n))$, and use the above argument (with Theorem \ref{CaTanPlane}) to get 
\begin{equation}\aligned
L\cap\overline{\mathbf{B}}_{\de_*}(\mathbf{z}_*)=\{(x,y)\in\R^n\times\R^n:\, x=x_*,\ y\in\overline{B}_{\de_*}(z_*)\}
\endaligned
\end{equation}
for some constant $\de_*\in(0,\de/2)$.
This contradicts to the choice of $\mathbf{z}_*$. We complete the proof of Theorem \ref{main}.
\end{proof}

\section{Application to Dirichlet problem of Lagrangian mean curvature equation}

In this section, let $\Om$ be a bounded, uniformly convex domain in $\R^n$.
We consider the Dirichlet problem for the Lagrangian mean curvature equation
\begin{equation}\label{Diri}
\left\{\begin{split}
\mathrm{tr}(\arctan D^2u)=&\th\qquad\qquad\quad\, \mathrm{on}\ \Om\\
u=&\psi \qquad\qquad\quad \mathrm{on}\ \p\Om.
\end{split}\right.
\end{equation}
From Theorem 1.1 in Bhattacharya \cite{Ba3} and Corollary 1.1 in Bhattacharya-Mooney-Shankar \cite{BMS}, for each $C^{1,1}$ function $\th:\overline{\Om}\, \to[\f{n-2}2\pi,\f n2\pi)$ and each $\psi\in C^0(\p\Om)$, the Dirichlet problem \eqref{Diri} admits a unique solution $u\in C^3(\Om)\cap C^0(\overline{\Om})$.
Boundedness of the Lagrangian graph $L=G_{Du}$ requires $u\in C^{0,1}(\overline{\Om})$. From \S 14 in \cite{GT}, boundary gradient estimate usually needs the boundary data $\psi\in C^{1,1}(\p\Om)$ (see also Lu \cite{Lu}).

Using Theorem \ref{main} and results in \cite{Ba3}, we can weaken the $C^{1,1}$ condition of the phase $\th$ to Lipschitz one under $C^{1,1}$ boundary data as follows.
\begin{theorem}\label{Dirith}
Let $n\ge2$,  $\th:\overline{\Om}\, \to[\f{n-2}2\pi,\f n2\pi)$ be a Lipschitz function on $\overline{\Om}$ and $\psi\in C^{1,1}(\p\Om)$, the Dirichlet problem \eqref{Diri}
admits a unique solution $u\in C^{2,\a}(\Om)\cap C^{0,1}(\overline{\Om})$ for each $\a\in(0,1)$.
\end{theorem}
\begin{proof}
There are two infinite sequences of smooth functions $\{\th_\ell\}_{\ell\ge1}$ and $\{\psi_\ell\}_{\ell\ge1}$ so that $\th_\ell:\overline{\Om}\, \to[\f{n-2}2\pi+\de_\ell,\f n2\pi)$ for a constant $\de_\ell>0$ with $\lim_{\ell\to\infty}\de_\ell=0$, $\lim_{\ell\to\infty}\th_\ell=\th$, $\lim_{\ell\to\infty}D(\th_\ell-\th)=0$ a.e., and $\psi_\ell\to\psi$ in $C^{1,1}$-sense on $\p\Om$.
From Theorem 4.1 in \cite{Ba3}, there exists a unique solution $u_\ell\in C^2(\overline{\Om})$ to 
\begin{equation}
\left\{\begin{split}
\mathrm{tr}(\arctan D^2u_\ell)=&\th_\ell\qquad\qquad\quad\, \mathrm{on}\ \Om\\
u_\ell=&\psi_\ell \qquad\qquad\quad \mathrm{on}\ \p\Om.
\end{split}\right.
\end{equation}
Let $\overline{\Th}\in(\f {n-2}2\pi,\f n2\pi)$ and $\underline{\Th}\in(-\f n2\pi,-\f {n-2}2\pi)$ be two constants  so that
$$\underline{\Th}\le\th_\ell\le\overline{\Th}\qquad\mathrm{on}\ \overline{\Om}\ \ \mathrm{for\ each}\ \ell\ge\ell_0,$$
where $\ell_0$ is a positive integer. 
From Theorem 4.1 in \cite{Ba3} again, for each $\ell\ge\ell_0$ there is a solution $\overline{u}_\ell$ (or $\underline{u}_\ell)$ in $C^2(\overline{\Om})$ to \eqref{SLE} with the phase $\overline{\Th}$ (or $\underline{\Th}$) and boundary data $\psi_\ell$ on $\p\Om$.
From comparison principle (see Theorem 3.1 and Theorem 6.1 in \cite{Ba3}), $\overline{u}_\ell\le u_\ell\le \underline{u}_\ell$ on $\overline{\Om}$, and 
$$\sup_{\p\Om}\left(|D\overline{u}_\ell|+|D\underline{u}_\ell|\right)\le c_\Omega,$$
where $c_\Omega$ depends only on $n,\Om$ and $|\psi|_{C^{1,1}(\p\Om)}$ (This was also pointed out in the proof of Lemma 5.1 in \cite{Lu}). 
Since $D\overline{u}_\ell$ and $D\underline{u}_\ell$ both satisfy minimal surface system (see \eqref{uijkThk} for instance), the maximum principle (for functions $\p_j\overline{u}_\ell$ and $\p_j\underline{u}_\ell$) implies
$$\sup_{\overline{\Om}}\left(|D\overline{u}_\ell|+|D\underline{u}_\ell|\right)\le\sup_{\p\Om}\left(|D\overline{u}_\ell|+|D\underline{u}_\ell|\right)\le c_\Omega.$$
Combining $\overline{u}_\ell\le u_\ell\le \underline{u}_\ell$ on $\overline{\Om}$, it follows that 
\begin{equation}\label{C1****}
\sup_{\overline{\Om}}|Du_\ell|\le\sup_{\overline{\Om}}\max\{|D\overline{u}_\ell|,|D\underline{u}_\ell|\}\le c_\Omega.
\end{equation}
From Theorem \ref{main} and the Evans-Krylov-Safonov theory, we get 
\begin{equation}\label{C2****}
|u_\ell|_{C^{2,\a}(\Om')}\le c'_\Omega\qquad \mathrm{for\ each}\ \Om'\subset\subset\Om,\ \a\in(0,1),
\end{equation}
where $c'_\Omega$ depends only on $n,\a,\Om',\Om$, $|\psi|_{C^{1,1}(\p\Om)}$ and the Lipschitz constant $\mathbf{Lip}\,\th$.
Up to choosing the subsequence of $\{\ell\}$, from \eqref{C2****} $u_\ell$ converges to a solution $u\in C^{2,\a}(\Om)$ satisfying $\mathrm{tr}(\arctan D^2u)=\th$. From \eqref{C1****}, we get $\sup_{\overline{\Om}}|Du|\le c_\Omega$ and $u=\psi$ on $\p\Om$.
The uniqueness follows from the maximum  principle for fully nonlinear elliptic equations. This completes the proof.
\end{proof}

As an application of Theorem \ref{Dirith}, we immediately have the following corollary by combining Theorem \ref{main} (with smooth phases approaching the Lipschitz phase), the uniqueness of solutions, and the Evans-Krylov-Safonov theory.
\begin{corollary}\label{main-Cor}
Let $u$ be a  $C^{1,1}$-solution to \eqref{LE*} a.e. on $B_1\subset\R^n$ with the Lipschitz phase $\th\ge\f{n-2}2\pi$ on $B_1$, then for each $\a\in(0,1)$
\begin{equation}\aligned\label{C2a**C11}
\sup_{x\in B_{1/2}}|D^2u(x)|+\sup_{x,x'\in B_{1/2},x\neq x'}|x-x'|^{-\a}|D^2u(x)-D^2u(x')|\le c^*
\endaligned
\end{equation}
where $c^*$ is a constant depending only on $n,\a$, $\sup_{B_1}|Du|$, and the Lipschitz norm of $\th$.
\end{corollary}

\begin{remark}
For Neumann problem to \eqref{LE*} with supercritical phase, refer to the work of Chen-Ma-Wei \cite{CMW}.
\end{remark}

\section{Appendix I}

For a vector $\be=(\be_1,\cdots,\be_n)\in[-\f{\pi}2,\f{\pi}2]^n$, let $\mathfrak{R}_\be:\,\R^n\times\R^n\to\R^n\times\R^n$ be a rotation defined by $\mathfrak{R}_\be(x,y)=(\hat{x}_1,\cdots,\hat{x}_n,\hat{y}_1,\cdots,\hat{y}_n)$ with
\begin{equation}\label{labxby}
\left\{\begin{split}
&\hat{x}_i=(\cos\be_i) x_i+(\sin\be_i) y_i\\
&\hat{y}_j=-(\sin\be_j) x_j+(\cos\be_j) y_j\\
\end{split}\right..
\end{equation}
For an open $\Om\subset\R^n$ and $u\in C^2(\Om)$, let $\Om_\be=\mathfrak{R}_\be(\Om\times\{0^n\})$, 
and $L_\be=\mathfrak{R}_\be(L)$ with $L=G_{Du}$ over $\Om$. 
Let $\mathbf{S}_\be=\mathrm{diag}\{\be_1,\cdots,\be_n\}$ denote a diagonal matrix.
Denote 
$$\sin\mathbf{S}_\be=\mathrm{diag}\{\sin\be_1,\cdots,\sin\be_n\}\qquad \mathrm{and} \qquad\cos\mathbf{S}_\be=\mathrm{diag}\{\cos\be_1,\cdots,\cos\be_n\}.$$
Let $\bar{x},\bar{y}$ be two mappings with components defined by
\begin{equation}
\left\{\begin{split}
&\bar{x}_i(x)=\cos(\be_i) x_i+\sin(\be_i)u_i(x)\\
&\bar{y}_j(x)=-\sin(\be_j) x_j+\cos(\be_j)u_j(x)\\
\end{split}\right..
\end{equation}
Let $J_\be$ be the Jacobi of the mapping $\bar{x}$, i.e.,
\begin{equation}\aligned
J_\be=\left(\f{\p \bar{x}_i}{\p x_j}\right)=\cos\mathbf{S}_\be+\sin\mathbf{S}_\be D^2u,
\endaligned
\end{equation}
and $\bar{J}_\be$ be the Jacobi of the mapping $\bar{y}$, i.e.,
\begin{equation}\aligned
\bar{J}_\be=\left(\f{\p \bar{y}_j}{\p x_i}\right)=-\sin\mathbf{S}_\be+\cos\mathbf{S}_\be D^2u.
\endaligned
\end{equation}
By a direct computation, it's not hard to see
\begin{equation}\aligned\label{JTbebarJbe}
J_\be^T\bar{J}_\be=\bar{J}_\be^TJ_\be.
\endaligned
\end{equation}

We suppose that $J_\be$ is reversible on $\Om$. 
From \eqref{JTbebarJbe}, we have
\begin{equation}\aligned\label{bJJb-1App}
\bar{J}_\be J^{-1}_\be=(\bar{J}_\be J^{-1}_\be)^T.
\endaligned
\end{equation}
Hence, $L_\be$ is a Lagrangian graph over $\Om_\be$ with the graphic function $D\bar{u}_\be$ satisfying
\begin{equation}\aligned\label{Dbarubebxx}
D\bar{u}_\be\Big|_{\bar{x}(x)}=\bar{y}(x)=-x\sin\mathbf{S}_\be+Du\Big|_{x}\cos\mathbf{S}_\be
\endaligned
\end{equation}
for some function $\bar{u}_\be$ on $\Om_\be$. Moreover, we have
\begin{equation}\aligned\label{D2bubebxx}
D^2\bar{u}_\be\Big|_{\bar{x}(x)}=\bar{J}_\be J^{-1}_\be=&\left(-\sin\mathbf{S}_\be+\cos\mathbf{S}_\be D^2u(x)\right)\left(\cos\mathbf{S}_\be+\sin\mathbf{S}_\be D^2u(x)\right)^{-1}\\
=&\left(-\tan\mathbf{S}_\be+D^2u(x)\right)\left(I_n+\tan\mathbf{S}_\be D^2u(x)\right)^{-1}.
\endaligned
\end{equation}
By a direct computation, we find that
\begin{equation}\aligned
\left(D^2u-\tan\mathbf{S}_\be\right)\left(I_n+\tan\mathbf{S}_\be D^2u\right)^{-1}
=\left(I_n+\tan\mathbf{S}_\be D^2u\right)^{-1}\left(D^2u-\tan\mathbf{S}_\be\right).
\endaligned
\end{equation}
Hence, we have
\begin{equation}\aligned\label{D2barubetan}
D^2\bar{u}_\be(\bar{x}(x))=\tan(\arctan D^2u(x)-\mathbf{S}_\be).
\endaligned
\end{equation}
Namely,
\begin{equation}\aligned\label{D2barubebxx}
\arctan\left(D^2\bar{u}_\be(\bar{x}(x))\right)=\arctan D^2u(x)-\mathbf{S}_\be.
\endaligned
\end{equation}

For $\be^*=(\f{\pi}2,\cdots,\f{\pi}2,-\f{\pi}2)$, $\sin\mathbf{S}_{\be^*}=\mathrm{diag}\{1,\cdots,1,-1\}$. 
If $D^2u$ is reversible, then from the first line in \eqref{D2bubebxx} it follows that
\begin{equation}\aligned
D^2\bar{u}_{\be^*}\Big|_{\bar{x}(x)}=-\sin\mathbf{S}_{\be^*} \left(D^2u(x)\right)^{-1}\sin\mathbf{S}_{\be^*}.
\endaligned
\end{equation}
Let $\la_1,\cdots,\la_n$ be eigenvalues of $D^2u$, and $I_*=\sin\mathbf{S}_{\be^*}$. Let $\la$ be an eigenvalue of $D^2\bar{u}_{\be^*}(\bar{x}(x))$ with the corresponding eigenvector $\mathbf{v}$. Then
\begin{equation}\aligned\label{D2uSinSv}
D^2u(x)I_* \mathbf{v}=D^2u(x)I_*\left(-\f1{\la}I_* \left(D^2u(x)\right)^{-1}I_* \mathbf{v}\right)=-\f1{\la}I_* \mathbf{v}.
\endaligned
\end{equation}
This means that $-\f1{\la_1},\cdots,-\f1{\la_n}$ are all $n$ eigenvalues of $D^2\bar{u}_{\be^*}$.

\section{Appendix II}

Given a positive continuous function $\phi$ on $(0,2]$ with $\lim_{t\to0}\phi(t)=\infty$.
There is a smooth positive function $\phi_*\le\phi$ on $(0,2]$ with $\lim_{t\to0}\phi_*(t)=\infty$. Here, $\phi_*$ can be obtained by piecewise smoothing (possibly infinitely many times).
Let  
$$\Phi(t)=\int_0^t\min\{\phi_*(s)/2,|\log(s/2)|\}ds>0\qquad\mathrm{for\ each}\ t\in(0,2].$$
It's clear that $\lim_{t\to0}\Phi(t)=0$.
Let $f(t)$ be a positive function on $(0,2]$ defined by
$$\f1{f(t)}=c+\int_{t}^2\f1{\Phi(s)}ds>0\qquad\mathrm{for\ each}\ t\in(0,2],$$
where $c>0$ is a constant so that $\phi_*\ge 2(n-1)\sup_{[0,2]}f$ on $(0,2]$.
Then 
\begin{equation}\aligned\label{propfApp1}
\lim_{t\to0}\f{1}{f(t)}=c+\int_0^2\f1{\Phi(t)}dt\ge\int_0^2\f1{\int_0^t|\log (s/2)|ds}dt=\infty.
\endaligned
\end{equation}
For a small $0<\ep<1$, we define a function $u_\ep$ on $B_1$ by
\begin{equation}\aligned
u_\ep(x)=\int_0^{|x|}\f1{f_\ep(t)}dt
\endaligned
\end{equation}
with $f_\ep=f(\ep+\cdot)$.
Since $D|x|=x/|x|$, we have
\begin{equation}\aligned
Du_\ep(x)=\f1{f_\ep(|x|)}\f{x}{|x|},
\endaligned
\end{equation}
and
\begin{equation}\aligned
D^2u_\ep(x)=\f1{|x| f_\ep(|x|)}I_n-\left(\f1{|x| f_\ep(|x|)}+\f{f_\ep'(|x|)}{f^2_\ep(|x|)}\right)\f{x^Tx}{|x|^2}.
\endaligned
\end{equation}
So, the eigenvalues of $D^2u_\ep$ are as follows:
\begin{equation}\aligned
\f1{|x| f_\ep(|x|)}\qquad&\mathrm{with\ multiplicity}\ n-1,\\ 
-\f{f_\ep'(|x|)}{f_\ep^2(|x|)}=-\f1{\Phi_\ep(|x|)}\qquad&\mathrm{with\ multiplicity}\ 1,
\endaligned
\end{equation}
where $\Phi_\ep=\Phi(\ep+\cdot)$. Hence, $u_\ep$ is non-convex with the minimal eigenvalue $\underline{\la}_\ep(x)=-1/\Phi_\ep(|x|)$. In particular, $\lim_{\ep\to0}\underline{\la}_\ep(0^n)=-\infty$.

Let $\th_\ep=\mathrm{tr}\left(\arctan D^2u_\ep\right)$. Then
\begin{equation}\aligned\label{thep***}
\th_\ep(x)=&\f{n-1}2\pi-(n-1)\arctan\left(|x| f_\ep(|x|)\right)-\arctan\left(\Phi^{-1}_\ep(x)\right)\\
=&\f{n-2}2\pi+\arctan\left(\Phi_\ep(|x|)\right)-(n-1)\arctan\left(|x| f_\ep(|x|)\right).
\endaligned
\end{equation}
Hence, 
\begin{equation}\aligned\label{thepREP}
D\th_\ep(x)=\f{\Phi_\ep'(|x|)}{1+\Phi_\ep^2(|x|)}\f{x}{|x|}-(n-1)\f{f_\ep(|x|)+|x|f'_\ep(|x|)}{1+|x|^2 f_\ep^2(|x|)}\f{x}{|x|}.
\endaligned
\end{equation}
From $\phi_*\ge 2(n-1)\sup_{[0,2]}f$ on $(0,1]$, for $0<\ep<<1$ it follows that
\begin{equation}\aligned\label{Phiep|x|fep|x|}
\Phi_\ep(|x|)\ge(n-1)\left(\sup_{[0,2]}f\right)\int_0^{|x|+\ep}1ds\ge(n-1)|x| f_\ep(|x|)\qquad \mathrm{for\ each}\ x\in B_1.
\endaligned
\end{equation}
This implies 
$$\f{n-2}2\pi\le\th_\ep(x)<\f{n}2\pi\qquad \mathrm{for\ each}\ x\in B_1.$$
Let $\phi_\ep=\phi_*(\ep+\cdot)$. From 
$$tf'_\ep(t)=\f{tf^2_\ep(t)}{\Phi_\ep(t)},$$
with \eqref{propfApp1}\eqref{Phiep|x|fep|x|} we have 
$$f_\ep(|x|)+|x|f'_\ep(|x|)\le f_\ep(|x|)+\f1{n-1}f_\ep(|x|)\le\f1{2(n-1)}\phi_\ep(|x|)$$
up to a choice of suitably large $c>0$.
Hence, from the definition of $\Phi$ and \eqref{thepREP} we have
\begin{equation}\aligned
|D\th_\ep|(x)\le(n-1)\f{f_\ep(|x|)+|x|f'_\ep(|x|)}{1+|x|^2 f_\ep^2(|x|)}+\f{|\Phi_\ep'|(|x|)}{1+\Phi_\ep^2(|x|)}\le\f{\phi_\ep(|x|)}2+\f{\phi_\ep(|x|)}2=\phi_\ep(|x|).
\endaligned
\end{equation}

\bibliographystyle{amsplain}

\end{document}